%% file: main.tex
\documentclass[11pt]{article}
\usepackage{enumerate}
\usepackage[OT1]{fontenc}
\usepackage[usenames]{color}
\usepackage{smile}
\usepackage[colorlinks,
            linkcolor=red,
            anchorcolor=blue,
            citecolor=blue
            ]{hyperref}
\usepackage{mathrsfs}
\usepackage{fullpage}
\usepackage{hyperref}
\usepackage[protrusion=true, expansion=true]{microtype}
\usepackage{float}
\usepackage{subfigure}
\usepackage{amsfonts,amsmath,amssymb,amsthm,url,xspace}
\usepackage{tikz}
\usepackage{verbatim}
\usetikzlibrary{arrows,shapes}
\usepackage{mathtools}
\usepackage{authblk}
\usepackage[bottom]{footmisc}
\usepackage{enumitem}
\usepackage{mathtools}
\usepackage{lscape}
\usepackage{caption, adjustbox}

\usepackage{algorithm}
\usepackage{algorithmicx,algpseudocode}
\usepackage[title]{appendix}
\ifx\counterwithout\undefined\usepackage{chngcntr}\fi
\counterwithout{equation}{section}

\usepackage{math}
\allowdisplaybreaks[1]
\mathtoolsset{showonlyrefs=true}

\usepackage{xargs}
\usepackage[colorinlistoftodos,prependcaption,textsize=tiny]{todonotes}
\newcommandx{\unsure}[2][1=]{\todo[linecolor=red,backgroundcolor=red!25,bordercolor=red,#1]{#2}}
\newcommandx{\change}[2][1=]{\todo[linecolor=blue,backgroundcolor=blue!25,bordercolor=blue,#1]{#2}}
\newcommandx{\info}[2][1=]{\todo[linecolor=OliveGreen,backgroundcolor=OliveGreen!25,bordercolor=OliveGreen,#1]{#2}}
\newcommandx{\improvement}[2][1=]{\todo[linecolor=Plum,backgroundcolor=Plum!25,bordercolor=Plum,#1]{#2}}

\usepackage{alphalph}

\allowdisplaybreaks

\algrenewcommand{\algorithmiccomment}[1]{\hfill$\blacktriangleright$ #1}
\algnewcommand{\LongComment}[1]{\hfill$\triangleright$ #1}

\begin{document}

\title{Derivative-Free Sequential Quadratic Programming for\\ Equality-Constrained Stochastic Optimization}

\author{Sen Na}
\affil{School of Industrial and Systems Engineering, Georgia Institute of Technology}

\date{}

\maketitle

\begin{abstract}

We consider solving nonlinear optimization problems with a stochastic objective and deterministic equality constraints, assuming that only zero-order information is available for both the objective and constraints, and that the objective is also subject to random sampling noise. Under this setting, we propose a Derivative-Free Stochastic Sequential Quadratic Programming (DF-SSQP) method, which employs an $\ell_2$ merit function to adaptively select the stepsize.
Due to the lack of derivative information, we adopt a \textit{simultaneous perturbation stochastic approximation} (SPSA) technique to randomly estimate the gradients and Hessians of both the objective and constraints. This approach requires only a dimension-independent number of zero-order evaluations -- as few as eight -- at each~iteration step.~A key distinction between our derivative-free method and existing derivative-based line-search or trust-region SSQP methods lies in the intricate random bias introduced into the gradient~and~Hessian estimates of the objective and constraints, brought about by stochastic zero-order approximations. To address this issue, we introduce an online debiasing technique based on momentum-style estimators that properly aggregate past gradient and Hessian estimates to reduce stochastic noise, while avoiding excessive memory costs via a moving averaging scheme. 
Under standard assumptions, we establish the global almost-sure convergence of the proposed DF-SSQP method. Notably,~we~\mbox{further}~\mbox{complement}~the~global analysis with local convergence guarantees by demonstrating that the rescaled iterates exhibit asymptotic normality, with a limiting covariance matrix resembling the minimax optimal covariance achieved by derivative-based methods, albeit larger due to the \mbox{absence}~of derivative information. Our local~analysis enables online statistical inference of model parameters leveraging DF-SSQP. Numerical experiments on benchmark nonlinear problems demonstrate both the global and local behavior of DF-SSQP.

\end{abstract}

\section{Introduction}\label{sec:1}

We consider solving nonlinear equality-constrained stochastic optimization problems:
\begin{equation}\label{Intro_StoProb}  
\min_{\bx \in \mR^d}\; f(\bx) = \mE_\P[F(\bx; \xi)], \quad \text{s.t.} \quad c(\bx) = \b0,
\end{equation}
where $f: \mR^d \rightarrow \mR$ denotes the stochastic objective function, $F(\cdot; \xi): \mR^d \rightarrow \mR$ denotes its realization with sample $\xi\sim \P$, and $c:\mR^d \rightarrow \mR^m$ denotes the deterministic equality constraints. Problem \eqref{Intro_StoProb} appears widely in a variety of applications in statistical machine learning and operations research, including constrained maximum likelihood estimation \citep{Dupacova1988Asymptotic}, multi-stage stochastic optimization \citep{Veliz2014Stochastic}, reinforcement learning \citep{Achiam2017Constrained}, portfolio management \citep{Cakmak2005Portfolio}, and network optimization \citep{Shakkottai2007Network}.

There exist numerous methods for solving constrained optimization problems, including projection-based methods, penalty methods, augmented Lagrangian methods, and sequential quadratic programming (SQP) methods. Among these, SQP is arguably one of the most effective methods for both small- and large-scale problems \citep{Nocedal2006Numerical}. It avoids the need of projection steps, which can be intractable for general constraints, and is robust to initialization, less affected by ill-conditioning issues, and flexible in incorporating advanced computational techniques, such as line search, trust region, and quasi-Newton updates.

In recent years, designing stochastic SQP (SSQP)-based methods for solving constrained stochastic optimization problems has attracted growing interest. \cite{Berahas2021Sequential} introduced the first~SSQP method for equality-constrained stochastic problems, which employs an $\ell_1$-penalized merit function and an adaptive mechanism for selecting both the penalty parameter and the stepsize, aiming to enforce a sufficient reduction on the $\ell_1$ merit function. The authors also established the ``liminf" convergence for the expectation of the KKT residual. Following \cite{Berahas2021Sequential}, several algorithmic and theoretical advancements have emerged. On the algorithmic side,
\cite{Berahas2023Stochastic} introduced the~step~decomposition in SSQP to address rank-deficient constraint Jacobians; \cite{Curtis2024Stochastic} incorporated an inexact quadratic program solver to improve computational efficiency; \cite{Berahas2023Accelerating} accelerated SSQP by leveraging variance reduction techniques; \cite{Curtis2023Stochastic, Curtis2024Sequential} extended SSQP to include deterministic box constraints; \cite{Fang2024Fully} further complemented these methods by designing a trust-region SSQP scheme, where the search direction and \mbox{stepsize}~(i.e.,~the~\mbox{trust-region}~radius) are computed jointly; and \cite{Shen2025Sequential} generalized the design of SSQP to expectation equality-constrained problems. On the theoretical side, \cite{Curtis2023Worst} and \cite{Na2025Statistical} analyzed the worst-case iteration and sample complexity of SSQP, considering constant and decaying stepsizes, respectively;
\cite{Lu2024Variance} established similar complexity results for stochastic penalty methods with variance reduction;
\cite{Curtis2025Almost} investigated the convergence 
behavior of the Lagrange multiplier; and \cite{Berahas2025Sequential, Fang2025High} addressed the high-probability first- and second-order iteration complexities under probabilistic oracles.

In addition to the above literature, recent studies have also observed that adaptively increasing the batch size in SSQP can significantly enhance performance. For example, \cite{Na2022adaptive} proposed the first SSQP method under this setup, where the derivatives of an augmented Lagrangian merit function, as well as the stepsize from stochastic line search, are computed with the batch size adaptively~determined based on probabilistic error bounds. Subsequently, \cite{Na2023Inequality} employed active-set strategy to accommodate nonlinear inequality constraints; \cite{Qiu2023sequential} developed a robust SSQP scheme; \cite{Berahas2022Adaptive} incorporated a norm test condition into SSQP, originally proposed for~SGD \citep{Bollapragada2018Adaptive}; \cite{Fang2024Trust} extended SSQP studies to establish second-order convergence guarantees using trust-region techniques; and \cite{Berahas2025Retrospective} designed a retrospective approximation SSQP scheme to achieve optimal gradient evaluation complexity.
Moreover, constrained stochastic problems are also related to the broader context of noisy optimization. We refer to \cite{Sun2023trust, Lou2024Noise, Oztoprak2023Constrained, Sun2024Trust, Berahas2025line, Berahas2025Optimistic, Curtis2025interior} for such studies. However, we 
mention that those methods are designed to be robust to (deterministic) adversarial noise, which is significantly \mbox{different}~from~methods designed for stochastic settings.

Although the aforementioned literature provides versatile computational methodologies for solving Problem \eqref{Intro_StoProb}, showing promising global convergence guarantees and iteration/sample complexities under favorable assumptions, the existing methods are all derivative-based.~This means that they require the evaluation of the gradient (actually, in many cases, the Hessian as well) of the objective and constraints. Such a requirement is restrictive for many applications where gradients are either unavailable or too expensive to compute. 
For example, in hyperparameter optimization, the goal is to tune parameters in neural networks or machine learning models to achieve the best output. While the output may be smooth with respect to some tuning parameters, computing higher-order information beyond zero-order is often infeasible due to the inherently black-box nature of the problem. Similarly, in PDE-constrained optimization, the objective function depends on the solution of the PDE. Gradients of the objective are typically computed using adjoint methods, which involve solving an additional adjoint PDE that~has comparable computational costs  as solving the original (state) PDE, effectively doubling the cost per iteration.
This significant computational burden associated with gradient evaluations motivates the desire of a \textbf{Derivative-Free SSQP} method (DF-SSQP) in the present paper.

Throughout the paper, we assume that only zero-order information is available for both the objective and constraints, and the objective evaluation is accessible only through realizations $F(\cdot;\xi)$. This setup situates our work within the broad framework of derivative-free optimization (DFO). DFO methods do not require the accessibility of derivatives, making them widely applicable to complex and even black-box problems. Representative DFO methods include finite-difference methods, model-based methods, coordinate search and pattern-search methods, and Nelder-Mead methods, among others. As the first trial, this paper leverages (randomized) finite-difference approximations to estimate the derivatives,~a~technique that has a long history in optimization and statistics, dating back to \cite{Kiefer1952Stochastic}. In particular, in the univariate $(d=1)$ and unconstrained case, \cite{Kiefer1952Stochastic} approximated the objective gradient by drawing a sample $\xi_k\sim\P$ and computing
\begin{equation*}
\hat{\nabla}F(\bx_{k};\xi_k)=\frac{F(\bx_{k}+b_k;\xi_k)-F(\bx_{k};\xi_k)}{b_k},
\end{equation*}
where $b_k>0$ is a deterministic sequence going to zero as $k\rightarrow\infty$. With $\hat{\nabla}F(\bx_{k};\xi_k)$, we then perform stochastic gradient descent update as $\bx_{k+1}=\bx_k - \alpha_k\hat{\nabla}F(\bx_{k};\xi_k)$. \cite{Blum1954Multidimensional} later extended this KW method to the multivariate case and established its almost sure convergence. These pioneering works have since been extended from various perspectives under different setups. 
To reduce the number of zero-order evaluations at each step, several randomized approximation methods have been proposed. 
\cite{Koronacki1975Random} employed a sequence of random unit vectors that are independent and uniformly distributed on the unit sphere and provided sufficient conditions for the convergence of the method. Later, \cite{Spall1992Multivariate, Spall2000Adaptive, Chen1999Kiefer} refined this approach to generic random directions, referring to the new method as \textit{Simultaneous Perturbation Stochastic Approximation} (SPSA). Numerous studies have shown that randomized approximations like SPSA significantly reduce the required number of observations or measurements. For a $d$-dimensional problem, the number of function evaluations required by the SPSA method is only $1/d$ of those required by the deterministic approximation, making it \textit{dimension-independent}. We refer to \cite{Spall2003Introduction, Kushner2012Stochastic, Bhatnagar2013Stochastic} for literature review of the SPSA and to \cite{Chen1988Lower, Hall2003Sequential, Dippon2003Accelerated, Mokkadem2007companion, Broadie2011General, Rasonyi2022Convergence, Chen2024Online, DuYi2024Derivative} for more KW-type algorithms and their empirical investigations. See also \cite{Conn2009Introduction, Larson2019Derivative, Custodio2017Chapter} for broad review of derivative-free~methods.

In this paper, we leverage the SPSA technique to randomly estimate the gradients (as well as Hessians if local convergence is an interest) of the objective and constraints of Problem \eqref{Intro_StoProb}. Specifically, at each iteration $\bx_k$, we generate a sample $\xi_k\sim\P$ and  a random direction $\bDelta_k\in\mR^d$, and approximate the objective gradient $\nabla F(\bx_k;\xi_k)\in\mR^d$ and the constraint Jacobian $\nabla c(\bx_k)\in\mR^{m\times d}$ as (the Hessian approximation is introduced in Section \ref{sec:2-1})
\begin{equation}\label{snequ:2}
\begin{aligned}
\hat{\nabla}F(\bx_k;\xi_k) & = \frac{F(\bx_k + b_k \bDelta_k; \xi_k) - F(\bx_k - b_k \bDelta_k; \xi_k)}{2b_k}\bDelta_k^{-1},\\
\hnabla c(\bx_k) & = \frac{c(\bx_k + b_k \bDelta_k)-c(\bx_k - b_k \bDelta_k)}{2b_k}\bDelta_k^{-T},
\end{aligned}
\end{equation}
where $b_k>0$ is still a deterministic sequence going to zero as $k\rightarrow\infty$, and $\bDelta_k^{-1}\coloneqq (\frac{1}{\bDelta_{k}^1},\ldots,\frac{1}{\bDelta_{k}^d})\in\mR^d$ is entrywise reciprocal of $\bDelta_k=(\bDelta_k^1,\ldots,\bDelta_k^d)$.

Applying the SPSA technique to SSQP introduces a key challenge: all gradient and Hessian estimates of the objective and constraints are subject to intricate random bias brought by both random direction $\bDelta_k$ and finite-difference approximation. In contrast, existing derivative-based line-search or trust-region SSQP methods all rely on \textit{unbiased} gradient and Hessian estimates. This bias not only poses fundamental difficulties in the analysis but also impairs the convergence of the method. As shown~even~for~unconstrained problems in \cite{Berahas2019Derivative, Sun2023trust}, methods with biased \mbox{derivative}~estimates converge only to a region near the optimal solution, whose radius expands as the bias level increases, ultimately leading to deterioration of the method.
To address this challenge, we propose an online debiasing technique based on momentum-style estimators, which properly aggregate all past gradient and Hessian estimates to eliminate noise, while avoiding excessive memory costs via the moving average scheme. Under reasonable assumptions, we demonstrate that the KKT residual of the iteration sequence $\bx_k $, along with the least-squares estimates of the dual variables, converges to zero almost surely from any initialization. More significantly, we complement the global analysis, primarily focused in the majority of existing SSQP literature, with new local convergence guarantees by showing~that~the~rescaled iterates exhibit asymptotic normality:
\begin{equation}\label{snequ:1}
1/\sqrt{\baralpha_k}\cdot(\boldsymbol{x}_k - \tx, \boldsymbol{\lambda}_k - \tlambda) \stackrel{d}{\longrightarrow} \mathcal{N}\left(\boldsymbol{0}, \bSigma^\star\right),
\end{equation}
where $\baralpha_k$ is the adaptive random stepsize and the limiting covariance matrix $\tSigma$ is given by a sandwich form (see Section \ref{sec:4} for details):
\begin{equation}\label{equ:SigmaStar}
\bSigma^\star \coloneqq (\nabla^2\mL(\tx,\tlambda))^{-1}\diag\rbr{\mE\sbr{\bDelta^{-1}\bDelta^T\text{Cov}(\nabla F(\tx;\xi))\bDelta\bDelta^{-T}}, \0} (\nabla^2\mL(\tx,\tlambda))^{-1}.
\end{equation}
Here, $\mL(\bx,\blambda) = f(\bx) + c^T(\bx)\blambda$ denotes the Lagrangian function, and the expectation is taken over~the randomness in $\bDelta$. We show that the covariance $\bSigma^\star$ in \eqref{equ:SigmaStar} closely resembles the \textit{minimax optimal covariance} achieved by derivative-based methods \citep{Duchi2021Asymptotic, Davis2024Asymptotic, Na2025Statistical, Du2025Online}:
\begin{equation}\label{equ:SigmaOpt}
\bSigma^\star_{op} = (\nabla^2\mL(\tx,\tlambda))^{-1}\diag\rbr{\text{Cov}(\nabla F(\tx;\xi)), \0} (\nabla^2\mL(\tx,\tlambda))^{-1}.
\end{equation}
However, $\bSigma^\star\succeq \bSigma^\star_{op}$ due to the absence of gradient computations. Furthermore, we show that
\begin{equation}\label{equ:SigmaDiff}
\|\tSigma - \bSigma^\star_{op}\| \asymp O(d),
\end{equation}
where $\asymp$ denotes the precise order in the sense that $d/C\leq \|\tSigma - \bSigma^\star_{op}\|\leq Cd$ for some constant $C$.

We would like to further elucidate our local convergence results \eqref{snequ:1}--\eqref{equ:SigmaDiff}, which concern the statistical efficiency of DF-SSQP. Existing derivative-based SSQP methods primarily focused on global convergence guarantees (or non-asymptotic convergence guarantees), with two notable exceptions in \cite{Na2025Statistical} and \cite{Du2025Online} that showed both SSQP and its averaged version can achieve optimal statistical efficiency \eqref{equ:SigmaOpt}, matching that of projection-based methods in \cite{Duchi2021Asymptotic, Davis2024Asymptotic, Jiang2025Online} for solving Problem \eqref{Intro_StoProb}.
This paper further extends this line of research, showing that the limiting covariance $\tSigma$ of DF-SSQP reflects a \textit{trade-off between statistical and computational efficiency}.~Derivative-based SSQP prioritizes statistical efficiency at the expense of computational efficiency, while DF-SSQP emphasizes computational efficiency but inevitably sacrifices certain statistical efficiency. In particular, DF-SSQP only computes dimension-independent number of function evaluations to approximate derivatives, while its statistical efficiency gap to the optimum (i.e., $\|\tSigma - \bSigma^\star_{op}\|$) sharply grows linearly with the dimension $d$. 
Compared to global analysis, our local analysis requires quantifying all sources of uncertainty in the method, including randomness in sampling (i.e., $\xi_k$), computation (i.e., $\bDelta_k$), and adaptivity (i.e., $\baralpha_k$).
Overall, our local results enable~online statistical inference for the solution $(\tx, \tlambda)$ based on the iterates $(\bx_k,\blambda_k)$ generated by DF-SSQP, which is of broad interest in statistics and machine learning applications. We demonstrate~the global~and local behavior of DF-SSQP through extensive numerical experiments on benchmark nonlinear problems.$\quad$

\subsection{Notation}

We use $\|\cdot\|$ to denote the $\ell_2$-norm for vectors and the operator norm for matrices. We let $I$ denote the identity matrix and $\b0$ denote the zero vector or matrix. Their dimensions are clear from the context. For the constraint $c: \mathbb{R}^d \to \mathbb{R}^m$, we define $ G(\bx) \coloneqq \nabla  c(\bx) \in \mathbb{R}^{m \times d}$ as its Jacobian matrix. For~$1\leq j\leq m$, we use the superscript $c^j(\bx)$ to denote the $j$-th component of $c(\bx)$; and for any iteration index $k$, we let $c_k= c(\bx_k)$ and $G_k= G(\boldsymbol{x}_k)=\nabla  c(\boldsymbol{x}_k)$ (similarly, $\nabla\mathcal{L}_k=\nabla\mathcal{L}(\boldsymbol{x}_k,\boldsymbol{\lambda}_k)$, etc.). We also use $O(\cdot)$ to denote the big-$O$ notation in the usual sense; that is, $a_k = O(b_k)$ if $|a_k|/|b_k|$ is bounded. Additionally, $O_p(\cdot)$ and $o_p(\cdot)$ denote big- and little-$O$ notation in probability sense, respectively.

\subsection{Structure of the paper}

In Section \ref{sec:2}, we introduce the design of our DF-SSQP method. The global convergence guarantee is presented in Section \ref{sec:3}, followed by the local convergence guarantee in Section \ref{sec:4}. Numerical~experiments are presented in Section \ref{sec:5}, and the conclusions are summarized in Section~\ref{sec:6}.~Additional~theoretical results and all proofs are provided in the appendix.

\section{Derivative-Free Stochastic Sequential Quadratic Programming}\label{sec:2}

In this section, we propose the DF-SSQP method, which is summarized in Algorithm \ref{Alg:DF-SSQP}. In Section \ref{sec:2-1}, we introduce the gradient and Hessian estimates of the objective and constraints using a randomized~finite-difference approximation, along with our debiasing, momentum-style step. Then, in \mbox{Section}~\ref{sec:2-2}, we~provide a detailed explanation of each step of DF-SSQP.

\subsection{Debiased derivatives via averaging}\label{sec:2-1}

Given the $k$-th iterate $\bx_k$, we draw a sample $\xi_{k}\sim \mathcal{P}$ and two independent random directions $\bDelta_{k},\tilde{\bDelta}_{k}\in\mR^d$. Let $\mathcal{P}_{\bDelta}$ denote the distribution of the random directions. Throughout the paper, we assume that $\bDelta\sim \mathcal{P}_{\bDelta}$ has mutually independent components, each symmetrically distributed about zero with absolute values bounded both from above and below (cf. Assumption \ref{ass:Delta}).

\vskip0.2cm

\noindent$\bullet$  \textbf{Gradient Estimate.} 
Let $\{b_{k}\}$ and $\{\beta_k\}$ be predefined positive sequences. As introduced in Section \ref{sec:1}, we approximate the objective gradient $\nabla F(\bx_k;\xi_k)\in\mR^d$ and constraint Jacobian $G_k = \nabla c(\bx_k)\in\mR^{m\times d}$ by $\hnabla F(\bx_k;\xi_k)$ and $\hnabla c(\bx_k)$, as defined in \eqref{snequ:2}. Unlike existing derivative-based SSQP methods, we further perform a debiasing step by (online) averaging the past estimates as
\begin{equation}\label{snequ:3}
\bar{\boldsymbol g}_k  =(1-\beta_k) \bar{\boldsymbol g}_{k-1}+\beta_k\hat{\nabla}F(\bx_k;\xi_k) \quad\quad \text{and} \quad\quad \barG_k = (1-\beta_k) \barG_{k-1} + \beta_k \hnabla c(\bx_k).
\end{equation}
This moving averaging technique is essential to our method. In Lemma \ref{lemma:average almost sure}, we will show the almost sure convergence of $\barg_k$ to $\nabla f_k$ and $\barG_k$ to $G_k$. In contrast, simple approximations $\hat{\nabla} F(\bx_k; \xi_k)$ and $\hat{\nabla} c(\bx_k)$ cannot be sufficiently close to their exact counterparts $\nabla f_k$ and $G_k$.

\vskip0.2cm

\noindent$\bullet$ \textbf{Hessian Estimate.}
The Hessian estimate is only necessary when local convergence property is of interest (cf. Section \ref{sec:4}). To estimate the objective and constraint Hessians, we let $\{\tilde{b}_{k}\}$ be another~predefined positive sequence. We first compute the gradient estimates:
\begin{equation}\label{snequ:12}
\begin{aligned}
\tilde{\nabla} F(\bx_k \pm b_k \bDelta_k;\xi_k) & = \frac{F(\boldsymbol x_{k}\pm b_{k} \boldsymbol \Delta_{k}+\tilde{b}_{k} \tilde{\boldsymbol \Delta}_{k}; \xi_{k}) - F(\boldsymbol x_{k}\pm b_{k} \boldsymbol \Delta_{k}; \xi_{k})}{\tilde{b}_k}\tbDelta_k^{-1} \in\mR^d,\\
\tilde{\nabla} c(\bx_k \pm b_k \bDelta_k) & = \frac{c(\bx_k \pm b_k \bDelta_k + \tilde{b}_{k} \tilde{\boldsymbol \Delta}_{k}) - c(\bx_k \pm b_k \bDelta_k) }{\tb_k}\tbDelta_k^{-T} \in\mR^{m\times d}.
\end{aligned}
\end{equation}
Here, we use $\tnabla$ to distinguish it from $\hnabla$, where $\tnabla$ employs a one-sided finite-difference approximation. 
This reduces the number of function evaluations as $F(\boldsymbol x_{k}\pm b_{k} \boldsymbol \Delta_{k}; \xi_{k})$ and $c(\bx_k \pm b_k \bDelta_k)$ are already computed from the gradient estimation. With the above estimates, we then estimate the Hessians as
\begin{equation}\label{snequ:4}
\begin{aligned}
\hnabla^2F(\bx_k;\xi_k) & = \frac{1}{2}\sbr{\frac{\delta \tilde{\nabla} F(\bx_k \pm b_k \bDelta_k;\xi_k)}{2b_k}\bDelta_k^{-T} + \bDelta_k^{-1}\frac{\{\delta \tilde{\nabla} F(\bx_k \pm b_k \bDelta_k;\xi_k)\}^T}{2b_k}}, \\
\hnabla^2c^j(\bx_k) & = \frac{1}{2}\sbr{\frac{\delta\tnabla c^j(\bx_k \pm b_k \bDelta_k)}{2b_k}\bDelta_k^{-T} + \bDelta_k^{-1} \frac{\{\delta\tnabla c^j(\bx_k \pm b_k \bDelta_k)\}^T}{2b_k}}, \quad \text{for }\; 1\leq j\leq m,
\end{aligned}
\end{equation}
where
\begin{equation}\label{snequ:11}
\begin{aligned}
\delta \tilde{\nabla} F(\bx_k \pm b_k \bDelta_k;\xi_k) & = \tilde{\nabla} F(\bx_k + b_k \bDelta_k;\xi_k) - \tilde{\nabla} F(\bx_k - b_k \bDelta_k;\xi_k) \in\mR^d,\\
\delta \tilde{\nabla} c^j(\bx_k \pm b_k \bDelta_k) & = \tilde{\nabla} c^j(\bx_k + b_k \bDelta_k) - \tnabla c^j(\bx_k - b_k \bDelta_k) \in\mR^d,
\end{aligned}
\end{equation}
and $\tilde{\nabla} c^j$ is the transpose of the $j$-th row of $\tilde{\nabla} c$. 
Since the Hessians are not crucial for the convergence of the algorithm, and the debiasing step can perform either weighted averaging as in \eqref{snequ:3} or uniform averaging (i.e., equal weights) as in \cite{Na2022Hessian}, and will actually focus on the Lagrangian Hessian, we defer its introduction to the algorithm description in Section \ref{sec:2-2}. (The gradient averaging weight~$\beta_k$ plays a crucial role while, in contrast, the Hessian averaging weight can be arbitrary.)

\subsection{Algorithm design}\label{sec:2-2}

Let us define $\mL(\boldsymbol{x}, \boldsymbol{\lambda}) = f(\boldsymbol{x}) + \boldsymbol{\lambda}^T c(\boldsymbol{x})$ as the Lagrangian function of \eqref{Intro_StoProb}, where $\boldsymbol{\lambda} \in \mathbb{R}^m$ denotes the dual vector. Under certain constraint qualifications, a necessary condition for $(\boldsymbol x^\star,\boldsymbol \lambda^\star)$ being a local~solution to \eqref{Intro_StoProb} is the KKT conditions:
\begin{equation}\label{equ:KKT}
\nabla\mL(\boldsymbol x^{\star},\boldsymbol{\lambda}^{\star})=\begin{pmatrix}\nabla_{\boldsymbol x}\mathcal{L}(\boldsymbol x^{\star},\boldsymbol{\lambda}^{\star})\\\nabla_{\boldsymbol\lambda}\mathcal{L}(\boldsymbol x^{\star},\boldsymbol{\lambda}^{\star})\end{pmatrix}=\begin{pmatrix}\nabla f(\boldsymbol{x}^{\star})+G(\boldsymbol{x}^{\star})^T\boldsymbol{\lambda}^{\star}\\c(\boldsymbol{x}^{\star})\end{pmatrix}=\begin{pmatrix}\boldsymbol{0}\\\boldsymbol{0}\end{pmatrix}.
\end{equation}
Our method can be regarded as an application of Newton's method to the equation $\nabla\mL(\bx,\blambda) = \b0$, involving three steps: gradient and Hessian estimation, computation of the Newton direction, and update of the primal-dual iterates with a properly selected stepsize. The method requires prespecified positive sequences $\{b_{k}, \tilde{b}_{k}, \alpha_k, \beta_k\}$ and four parameters $\sigma, \varepsilon \in (0,1)$, $\psi\geq0$, $p\geq 1$. The method~is~initialized at $(\bx_0, \blambda_0)\in\mR^d\times\mR^m$, $\barg_{-1}\in\mR^d$, $\barG_{-1}\in\mR^{m\times d}$, $\barB_{-1} = I\in\mR^{d\times d}$, and $\tau_{-1}, \nu_{-1}>0$.

Given $(\bx_k,\boldsymbol \lambda_k)$ at the $k$-th iteration, we first obtain the gradient and Jacobian estimators $\barg_k$ and $\barG_k$ as in \eqref{snequ:3}. To exhibit promising local properties, we also compute the Hessian estimators $\hnabla^2F(\bx_k;\xi_k)$~and $\{\hnabla^2c^j(\bx_k)\}_{j=1}^m$ as in \eqref{snequ:4}. Then, we need to regularize the Jacobian $\barG_k$ as
\begin{equation}\label{snequ:9}
\tG_k=\barG_k+ \delta_{k}^G,
\end{equation}
where $\delta_k^G\in\mR^{m\times d}$ is a perturbation/regularization matrix such that $\tG_k$ has full row rank. After~obtaining this $\tG_k$, we then compute the following three quantities:
\begin{equation}\label{snequ:5}
\bar{\nabla}_{\boldsymbol x} \L_{k} = \bar{\boldsymbol g}_k+\tG_k^T\boldsymbol\lambda_k, \quad \hnabla^2_{\boldsymbol x} \L_{k}=\hnabla^2F(\bx_k;\xi_k)+\sum_{j=1}^{m}\boldsymbol{\lambda}_{k}^{j} \hnabla^2 c^j(\bx_k), \quad \barB_{k} = (1-\beta_k)\barB_{k-1} + \beta_k\hnabla^2_{\boldsymbol x} \L_{k}.
\end{equation}
Here, $\bar{\nabla}_{\bx} \L_{k}$ and $\barB_k$ denote the (debiased) estimates of the Lagrangian gradient and Hessian with respect to $\bx$. We emphasize that (i) we can simply set $\barB_k = I$ for the purpose of global convergence; and (ii) the Hessian averaging weight is not as crucial as that of the gradient averaging. For simplicity, we~use the same weight $\beta_k$, although uniform averaging with $\beta_k = 1/k$ also works.

To ensure that the Newton system is well-defined, we also have to regularize the Hessian $\barB_k$ as:$\;\;$
\begin{equation}\label{snequ:6}
\tB_k = \barB_k + \delta_k^B,
\end{equation}
where $\delta_k^B\in\mR^{d\times d}$ is a perturbation/regularization matrix such that $\tB_k$ is positive definite in the null space $\text{ker}(\tG_k)$. With the above derivative approximations, we then solve the following Newton system:
\begin{equation}\label{def:acc2}
\underbrace{
\begin{pmatrix}
\tB_k & \tG_k^T\\
\tG_k & \0
\end{pmatrix}}_{\tW_k}\underbrace{\begin{pmatrix}
\tDelta \bx_k\\
\tDelta \blambda_k
\end{pmatrix}}_{\tDelta\bz_k} = - \underbrace{\begin{pmatrix}
\bnabla_{\bx}\mL_k\\
c_k
\end{pmatrix}}_{\bnabla \mL_k},	
\end{equation}
where $\tW_k$ and $\bnabla\mathcal{L}_{k}$ represent the Lagrangian Hessian and gradient, and $\tDelta\bz_{k}$ is the (exact) Newton direction. We mention that the regularizations in \eqref{snequ:9} and \eqref{snequ:6} are intended to ensure that $\tW_k$ is invertible and the system \eqref{def:acc2} is well-defined \cite[Lemma 16.1]{Nocedal2006Numerical}.

After obtaining the Newton direction $\tDelta\bz_k = (\tDelta \bx_k, \tDelta \blambda_k)$, we update the primal-dual iterate with a properly selected stepsize $\baralpha_k$ as:
\begin{equation*}
(\bx_{k+1}, \blambda_{k+1}) = (\bx_k, \blambda_k) + \baralpha_k (\tDelta \bx_k, \tDelta \blambda_k). 
\end{equation*}
Similar to \cite{Berahas2021Sequential, Berahas2023Stochastic, Berahas2023Accelerating} and many references therein, the stepsize $\baralpha_k$ is selected to achieve a sufficient reduction on an $\ell_2$ merit function:
\begin{equation*}
\phi_{\tau}(\bx)=\tau f(\boldsymbol x)+\|c(\boldsymbol x)\|.
\end{equation*}
In particular, given $\tau>0$, we define its local model at $\bx_k$ along the direction $\bd\in\mR^d$ as
\begin{equation*}
q(\bd; \tau, \bx_k, \barg_k,\tB_k)=\tau\left(f_k + \barg_k^T\bd+\frac{1}{2}\max\{\boldsymbol d^T\tB_k\boldsymbol d,0\}\right)+\|c_k+\tG_k \bd\|.
\end{equation*}
When $\bd$ satisfies $c_k+\tG_k \bd = \0$ as in \eqref{def:acc2}, the reduction of the local model is given by
\begin{equation}\label{snequ:10}
\Delta q(\bd; \tau, \bx_k, \barg_k,\tB_k) \coloneqq q(\0; \tau, \bx_k, \barg_k,\tB_k) - q(\bd; \tau, \bx_k, \barg_k,\tB_k) = -\tau (\barg_k^T\bd+0.5\max\{\boldsymbol d^T\tB_k\boldsymbol d,0\}) + \|c_k\|.
\end{equation}
The above formula motivates us to define
\begin{equation} \label{def:acc3}
\tau_k^{\text{trial}}\leftarrow\begin{cases}
\infty & \text{if }\; \bar{\boldsymbol g}_k^T\tilde{\Delta} \boldsymbol{x}_k+ \max\{\tilde{\Delta} \boldsymbol{x}_k^T\tB_k\tilde{\Delta} \boldsymbol{x}_k, 0\}\leq0,\\
\frac{(1-\sigma)\|c_k\|}{\barg_k^T\tilde{\Delta} \boldsymbol{x}_k+ \max\{\tilde{\Delta} \boldsymbol{x}_k^T\tB_k\tilde{\Delta} \boldsymbol{x}_k, 0\}} & \text{otherwise},
\end{cases}
\end{equation}
followed by the rule of updating $\tau_k$ from $\tau_{k-1}$ as
\begin{equation}\label{def:acc4}
\tau_k\leftarrow\begin{cases}{\tau}_{k-1}&\text{if }\; {\tau}_{k-1}\leq\tau_k^{\text{trial}}, \\
(1-\epsilon)\tau_k^{\text{trial}}&\text{otherwise}.\end{cases}
\end{equation}
Since the above merit parameter rule ensures $\tau_k\leq\tau_k^{\text{trial}}$, it follows that
\begin{equation}\label{reduction}
\Delta q(\tDelta\bx_k; \tau_k, \boldsymbol x_k,\bar{\boldsymbol g}_k,\tB_k)\geq\frac{1}{2}\tau_k  \max\{\tilde{\Delta} \boldsymbol{x}_k^T\tB_k\tilde{\Delta} \boldsymbol{x}_k, 0\}+\sigma\|c_k\|.
\end{equation}

Next, we define the updating rule for a ratio parameter $\nu_k$, which builds a connection between the reduction of the local model $q(\tDelta\bx_k; \tau_k, \bx_k, \bar{\boldsymbol g}_k,\tB_k)$ and the magnitude of the step $\|\tDelta \bx_k\|^2$. In particular, we let
\begin{equation}\label{def:acc6}
\nu_k\leftarrow\begin{cases}
\nu_{k-1} & \text{if }\; \nu_{k-1}\leq\nu_k^{\text{trial}},\\(1-\epsilon)\nu_k^{\text{trial}}&\text{otherwise},
\end{cases} \quad\quad \text{where }\quad  {\nu}_{k}^{\text{trial}}\leftarrow\frac{\Delta q(\tDelta\bx_k; \tau_k, \bx_k, \bar{\boldsymbol g}_k,\tB_k)}{\|\tilde{\Delta} \boldsymbol{x}_k\|^{2}}.
\end{equation}
This definition ensures $\nu_k\leq \nu_k^{\text{trial}} = \Delta q(\tDelta\bx_k; \tau_k, \bx_k, \bar{\boldsymbol g}_k,\tB_k)/\|\tilde{\Delta} \boldsymbol{x}_k\|^{2}$. In the end, our adaptive random stepsize $\baralpha_k$ can be selected from any scheme as long as, for a prespecified sequence $\{\alpha_k\}$ and $p\geq 1$,
\begin{equation}\label{snequ:7}
\frac{\nu_k\alpha_k}{\tau_k\kappa_{\nabla f}+\kappa_{\nabla c}} \leq \baralpha_k \leq \frac{\nu_k\alpha_k}{\tau_k\kappa_{\nabla f}+\kappa_{\nabla c}} + \psi \alpha_k^p,
\end{equation}
where $\kappa_{\nabla f}$ and $\kappa_{\nabla c}$ are (estimated) Lipschitz constants of $\nabla f$ and $\nabla c$. We summarize the above DF-SSQP method in Algorithm \ref{Alg:DF-SSQP} and explain the above stepsize selection in the following remark.

\begin{remark}
The above stepsize selection condition \eqref{snequ:7} follows existing designs of derivative-based SSQP \citep{Berahas2021Sequential, Berahas2022Adaptive,Curtis2024Sequential, Curtis2024Stochastic, Na2025Statistical}. Essentially, we just set $\baralpha_k = O(\alpha_k)$, while to introduce the adaptivity into the method, we multiply $\alpha_k$ by the ratio~$\nu_k/(\tau_k\kappa_{\nabla f}+\kappa_{\nabla c})$ and are allowed to increment it with an adaptivity gap $\psi\alpha_k^p$. 
The adaptivity gap is crucial as it distinguishes our random stepsize schemes from deterministic stepsize schemes ($\psi=0$). In the theoretical analysis, we will provide a condition on $p$ to control the adaptivity gap, and the commonly used setting in aforementioned works, $p=2$, will automatically satisfy the condition. 
The ratio $\nu_k/(\tau_k\kappa_{\nabla f}+\kappa_{\nabla c})$, though depends on $k$, will stabilize when $k$ is sufficiently large under proper assumptions. It is less crucial in our study where $\alpha_k$ is a decaying stepsize and determines the convergence rate (i.e., the method still works in the same way if $\alpha_k\leq \baralpha_k\leq \alpha_k+\psi\alpha_k^p$); but the ratio can be particularly effective~when $\alpha_k=\alpha$ is a constant. The inspiration of the ratio comes from imposing the Armijo~condition:
\begin{equation}\label{snequ:8}
\phi_{\tau_k}(\bx_k + \baralpha_k\tDelta\bx_k) \leq \phi_{\tau_k}(\bx_k) - \gamma\baralpha_k \Delta q(\tDelta\bx_k; \tau_k, \boldsymbol x_k,\bar{\boldsymbol g}_k,\tB_k)\quad \quad \text{for }\; \gamma\in(0,1).
\end{equation}
In fact, applying the Taylor's expansion and noting that $\kappa_{\nabla f}$ and $\kappa_{\nabla c}$ are Lipschitz constants of $\nabla f$ and $\nabla c$, we know for $\baralpha_k\leq 1$ that
\begin{align*}
\phi_{\tau_k}(\bx_k & + \baralpha_k\tDelta\bx_k) = \tau_{k} f(\bx_k + \baralpha_k\tDelta\bx_k) +\|c(\bx_k + \baralpha_k\tDelta\bx_k)\| \\
& \leq  \tau_k(f_k + \baralpha_k \nabla f_k^T\tDelta\bx_k) + \|c_k + \baralpha_k G_k\tDelta\bx_k\|+  \frac{1}{2}(\tau_{k}\kappa_{\nabla f} + \kappa_{\nabla c})\baralpha_k^2\|\tDelta\bx_k\|^2\\
& = \phi_{\tau_k}(\bx_k) + \baralpha_k(\tau_{k}\nabla f_k^T\tDelta\bx_k + \|c_k+G_k\tDelta\bx_k\| - \|c_k\|) + \frac{1}{2}(\tau_{k}\kappa_{\nabla f} + \kappa_{\nabla c})\baralpha_k^2\|\tDelta\bx_k\|^2\\
& \stackrel{\mathclap{\eqref{snequ:10}}}{\leq} \phi_{\tau_k}(\bx_k) - \baralpha_k\Delta q(\tDelta\bx_k;\tau_k, \bx_k, \nabla f_k, \tB_k) + \baralpha_k\|c_k+G_k\tDelta\bx_k\|+ \frac{1}{2}(\tau_{k}\kappa_{\nabla f} + \kappa_{\nabla c})\baralpha_k^2\|\tDelta\bx_k\|^2.
\end{align*}
Supposing for the moment that $\barg_k\rightarrow\nabla f_k$ and $\tG_k \rightarrow G_k$ (as proved in Lemma \ref{lemma:average almost sure}), we use $c_k + \tG_k\tDelta\bx_k = \0$ from \eqref{def:acc2} and have for large enough $k$ that ($\lesssim$ only means for ``intuition")
\begin{equation*}
\phi_{\tau_k}(\bx_k + \baralpha_k\tDelta\bx_k) \lesssim  \phi_{\tau_k}(\bx_k) - \baralpha_k\Delta q(\tDelta\bx_k;\tau_k, \bx_k, \barg_k, \tB_k) + \frac{1}{2}(\tau_{k}\kappa_{\nabla f} + \kappa_{\nabla c})\baralpha_k^2\|\tDelta\bx_k\|^2.
\end{equation*}
Combining the above display with \eqref{snequ:8}, we know \eqref{snequ:8} can be satisfied as long as
\begin{equation*}
\baralpha_k \leq \frac{2(1-\gamma)\Delta q(\tDelta\bx_k;\tau_k, \bx_k, \barg_k, \tB_k)}{(\tau_{k}\kappa_{\nabla f} + \kappa_{\nabla c})\|\tDelta\bx_k\|^2}.
\end{equation*}
Note that $\nu_k/(\tau_{k}\kappa_{\nabla f} + \kappa_{\nabla c})$ is a lower bound of the above right-hand side corresponding to $\gamma=1/2$. 
\end{remark}

\begin{algorithm}[tbh!]\caption{Derivative-Free Stochastic SQP (DF-SSQP)}\label{Alg:DF-SSQP}
\begin{algorithmic}[1]
\State \textbf{Input:} initial iterate $(\bx_0, \blambda_0)\in\mR^d\times\mR^m$, $\barg_{-1}\in\mR^d$, $\barG_{-1}\in\mR^{m\times d}$, $\barB_{-1}=I$, ${\tau}_{-1}, {\nu}_{-1}>0$;~positive sequences $\{b_{k}, \tilde{b}_{k}, \alpha_k, \beta_k\}$, tuning parameters $\sigma, \varepsilon \in (0,1)$, $\psi\geq0$, $p\geq 1$.
\For {$k=0,1,\cdots,$}
\State Compute derivative approximations with debiasing steps to obtain $\tG_k$, $\tB_k$, $\bnabla_{\bx}\mL_k$.
\State  Solve Newton system \eqref{def:acc2} to obtain $(\tDelta \bx_k, \tDelta \blambda_k)$.
\State Compute $\tau_k$ as in \eqref{def:acc4}, $\nu_k$ as in \eqref{def:acc6}, and then select any random stepsize $\baralpha_k$ as in \eqref{snequ:7}.
\State Update $(\bx_{k+1}, \blambda_{k+1}) \leftarrow (\bx_k, \blambda_k) + \baralpha_k (\tDelta \bx_k, \tDelta \blambda_k).$
\EndFor
\end{algorithmic}
\end{algorithm}

\section{Global Convergence Analysis} \label{sec:3}

In this section, we establish the global almost sure convergence guarantee for Algorithm \ref{Alg:DF-SSQP}. We begin by stating assumptions.

\begin{assumption}\label{ass:1-1}
Let $\mathcal{X} \subseteq \mathbb{R}^d$ be an open convex set that contains the evaluation sequences $\{\bx_k, \bx_k \pm b_k \bDelta_k, \boldsymbol{x}_k \pm b_k \boldsymbol{\Delta}_k + \widetilde{b}_k \widetilde{\boldsymbol{\Delta}}_k\}$. We assume that the objective $f(\bx)$ and constraints $c(\bx)$ are thrice differentiable, with bounded first, second, and third derivatives over $\mX$, and $f(\bx)$ is bounded below by~$f_{\text{inf}}$ over $\mX$. Moreover, we assume there exist constants $\kappa_{c}$, $\kappa_{1,G}$, $\kappa_{2,G}$,  $\kappa_{1,\tG}$, $\kappa_{2,\tG}>0$ such that
\begin{equation*}
\|c_k\|\leq \kappa_{c},\quad\quad\kappa_{1,G} \cdot I \preceq G_k G_k^T\preceq \kappa_{2,G} \cdot I,\quad\quad  \kappa_{1,\tG} \cdot I \preceq \tilde{G}_k \tilde{G}_k^T \preceq \kappa_{2,\tG} \cdot I, \quad\forall k\geq 0.
\end{equation*}
Similarly, we assume the regularization $\delta_k^B$ in \eqref{snequ:6} ensures that $\tB_k$ satisfies $\boldsymbol{x}^T\tB_k\boldsymbol{x}\geq\kappa_{1,\tB}\|\boldsymbol{x}\|^2$ for any $\boldsymbol{x}\in\{\boldsymbol{x}\in\mathbb{R}^d:\tG_k\boldsymbol{x}=\boldsymbol{0}\}$ and $\|\tB_k\|\leq\kappa_{2,\tB}$, for some constants $\kappa_{1,\tB},\kappa_{2,\tB}>0$. 	
\end{assumption}

Assumption \ref{ass:1-1} is standard in the SSQP and/or derivative-free optimization literature. In particular, the existence of an open convex set $\mX$ and the boundedness of the associated quantities of the objective and constraints within the set have been widely imposed in \cite{Bertsekas1982Constrained, Berahas2021Sequential, Berahas2023Stochastic, Curtis2024Stochastic, Fang2024Fully, Fang2024Trust}. The requirement for thrice differentiability arises from derivative-free, simultaneous perturbation techniques \citep{Spall1992Multivariate, Spall2000Adaptive, Spall2003Introduction}.~This~\mbox{assumption}~can~certainly be relaxed if we are only concerned with global convergence without approximating Hessians.$\quad\;\;$

The exact Jacobian $G_k$ is assumed to have full row rank, which is also commonly assumed in the aforementioned literature. \cite{Berahas2023Stochastic} relaxed the full-rank condition to a rank-deficient~scenario, although that study employs more sophisticated (derivative-based) designs with weaker convergence guarantees. In addition, we assume our regularization $\delta_k^G$ in \eqref{snequ:9} perturbs $\barG_k$ to $\tG_k$ to ensure that $\tG_k$ is also full row-rank. In the subsequent analysis, we further require $[\kappa_{1,G}, \kappa_{2,G}]\subseteq (\kappa_{1,\tG}, \kappa_{2,\tG})$ to have the perturbation vanish in the limit, provided we can show $\barG_k\rightarrow G_k$ as $k\rightarrow\infty$. Analogously, we assume $\delta_k^B$ in \eqref{snequ:6} perturbs $\barB_k$ to $\tB_k$ to ensure that $\tB_k$ is lower bounded in the null space $\text{ker}(\tG_k)$. As introduced earlier, this condition, together with the full row-rank condition of $\tG_k$, ensures the well-definedness of the Newton system \eqref{def:acc2}.

\begin{assumption}\label{ass:2-1}
For any $ \xi \sim \mathcal{P} $ and $ \bx \in \mathcal{X}$, we assume $\mathbb{E}[ F\left(\boldsymbol{x}; \xi \right) \mid \bx] = f(\boldsymbol{x})$ and there exists~a~constant $\Upsilon_{m} > 0$ such that 
\begin{subequations}\label{cond:grad}
\begin{align}
\text{Bounded $r$-moment}:& \quad\quad \mE[\|\nabla F(\boldsymbol{x}; \xi) - \nabla f(\boldsymbol{x})\|^r\mid \bx]  \leq \Upsilon_{m}, \label{cond:grad:smo}\\
\text{Uniformly bounded}:& \quad\quad \|\nabla F(\boldsymbol{x}; \xi) - \nabla f(\boldsymbol{x})\| \leq \Upsilon_{m}. \label{cond:grad:bound}
\end{align}
\end{subequations}
\end{assumption}

We note that \eqref{cond:grad:bound} implies \eqref{cond:grad:smo} if we redefine $\Upsilon_{m} \leftarrow \Upsilon_{m}^r$ in \eqref{cond:grad:smo}. In general, we only assume~that $\nabla F(\bx;\xi)$ has a bounded $r$-moment for some appropriate $r \geq 1$ as in \eqref{cond:grad:smo} when studying the properties of the finite-difference estimate $\hat{\nabla} F(\bx;\xi)$ in \eqref{snequ:2} and the debiased estimate $\barg$ in \eqref{snequ:3}. However, we impose the stronger condition \eqref{cond:grad:bound} to establish the global convergence guarantee of DF-SSQP, in line with~the existing SSQP literature \citep{Berahas2021Sequential, Berahas2023Stochastic, Curtis2024Stochastic, Na2022adaptive, Na2023Inequality, Fang2024Fully, Fang2024Trust}.

While unconstrained methods only require a bounded variance condition, the boundedness condition is crucial for constrained methods to ensure the stabilization of the merit and ratio parameters $(\tau_k,\nu_k)$. This stabilization is provably guaranteed only when gradients are bounded, even in deterministic settings \citep{Bertsekas1982Constrained}. Stabilizing these parameters is important for asymptotic analysis, as we want the iterates to reduce the same merit function (at least for all sufficiently large $k$), rather than a different merit function at each step.
That being said, condition \eqref{cond:grad:bound} naturally holds for finite-sum problems in machine learning, which are a key application of DFO methods. Additionally, the~\mbox{boundedness}~of~gradient noise can be replaced by a uniform Lipschitz continuity condition on the objective functions $F(\bx;\xi)$. We mention that \cite{Sun2023trust, Sun2024Trust} imposed a bounded gradient noise condition and incorporated the bound into the design of a trust-region method. Our study differs from theirs in that~$\Upsilon_{m}$ is unknown in our setting.

The next assumption regards the distribution $\P_{\bDelta}$ of the random direction $\bDelta\in\mR^d$, which is standard in the simultaneous perturbation literature \citep{Spall1992Multivariate, Spall2000Adaptive, Spall2003Introduction} and can be satisfied~by various direction generation distributions; e.g., $\bDelta$ has independent Rademacher entries.

\begin{assumption}\label{ass:Delta}
For $k\geq 0$, we assume $\bDelta_k, \tbDelta_k\sim\P_{\bDelta}$ are independent. For any $\bDelta\sim\P_{\bDelta}$, we~assume $\bDelta$ has mutually independent entries, each symmetrically distributed about zero with absolute value bounded both from above and below by some constants $\kappa_{\bDelta_1}, \kappa_{\bDelta_2}>0$:
\begin{equation*}
\kappa_{\bDelta_1} \leq |\bDelta^j| \leq \kappa_{\bDelta_2},\quad\quad \text{for }\; 1\leq j \leq d.
\end{equation*}
Here, the superscript $j$ denotes the $j$-th entry of $\bDelta$.
\end{assumption}

Finally, to ease later presentation, we state several polynomial sequences in the next assumption.

\begin{assumption}\label{ass:3-1}
We let 
\begin{equation*}
\alpha_k = \frac{\iota_{1}}{(k+1)^{p_1}},\quad\quad \beta_k= \frac{\iota_{2}}{(k+1)^{p_2}}, \quad\quad b_k = \frac{\iota_{3}}{(k+1)^{p_3}}, \quad\quad \tilde{b}_k = \frac{\iota_{4}}{(k+1)^{p_4}},
\end{equation*}
where $\iota_i, p_i>0$ for $i = 1,2,3,4$.
\end{assumption}

In the next subsection, we present preliminary guarantees for derivative approximations, which~serve as the foundation for establishing the global convergence of DF-SSQP.

\subsection{Guarantees for derivative approximations}

Let us introduce some additional notation. We define $\F_{-1}\subseteq\F_0\subseteq\F_1\cdots$ as a filtration of $\sigma$-algebras, where $\mF_k=\sigma(\{\xi_i,\boldsymbol \Delta_{i},\widetilde{\boldsymbol \Delta}_{i}\}_{i=0}^{k})$, $\forall k\geq 0$ contains all the randomness before performing the $(k+1)$-th iteration, and $\mF_{-1} = \sigma(\{\bx_0,\blambda_0\})$ is the trivial $\sigma$-algebra. For a random vector/matrix sequence $\{Y_k\}$ and a deterministic scalar sequence $\{y_k\}$, we write $Y_k=O(y_k)$ if $\|Y_k\|/y_k$ is uniformly bounded over sample paths. Recall that we denote $c_k= c(\bx_k)$, $G_k= \nabla c_k = \nabla  c(\boldsymbol{x}_k)$ (similar for $\nabla f_k$, $\nabla^2f_k$ etc.) for notational simplicity.

Our first result characterizes the bias of the randomized gradient and Hessian approximations. We observe that the conditional bias converges to zero as $k$ goes to infinity almost surely.

\begin{lemma}\label{lemma:finite difference for gradient}
Under Assumptions \ref{ass:1-1}, \ref{ass:2-1}, \ref{ass:Delta}, we have for $1\leq j\leq m$,
\begin{alignat*}{2}
\mE[\hnabla F(\bx_k;\xi_k)-\nabla f_k\mid\mF_{k-1}]& = O(b_k^2) ,\hskip2.45cm
\mE[ \hnabla c_k-\nabla c_k\mid\mF_{k-1}]&& = O(b_k^2),\\
\mE[\hnabla^2F(\bx_k;\xi_k)-\nabla^2 f_k\mid\mF_{k-1}] & = O(b_k+\tb_k^2/b_k),\quad\quad \mE[\hnabla^2 c^j_k-\nabla^2 c^j_k\mid\mF_{k-1}]&& = O(b_k+\tb_k^2/b_k).
\end{alignat*}
\end{lemma}

\begin{proof}
See Appendix \ref{appendix lemma:finite difference for gradient}. 
\end{proof}

In the following lemma, we demonstrate the almost sure convergence of the unconditional bias in the debiased gradient and Hessian approximations computed via the moving averaging technique.

\begin{lemma}\label{lemma:average almost sure}
Under Assumptions \ref{ass:1-1}, \ref{ass:2-1}\eqref{cond:grad:smo}, \ref{ass:Delta}, \ref{ass:3-1}, we further assume that
\begin{equation}\label{cond:3.6}
p_2\in(0.5,1], \quad\quad p_1>p_2, \quad\quad  p_3>0.5-0.5p_2, \quad\quad r\geq 2,\quad\quad r(p_1-p_2)>1. 
\end{equation}
Then, we have $\barg_k - \nabla f_k\rightarrow \0$ and $\barG_k- G_k\rightarrow \0$ as $k\rightarrow\infty$ almost surely. Furthermore, if $\delta_k^G$ ensures~that $[\kappa_{1,G},\kappa_{2,G}]\subseteq(\kappa_{1,\tG},\kappa_{2,\tG})$, then there exists a (potentially random) $K^\star_G<\infty$ such that for all $k\geq K^\star_G$, $\tG_k = \barG_k$, i.e., $\delta_k^G = \0$.
\end{lemma}

\begin{proof}
See Appendix \ref{appendix lemma:average almost sure}.
\end{proof}

The next lemma establishes the convergence rate in expectation of $\barg_k$ and $\barG_k$.

\begin{lemma}\label{final fundamental lemma}
	
Under Assumptions \ref{ass:1-1}, \ref{ass:2-1}\eqref{cond:grad:smo}, \ref{ass:Delta}, \ref{ass:3-1}, we further assume that
\begin{equation}\label{cond:3.7}
p_2\in(0, 1],\quad\quad p_1>p_2,\quad\quad r\geq 2,\quad\quad \iota_{2}>0.5>p_3/\iota_{2} \;\;(\text{if } p_2=1),\quad\quad p_1<1+\iota_{2} \;\;(\text{if } p_2=1).
\end{equation}
Then, we have
\begin{equation*}
\mE[\left\|\barg_k-\nabla f_k\right\|^2]=O(\beta_k+b_k^4+\alpha_k^2/\beta_k^2), \quad\quad \mE[\left\|\bar G_k-G_k\right\|^2]=O(\beta_k+b_k^4+\alpha_k^2/\beta_k^2).
\end{equation*}
\end{lemma}

\begin{proof}
See Appendix \ref{appendix lemma:final fundamental lemma}.
\end{proof}

The convergence rate in expectation established in Lemma \ref{final fundamental lemma} resembles the rate shown in \cite{Na2024Asymptotically}; however, it includes an additional term $b_k^4$, which arises from the bias introduced by our derivative-free estimator. Notably, the result of Lemma \ref{final fundamental lemma} can be improved through local analysis, as the direction $\tilde{\Delta} \boldsymbol{x}_k$ is merely treated as a term with bounded second moment in the global analysis, while it is shown to vanish in the local analysis. 
Further details on refining the bound of $\tilde{\Delta} \boldsymbol{x}_k$ will be provided in the statistical inference analysis in Section \ref{sec:4}. Specifically, see Lemma \ref{lem:local:rate} and Lemma \ref{lem:2} in Appendix \ref{pf:lem:local:rate} for the improvement of the error term $\alpha_k^2/\beta_k^2$.

\subsection{Global almost sure convergence}

In this subsection, we establish the global almost sure convergence of DF-SSQP. We first decompose~the direction step $\tDelta \bx_k$ as a tangential step $\bu_k$ and a normal step $\bv_k$ as
\begin{equation} \label{decomposition}
\tilde{\Delta} \boldsymbol{x}_k = \boldsymbol u_k + \boldsymbol v_k, \quad \text{where} \quad \boldsymbol u_k \in \text{Null}(\tilde{G}_k) \quad \text{and} \quad \boldsymbol v_k \in \text{Range}(\tilde{G}_k^T).
\end{equation}
The first lemma establishes an upper bound for $\boldsymbol v_k$ in terms of $c_k$ in (i), a lower bound for the curvature of $\tB_k$ along $\tDelta \bx_k$ in terms of $\bu_k$ in (ii), and a lower bound on the reduction of the local model in (iii).

\begin{lemma}\label{thm:convergence1}
Under Assumption \ref{ass:1-1}, there exist constants $ \kappa _{v}, \kappa _{u}, \kappa _{q } >0$ such that the following~statements hold true for all $k\geq 0$.
\begin{enumerate}[label=(\alph*),topsep=2pt]
\setlength\itemsep{0.0em}
\item $\boldsymbol v_k$ satisfies max$\{ \| \boldsymbol v_k\|, \| \boldsymbol v_k\|^2\} \leq \kappa _v\|c_k\|$.
\item If $\|\boldsymbol u_k\|^2 \geq \kappa _{u} \|\boldsymbol v_k\|^2$, then $\tilde{\Delta} \boldsymbol{x}_k^T\tB_k\tilde{\Delta} \boldsymbol{x}_k\geq 0.5\kappa_{1,\tB}\|\boldsymbol u_k\|^2$.
\item The reduction of the local model satisfies $\Delta q(\tDelta\bx_k; \tau_k, \boldsymbol x_k,\bar{\boldsymbol g}_k,\tB_k)\geq\kappa _{q }\tau_k(\|\tDelta \bx_k\|^2+\|c_k\|)$.
\end{enumerate}	
\end{lemma}

\begin{proof}
See Appendix \ref{appendix thm:convergence1}.
\end{proof}

In the next lemma, we demonstrate the stabilization of the merit and ratio parameters~$(\tau_k,\nu_k)$,~which is the only, yet crucial, result for which we require the boundedness condition \eqref{cond:grad:bound}.

\begin{lemma}\label{parameter stabilize}
Under Assumptions \ref{ass:1-1}, \ref{ass:2-1}\eqref{cond:grad:bound}, \ref{ass:Delta}, there exist a (potentially random) $K^\star_{\tau\nu}<\infty$ and deterministic constants $\tilde{\tau}, \tilde{\nu}>0$ such that for all $k\geq K^\star_{\tau\nu}$, $\tau_k = \tau_{K^\star_{\tau\nu}} \geq \tilde{\tau}$ and $\nu_k = \nu_{K^\star_{\tau\nu}} \geq \tilde{\nu}$.
\end{lemma}

\begin{proof}
See Appendix \ref{appendix parameter stabilize}.
\end{proof}

Then, we establish the liminf-type convergence guarantee for the reduction of the local model, which is a key step toward proving the limit-type convergence guarantee for Algorithm \ref{Alg:DF-SSQP}. Let us~denote $(\Delta \bx_k,\Delta\blambda_k)$ to be the solution of \eqref{def:acc2}, but with $\barg_k$ replaced by $\nabla f_k$ and $\tilde{G}_k$ replaced by $G_k$.

\begin{lemma}\label{convergence2}

Under Assumptions \ref{ass:1-1}, \ref{ass:2-1}\eqref{cond:grad:smo}, \ref{ass:Delta}, \ref{ass:3-1}, we further assume that (i) $\delta_k^G$ ensures $[\kappa_{1,G},\kappa_{2,G}]\subseteq(\kappa_{1,\tG},\kappa_{2,\tG})$, (ii) $p_1, p_2, p_3, r$ satisfy
\begin{equation}\label{cond:3.10}
p_1\in(0.75,1],\quad\quad p_2\in(0.5,2p_1-1),\quad\quad p_3>0.5-0.5p_2,\quad\quad r(p_1-p_2)>1,
\end{equation}
and (iii) the statement of Lemma \ref{parameter stabilize} holds (ensured by \eqref{cond:grad:bound}). Then, we have almost surely
\begin{equation*}
\liminf_{k \to \infty} \Delta q(\tDelta\bx_k;\tau_k, \bx_k, \barg_k, \tB_k)=0 \quad\quad\text{ and }\quad\quad \liminf_{k \to \infty} (\|\Delta \boldsymbol{x}_k\|+\|c_k\|)=0.
\end{equation*}
\end{lemma}

\begin{proof}
See Appendix \ref{appendix convergence2}.
\end{proof}

We note that the condition \eqref{cond:3.10} implies both \eqref{cond:3.6} and \eqref{cond:3.7}. In particular, since $p_1\leq 1$, we have $2p_1-1\leq \min\{p_1, 1\}$; thus, \eqref{cond:3.10} implies $p_2<\min\{p_1, 1\}$ as required by \eqref{cond:3.6} and \eqref{cond:3.7}. Furthermore,~using $p_2>0.5$ and $p_1\leq 1$, we obtain $r(p_1-p_2)>1\Rightarrow r(p_1-0.5)>1\Rightarrow r>2$; thus, the condition $r\geq 2$ in \eqref{cond:3.6} and \eqref{cond:3.7} is also satisfied. In addition, since $p_1>0.75$ implies $2p_1-1>0.5$, we note that a feasible region always exists for our parameters $\{p_1,p_2,p_3, r\}$.

In the next theorem, we establish the global convergence guarantee of Algorithm \ref{Alg:DF-SSQP}. Given the primal iterate $\bx_k$ generated by Algorithm \ref{Alg:DF-SSQP}, we define the least squares estimate of the dual solution $\boldsymbol \lambda_k^\star$ as $\blambda_k^\star = -[\tG_k\tG_k^T]^{-1}\tG_k\barg_k$ (note that Assumption \ref{ass:1-1} ensures $\tG_k$ has full row rank,~making $\boldsymbol{\lambda}_k^\star$ well-defined). The next theorem states that the KKT residual of the primal solution $\bx_k$, along with its~least-squares dual estimate $\blambda_k^\star$, converges to zero from any initialization almost surely.

\begin{theorem}\label{final convergence}
Under the same conditions as in Lemma \ref{convergence2}, we have
\begin{equation*}
\lim_{k \to \infty} (\|\nabla f_k+ G_k^T\boldsymbol\lambda_k^\star\|_2+\|c_k\|_2)=0\quad \text{almost surely}.
\end{equation*}
\end{theorem}

\begin{proof}
See Appendix \ref{appendix final convergence}.
\end{proof}

We note that our almost sure convergence result matches those established for both line-search-based SSQP methods \citep{Na2022adaptive, Na2023Inequality, Curtis2025Almost} and trust-region-based SSQP methods \citep{Fang2024Fully, Fang2024Trust}.
This almost sure guarantee differs from some prior works that established a liminf-type convergence guarantee for the expected KKT residual \citep{Berahas2021Sequential, Berahas2023Stochastic}. Furthermore, all prior works studied derivative-based methods, while our almost sure result is established for derivative-free SSQP schemes by leveraging the simultaneous perturbation technique.

\section{Local Asymptotic Normality}\label{sec:4}

In this section, we establish the local asymptotic normality guarantee for the iterates $(\boldsymbol x_k,\boldsymbol\lambda_k)$ of Algorithm \ref{Alg:DF-SSQP}. To set the stage for statistical inference, we first introduce several local assumptions that~aim to characterize the algorithm's asymptotic behavior.

\begin{assumption}\label{ass:4.1}
We assume $\boldsymbol x_k\to\boldsymbol x^\star$ as $k\rightarrow\infty$ to a strict local solution $\boldsymbol x^\star$ that satisfies: 
\begin{enumerate}[label=(\alph*),topsep=2pt]
\setlength\itemsep{0.0em}
\item Linear Independence Constraint Qualification (LICQ): $G^\star=\nabla c(\bx^\star)$ has full row rank.
\item Second-Order Sufficient Condition (SOSC): let $\tlambda\in\mR^m$ be the unique Lagrangian multiplier vector satisfying the KKT conditions \eqref{equ:KKT}. We assume $\boldsymbol{x}^T\nabla_{\boldsymbol x}^2\mathcal{L}(\tx,\tlambda)\boldsymbol{x}\geq\kappa_{1,B}\|\boldsymbol{x}\|^2$ for any $\boldsymbol{x}\in\{\boldsymbol{x}\in\mathbb{R}^d:G^\star\boldsymbol{x}=\boldsymbol{0}\}$ and $\|\nabla_{\boldsymbol x}^2\mathcal{L}^\star\|\leq\kappa_{2,B}$ for some constants $\kappa_{1,B}, \kappa_{2,B}>0$.
\end{enumerate}
\end{assumption}

\begin{assumption} \label{ass:4.2}
We assume the Hessian of the sample function has bounded variance near the solution $\tx$. That is, for some $\delta>0$ and any $\bx\in \mX\cap\{\bx:\|\bx - \tx\|_2\leq \delta\}$, there exists a constant $\Upsilon_{n}$ such~that
\begin{equation*}
\mE[\|\nabla^2F(\bx;\xi)-\nabla^2f(\bx)\|^2\mid \bx] \leq \Upsilon_n.
\end{equation*}
\end{assumption}

\begin{assumption}\label{ass:4.3}
We assume almost surely, $\tau_k=\tau$, $\nu_k=\nu$, $\forall k\geq K^\star_{\tau\nu}$ for a (potentially random)~index $K^\star_{\tau\nu}<\infty$ and two deterministic constants $\tau, \nu>0$.
\end{assumption}

Assumption \ref{ass:4.1} is standard in the literature for analyzing the local asymptotic behavior of both deterministic and stochastic algorithms for solving constrained nonlinear nonconvex problems \citep{Bertsekas1982Constrained,Nocedal2006Numerical,Duchi2021Asymptotic, Davis2024Asymptotic, Na2025Statistical}. It is also well known that LICQ and SOSC are necessary conditions even for establishing the asymptotic normality of offline $M$-estimation \cite[Chapter 5]{Shapiro2021Lectures}, which ensure the limiting covariance matrix of the $M$-estimator is well-defined (see \cite[(1.3)]{Na2025Statistical} for more~details).

Similar to the conditions on the perturbation $\delta_k^G$ in \eqref{snequ:9}, we will later require that the perturbation $\delta_k^B$ in \eqref{snequ:6} perturbs $\barB_k$ to $\tB_k$ such that the bounds of $\tB_k$ satisfy $[\kappa_{1,B}, \kappa_{2,B}]\subseteq(\kappa_{1,\tB}, \kappa_{2,\tB})$. This~ensures that the perturbation $\delta_k^B$ vanishes in the limit as long as $\barB_k\rightarrow \nabla_{\bx}^2\mL^\star$ as $k\rightarrow\infty$. We also recall that the Hessian approximation is only used to achieve favorable local convergence properties; hence,~the~bounded variance condition is imposed only locally in Assumption \ref{ass:4.2}.

Assumption \ref{ass:4.3} enforces that the merit and ratio parameters $(\tau_k, \nu_k)$ stabilize almost surely at some constants $(\tau, \nu)$. By Lemma \ref{parameter stabilize}, we know that the boundedness condition \eqref{cond:grad:bound} ensures $(\tau_k, \nu_k)$ always stabilize, although the limiting values $(\tau_\infty, \nu_\infty) = (\tau_{K^\star_{\tau\nu}}, \nu_{K^\star_{\tau\nu}})$ may vary across different runs. 
On~the other hand, the assumption that $(\tau_\infty, \nu_\infty) = (\tau, \nu)$ are constants is made solely to streamline the~analysis and highlight the core derivation. In particular, $(\tau_k, \nu_k)$ only play a role in affecting the stepsize $\baralpha_k$ via the factor $\nu_k / (\tau_k \kappa_{\nabla f} + \kappa_{\nabla c})$ in \eqref{snequ:7}, which, as shown in Theorem \ref{thm:normality}, may scale the variance of the limiting normal distribution. 
Since $(\tau_k, \nu_k)$ are updated multiplicatively by a factor of $1 - \epsilon$ and are constrained within deterministic lower and upper bounds (cf. Lemma \ref{parameter stabilize}), we know the limiting pair $(\tau_\infty, \nu_\infty)$ can only take finitely many discrete values $\{(\tau_{(i)}, \nu_{(i)})\}_{i=1}^N$, forming a discrete distribution.~Consequently,~the factor $\nu_\infty / (\tau_\infty \kappa_{\nabla f} + \kappa_{\nabla c})$ also follows a discrete distribution with finite support.
Therefore, by adjusting the filtration from $\mF_k$ to the trace filtration $\mF_k \cap \{(\tau_\infty, \nu_\infty) = (\tau_{(i)}, \nu_{(i)})\}$\footnote{The trace filtration contains all randomness from $\bx_0$ to $\bx_k$ conditioned on the event $\{(\tau_\infty, \nu_\infty) = (\tau_{(i)}, \nu_{(i)})\}$.}, we can follow the same line of analysis and obtain a limiting mixture normal distribution with $N$ components, where each~component has the weight $P(\{(\tau_\infty, \nu_\infty) = (\tau_{(i)}, \nu_{(i)})\})$. Since this extension is tedious and of limited interest and contribution, we leave it for future work.

In the following lemma, we demonstrate that the iterates $(\boldsymbol x_k,\boldsymbol{\lambda}_k)$ converge almost surely to the local solution $(\tx, \tlambda)$. Note that the conditions on $(p_1, p_2, p_3, r)$ and $\delta_k^G$ below are implied by (i.e.,~weaker than) those required for the global convergence in Theorem \ref{final convergence} (i.e., Lemma \ref{convergence2}).

\begin{lemma}\label{iterate convergence}
Under Assumptions \ref{ass:1-1}, \ref{ass:2-1}\eqref{cond:grad:smo}, \ref{ass:Delta}, \ref{ass:3-1}, \ref{ass:4.1}, we further assume that (i) $\delta_k^G$ ensures $[\kappa_{1,G},\kappa_{2,G}]\subseteq(\kappa_{1,\tG},\kappa_{2,\tG})$, and (ii) $p_1, p_2, p_3, r$ satisfy
\begin{equation}\label{cond:4.4}
p_1\in(0.5,1],\quad\quad p_2\in(0.5,p_1),\quad\quad p_3>0.5-0.5p_2,\quad\quad r(p_1-p_2)>1.
\end{equation}
Then, $(\boldsymbol x_k,\boldsymbol\lambda_k) \to (\tx,\tlambda)$ as $k\rightarrow \infty$ almost surely.
\end{lemma}

\begin{proof}
See Appendix \ref{appendix iterate convergence}.
\end{proof}

With the convergence of the iterates, we further illustrate the convergence of the Hessian approximations. Noting that together with the convergence of $\tG_k$ in Lemma \ref{lemma:average almost sure}, we obtain the convergence of the KKT matrix $\tW_k$ in \eqref{def:acc2}. Note that for Hessian convergence, we additionally impose Assumption \ref{ass:4.2} and a condition on $p_4$ (cf. $\tilde{b}_k = \iota_{4}/(k+1)^{p_4}$) upon the conditions in Lemma \ref{iterate convergence}.

\begin{lemma}\label{lem:Hessian Convergence}
	
Under Assumptions \ref{ass:1-1}, \ref{ass:2-1}\eqref{cond:grad:smo}, \ref{ass:Delta}, \ref{ass:3-1}, \ref{ass:4.1}, \ref{ass:4.2}, we further assume that (i) $\delta_k^G$ ensures $[\kappa_{1,G},\kappa_{2,G}]\subseteq(\kappa_{1,\tG},\kappa_{2,\tG})$, and (ii) $p_1,p_2,p_3,p_4,r$ satisfy
\begin{equation}\label{cond:4.5}
p_1\in(0.5,1], \quad\quad p_2\in(0.5,p_1),\quad\quad p_3>0.5-0.5p_2, \quad\quad p_4>0.5p_3,\quad\quad r(p_1-p_2)>1.
\end{equation}
Then, $\barB_k\rightarrow\nabla_{\bx}^2\mL^\star$ as $k\rightarrow\infty$ almost surely. Furthermore, if $\delta_k^B$ ensures $[\kappa_{1,B}, \kappa_{2,B}]\subseteq(\kappa_{1,\tB}, \kappa_{2,\tB})$, then there exists a (potentially random) $K_B^\star<\infty$ such that for all $k\geq K_B^\star$, $\tB_k = \barB_k$, i.e., $\delta_k^B = \0$.
\end{lemma}

\begin{proof}
See Appendix \ref{appendix Hessian Convergence}.
\end{proof}

To proceed to establishing the asymptotic normality guarantee of the iterate, we next provide the local convergence rates of the iterate and the gradient approximation. We use $\bz_k = (\bx_k - \tx, \blambda_k - \tlambda)$ to denote the error of the primal-dual pair, and define two matrices used frequently later
\begin{equation*}
W^\star = \nabla^2\mL^\star = \begin{pmatrix}
\nabla_{\bx}^2\mL^\star & (G^\star)^T\\
G^\star & \0
\end{pmatrix}\quad\text{ and }\quad \tOmega = \begin{pmatrix}
\mE[\bDelta^{-1}\bDelta^T\text{Cov}(\nabla F(\tx;\xi))\bDelta\bDelta^{-T}] & \0\\
\0 & \0
\end{pmatrix}.
\end{equation*}
Our local neighborhood is characterized by a stopping time, defined for any $k_0 \geq 0$ and $\epsilon>0$ as~follows:
\begin{align}\label{equ:def:tau}
& \tau_{k_0}(\epsilon)  = \inf\bigg\{k\geq k_0: \|\bz_k\| >\epsilon^2 \text{ OR } \|\tW_k^{-1}\|>\frac{1}{\epsilon} \text{ OR } \|\nabla\mL_k - \tW_k\bz_k\| > 0.25\epsilon^2\|\bz_k\| \nonumber \\
&  \hskip2.2cm \text{ OR } \|\nabla\mL_k\|>\frac{\|\bz_k\|}{\epsilon} \text{ OR } \delta_k^G\neq \0 \text{ OR } \delta_k^B \neq \0 \text{ OR } \|(\bx_k, \blambda_k)\|>\frac{1}{\epsilon} \nonumber\\
&\hskip2.2cm \text{ OR }  \|\nabla\mL_k - W^\star\bz_k\| >\frac{\|\bz_k\|^2}{\epsilon} \text{ OR } \frac{\nu_k}{\tau_k\kappa_{\nabla f} +\kappa_{\nabla c}}\neq \frac{\nu}{\tau\kappa_{\nabla f} +\kappa_{\nabla c}}\eqqcolon\zeta \bigg\}.
\end{align}
As expected, when $\epsilon$ is chosen sufficiently small, for each run of the algorithm, there always exists a~(potentially random) $\tilde{k}_0 > 0$ such that $\tau_{k_0}(\epsilon) = \infty$ for all $k_0 \geq \tilde{k}_0$.

With the definition \eqref{equ:def:tau}, we have the following local convergence rate result.

\begin{lemma}\label{lem:local:rate}
Under Assumptions \ref{ass:1-1}, \ref{ass:2-1}\eqref{cond:grad:smo}, \ref{ass:Delta}, \ref{ass:3-1}, and we further assume that 
\begin{equation}\label{cond:4.6}
p_1\in(0, 1],\quad\quad p_2\in(0, p_1),\quad\quad r\geq 2,\quad\quad  \zeta\iota_{1}>0.5 \;\;(\text{if } p_1=1).
\end{equation}
Then, for any $\epsilon\in(0, 1 - 0.5/(\zeta\iota_{1})\1_{p_1=1})$, there exists a deterministic integer $\bar{k}_0>0$ such that for~any $k_0\geq \bar{k}_0$, there exists a constant $\Upsilon(k_0)$ (depending on $k_0$) such that 
\begin{equation*}
\max\cbr{\mE[\|\bz_k\|^2 \1_{\tau_{k_0}(\epsilon)>k}],\; \mE[\|\bnabla\mL_k - \nabla\mL_k\|^2\1_{\tau_{k_0}(\epsilon)>k}]} \leq \Upsilon(k_0)\rbr{\beta_k + b_k^4} \quad\quad \text{ for any } k\geq k_0.
\end{equation*}

\end{lemma}

\begin{proof}
See Appendix \ref{pf:lem:local:rate}.
\end{proof}

The above lemma also leads to the local convergence rate of the Hessian approximation.

\begin{lemma}\label{lem:local:rate:Hessian}
Under the setup of Lemma \ref{lem:local:rate} and additionally supposing Assumptions \ref{ass:4.1}, \ref{ass:4.2} hold~and $p_4>0.5p_3$, we have
\begin{equation*}
\|\tW_k - W^\star\|^2\1_{\tau_{k_0}(\epsilon)>k} = O_p\rbr{\beta_k+b_k^2+\tilde{b}_k^4/b_k^2}.
\end{equation*}
\end{lemma}

\begin{proof}
See Appendix \ref{pf:lem:local:rate:Hessian}.
\end{proof}

Combining all above lemmas, we are ready to state asymptotic normality result.

\begin{theorem}\label{thm:normality}

Under Assumptions \ref{ass:1-1}, \ref{ass:2-1}\eqref{cond:grad:smo}, \ref{ass:Delta}, \ref{ass:3-1}, \ref{ass:4.1}, \ref{ass:4.2}, \ref{ass:4.3}, and we further assume that (i) $\delta_k^G$ ensures $[\kappa_{1,G},\kappa_{2,G}]\subseteq(\kappa_{1,\tG},\kappa_{2,\tG})$ and $\delta_k^B$ ensures $[\kappa_{1,B}, \kappa_{2,B}]\subseteq(\kappa_{1,\tB}, \kappa_{2,\tB})$, (ii)~$p, p_1, p_2, p_3, p_4, r$~\mbox{satisfy}
\begin{equation}\label{cond:4.8}
\begin{gathered}
p_1\in(0.5, 1], \quad p_2\in(0.5,p_1), \quad p_3>\max\cbr{0.5-0.5p_2, 0.25p_1}, \quad p_4> 0.5p_3+0.25(p_1-p_2),\\
p> 1.5-0.5p_2/p_1, \quad\quad r(p_1-p_2)>1, \quad\quad r\geq 3,
\end{gathered}
\end{equation}
and $\zeta\iota_{1}>0.5$ if $p_1=1$. Then, we have
\begin{equation}\label{equ:normality}
1/\sqrt{\baralpha_k}\cdot(\bx_k- \tx, \blambda_k -\tlambda)\stackrel{d}{\longrightarrow}\N\rbr{\0,\; \omega\cdot(W^\star)^{-1}\tOmega (W^\star)^{-1}}\quad \text{with }\;\;  \omega = \begin{cases*}
\frac{\zeta\iota_{1}}{2\zeta\iota_{1}-1} & \text{ if } $p_1=1$,\\
0.5 &  \text{ if } $p_1<1$.
\end{cases*}
\end{equation}
\end{theorem}

\begin{proof}
See Appendix \ref{pf:thm:normality}.
\end{proof}

We note that the conditions on $\{p,p_1,p_2,p_3,p_4,r\}$ can be easily satisfied. The condition \eqref{cond:4.8} implies \eqref{cond:4.4}, \eqref{cond:4.5}, \eqref{cond:4.6}, thereby ensuring that Lemmas \ref{iterate convergence}, \ref{lem:Hessian Convergence}, \ref{lem:local:rate}, \ref{lem:local:rate:Hessian} naturally hold.
We strengthen the condition on $r$ from $r(p_1-p_2)>1$ (as used in \eqref{cond:4.4}, \eqref{cond:4.5}) to additionally require $r\geq 3$, which ensures that the gradient estimate $\nabla F(\bx;\xi)$ has a bounded third moment and is standard in establishing asymptotic normality guarantee \citep{Davis2024Asymptotic, Na2025Statistical}.
On the other hand, the conditions~on $\{p_1, p_2\}$ in \eqref{cond:4.8} for local convergence are weaker than those in \eqref{cond:3.10} for global convergence. The technical reason for this relaxation is that we are able to refine the bound on $\tilde{\Delta}\bx_k$ and show that it vanishes in probability in local analysis. This can be seen by comparing Lemma \ref{final fundamental lemma} with Lemma \ref{lem:local:rate}, where~the former contains the term $\alpha_k^2 / \beta_k^2$, while the latter does not.

The above theorem illustrates that the rescaled primal-dual error by the random stepsize converges in distribution to a Gaussian distribution with mean zero and covariance $\omega\cdot (W^\star)^{-1}\tOmega (W^\star)^{-1}$. To achieve optimal asymptotic rate (i.e., $\sqrt{t}$-consistency), let us set $p_1=1$. Then, Theorem \ref{thm:normality} implies~that
\begin{equation*}
\sqrt{t}\cdot(\bx_k- \tx, \blambda_k -\tlambda)\stackrel{d}{\longrightarrow}\mN\rbr{\0,\; \frac{(\zeta\iota_{1})^2}{2\zeta\iota_{1}-1}\cdot (W^\star)^{-1}\tOmega (W^\star)^{-1}}.
\end{equation*}
Thus, the minimum variance is achieved by setting $\iota_{1} \coloneqq 1/\zeta$, leading to the asymptotic covariance
\begin{equation*}
\bSigma^\star \coloneqq (W^\star)^{-1}\tOmega (W^\star)^{-1}.
\end{equation*}
On the other hand, we know from \cite{Duchi2021Asymptotic, Davis2024Asymptotic, Na2025Statistical, Du2025Online} that the \textit{minimax optimal covariance} achieved by various derivative-based methods for Problem \eqref{Intro_StoProb} is given by (recall \eqref{equ:SigmaOpt})
\begin{equation*}
\bSigma^\star_{op} \coloneqq (W^\star)^{-1}\diag\rbr{\text{Cov}(\nabla F(\tx;\xi)), \0} (W^\star)^{-1}.
\end{equation*}
The next proposition shows that the proposed derivative-free method, while more computationally efficient, is less statistically efficient than derivative-based methods in the sense that $\bSigma^\star \succeq \bSigma^\star_{op}$. Moreover, the statistical efficiency gap grows linearly with the dimension $d$, even though the computational efficiency gap also becomes more and more promising, as the proposed method requires only a \textit{dimension-independent} number of function evaluations.

\begin{proposition}\label{prop:1}
Suppose $\bDelta\sim\P_{\bDelta}$ satisfies Assumption \ref{ass:Delta}. We have $\bSigma^\star \succeq \bSigma^\star_{op}$. Furthermore, there exists a constant $\Upsilon>0$ such that
\begin{equation*}
(d-1)/\Upsilon \leq \|\bSigma^\star-\bSigma^\star_{op}\| \leq \Upsilon\cdot (d-1).
\end{equation*}

\end{proposition}

\begin{proof}
See Appendix \ref{pf:prop:1}.
\end{proof}

To conclude this section, we turn our attention to performing statistical inference in practice. In particular, to conduct hypothesis testing and construct confidence intervals or regions for $(\tx, \tlambda)$, a consistent estimator of the limiting covariance in Theorem \ref{thm:normality} is required. The next proposition provides a simple plug-in estimator for this purpose.

\begin{proposition}\label{prop:2}
Under the conditions of Theorem \ref{thm:normality} and strengthen $r\geq 4$ in \eqref{cond:grad:smo}, we define
\begin{equation*}
\bSigma_k = \tW_k^{-1}\cdot \diag\rbr{\frac{1}{k+1}\sum_{t=0}^{k}\rbr{\hnabla F(\bx_t;\xi_t)+\hnabla^T c(\bx_t)\blambda_t}\rbr{\hnabla F(\bx_t;\xi_t)+\hnabla^T c(\bx_t)\blambda_t}^T,\; \0}\cdot \tW_k^{-1}
\end{equation*}
and have $\bSigma_k\rightarrow \tSigma = (W^\star)^{-1}\tOmega (W^\star)^{-1}$ as $k\rightarrow\infty$ almost surely.
\end{proposition}

\begin{proof}
See Appendix \ref{pf:prop:2}.
\end{proof}

We mention that requiring the gradient estimate $\nabla F(\bx;\xi)$ to have a bounded fourth moment (i.e., $r \geq 4$) is standard for establishing the consistency of the plug-in covariance estimator; see \cite{Chen2020Statistical, Davis2024Asymptotic, Na2025Statistical} and references therein. With the above covariance estimator in Proposition \ref{prop:2},  we can construct the confidence interval of the quantity $(\bw_\bx, \bw_\blambda)^T(\tx, \tlambda)$ for any vector $\bw=(\bw_\bx, \bw_\blambda)$ as follows:
\begin{equation*}
P\rbr{(\bw_\bx, \bw_\blambda)^T(\tx, \tlambda) \in \sbr{(\bw_\bx, \bw_\blambda)^T(\bx_k, \blambda_k) \pm z_{1-\varphi/2}\sqrt{\baralpha_k\cdot\omega\cdot\bw^T\bSigma_k\bw} } }\rightarrow 1-\varphi\quad \text{as}\quad k\rightarrow\infty.
\end{equation*}
Here, for $\varphi\in(0, 1)$, $z_{1-\varphi/2}$ denotes the $(1-\varphi/2)$-quantile of the standard Gaussian distribution.

\section{Numerical Experiment}\label{sec:5}

In this section, we compare derivative-free methods with derivative-based methods on benchmark constrained nonlinear problems in CUTEst test set \citep{Gould2014CUTEst}. For both DF-SSQP~and~derivative-based SSQP, we consider first- and second-order variants. 
The first-order methods do not estimate~$\hnabla^2_{\boldsymbol{x}} \L_k$ in \eqref{snequ:5} and instead set it as $I$. The second-order methods estimate it either via a derivative-free approach in \eqref{snequ:12}, \eqref{snequ:4}, \eqref{snequ:11}, or obtain it directly from the CUTEst package. Note that no debiasing step~is~performed for the derivative-based methods, i.e., $\beta_k = 1$ in \eqref{snequ:3} and \eqref{snequ:5}.

For both derivative-free and derivative-based SSQP, we perform 200 independent runs for each problem under each setup and set the total number of iterations to $10^5$. For DF-SSQP, we consider the setting where any order of derivatives of both the objective and constraints are inaccessible, and we apply the SPSA approach to estimate them (see \eqref{snequ:2}, \eqref{snequ:12}--\eqref{snequ:11}). The random directions $\bDelta_k$ and $\tilde{\bDelta}_k$ have independent entries drawn from the Rademacher distribution, taking values $\pm1$ with equal~probability.~We set the prespecified stepsize, momentum weight, and discretization sequences as $\alpha_k = 1/t^{0.751}$,~$\beta_k = 1/t^{0.501}$, $b_k = \tilde{b}_k = 1/t^{0.25}$, $p = 1.5$ according to \eqref{cond:3.10} and \eqref{cond:4.8}, and designate the first one-fifth of the iterations as the burn-in {period}. For derivative-based SSQP, we use the same $\alpha_k$ and $p$. The objective values, gradients, and Hessians (when applicable) are generated by adding Gaussian noise to the true deterministic quantities. Specifically, $F(\bx_k;\xi)\sim \mN(f_k, \sigma^2)$, $\nabla F(\bx_k,\xi)\sim\mN(\nabla f_k, \sigma^2(I+\1\1^T))$, and $[\nabla^2F(\bx_k;\xi)]_{i,j} \sim\mN([\nabla^2f_k]_{i,j}, \sigma^2)$. Here, $\1$ denotes the $d$-dimensional all-ones vector. We vary the noise variance as $\sigma^2 \in \{10^{-4}, 10^{-2}, 10^{-1}, 1\}$.

\subsection{Global convergence}\label{sec:5.1}

We compare the final KKT residuals, primal-dual iterate errors, computational flops per iteration, and running times of four methods: first- and second-order DF-SSQP and first- and second-order derivative-based SSQP, denoted as \texttt{DF-Id}, \texttt{DF-Hess}, \texttt{DB-Id}, and \texttt{DB-Hess}, respectively. The results are summarized in Figure~\ref{fig:1}.

\begin{figure}[!t]
\centering
\subfigure[KKT]{\includegraphics[width=0.45\textwidth]{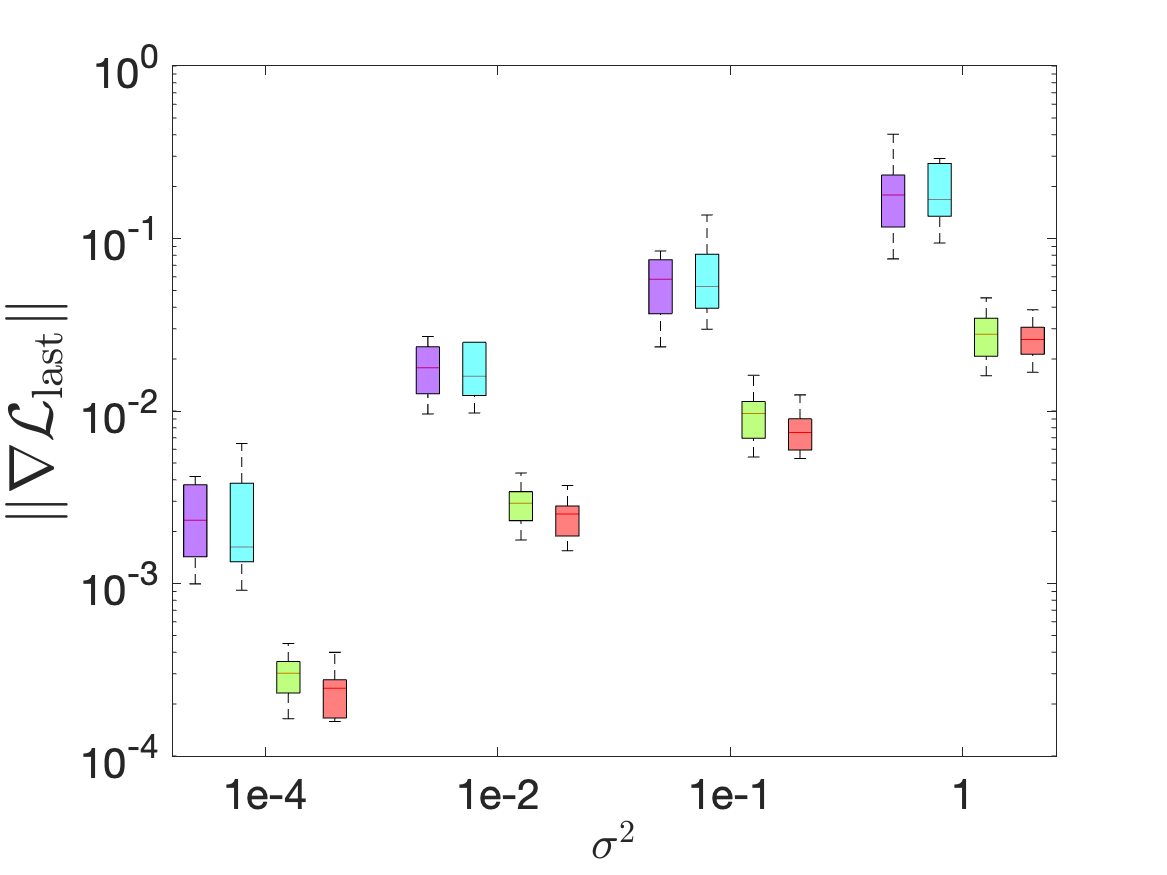}}
\subfigure[Primal-dual error]{\includegraphics[width=0.45\textwidth]{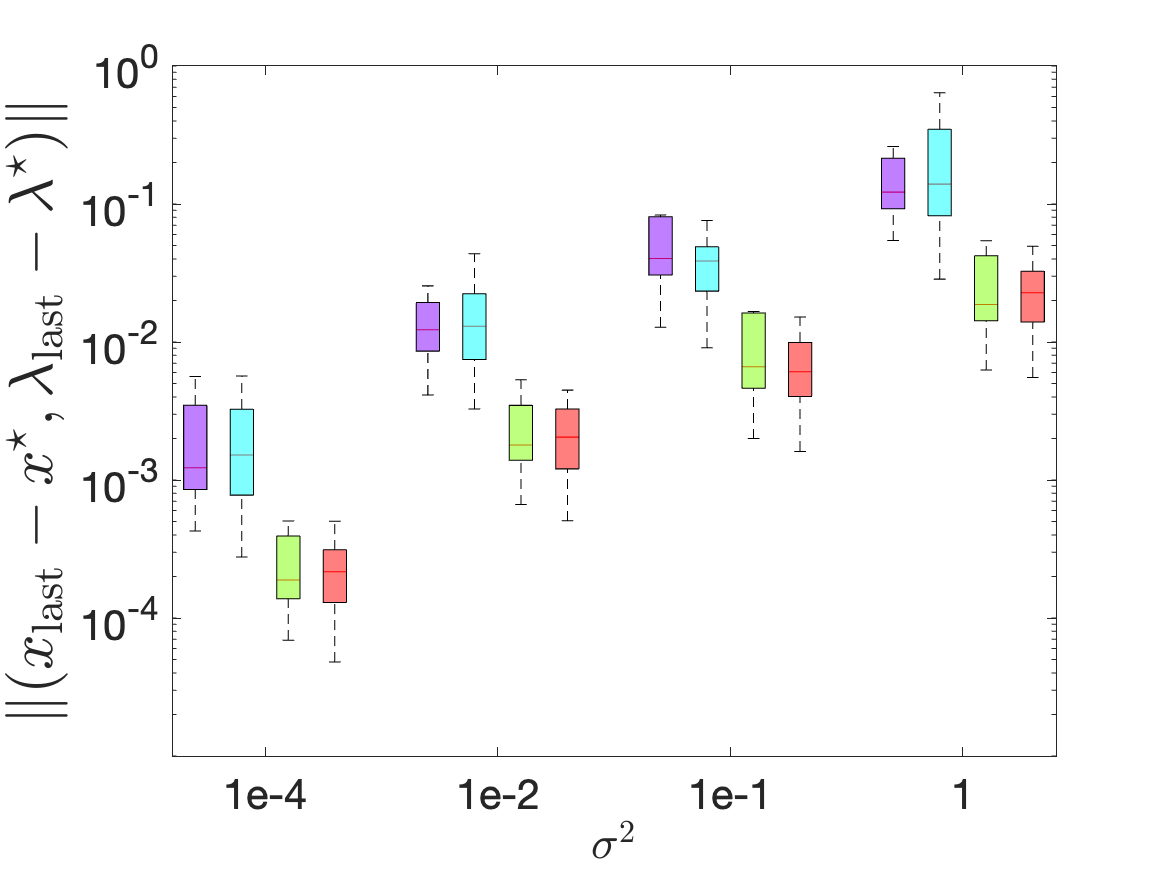}}
\vskip-0.4cm
\subfigure[Flops per iteration]{\includegraphics[width=0.45\textwidth]{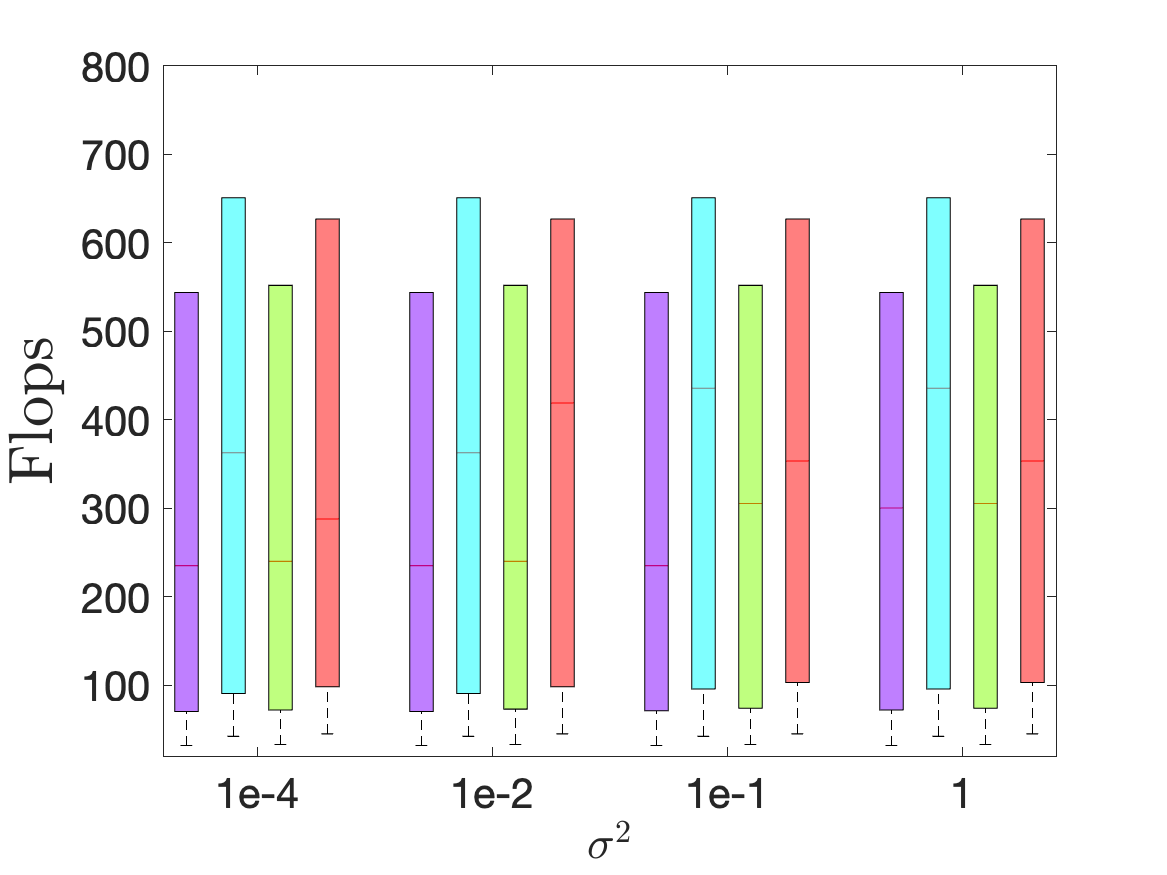}}
\subfigure[Running time]{\includegraphics[width=0.45\textwidth]{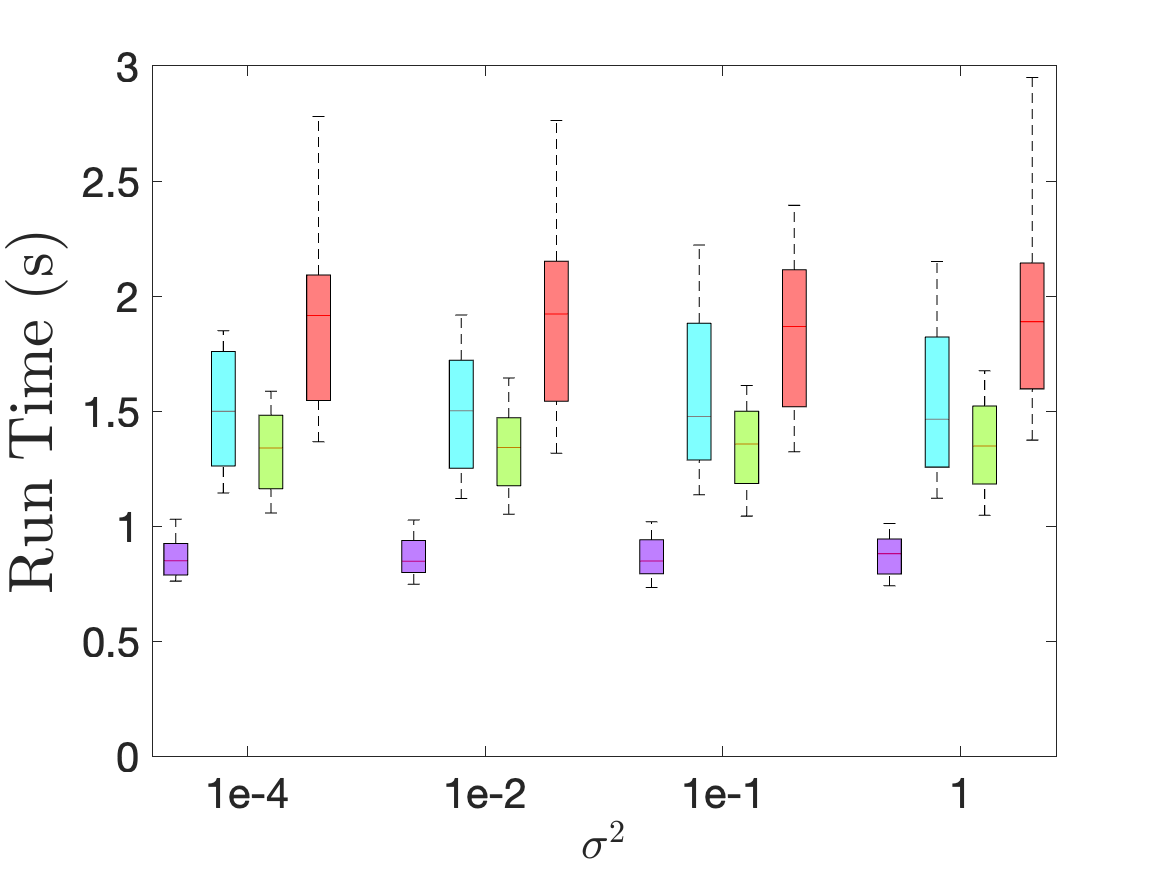}}
\includegraphics[width=0.45\textwidth]{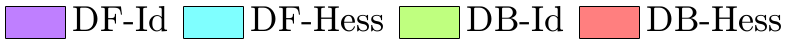}
\vspace{4pt}
\caption{Boxplots over CUTEst problems. Each panel has four different noise levels, and each noise has four different methods.}
\label{fig:1}\vskip-0.4cm
\end{figure}

Not surprisingly, there are considerable disadvantages to not having derivative information, especially in conjunction with additional random noise in the objective value estimates. Hence, we cannot~expect the performance of derivative-free methods to be as competitive as that of derivative-based~methods. From Figure~\ref{fig:1}(a)-(b), we observe that the performance of DF-SSQP degrades, exhibiting higher KKT residuals and iterate errors. This suggests that a near-optimal solution obtained by DF-SSQP is often less accurate than that produced by a derivative-based SSQP method.
On the other hand, for both~types of methods, we do not observe a significant advantage in approximating second-order information from noisy observations for facilitating global convergence; this will, however, become clearer in the local study presented in Section \ref{sec:5.2}. In terms of flops per iteration, all four methods yield comparable results, with first-order methods showing slightly lower costs. This is because all methods have to solve the Newton system \eqref{def:acc2} at each step, which is the dominant computational cost. In terms of running time, we observe that first-order methods reach stationarity faster than second-order methods, and that derivative-free methods are faster than their derivative-based counterparts.

\subsection{Local normality and inference}\label{sec:5.2}

We illustrate the local convergence behavior of DF-SSQP stated in Theorem \ref{thm:normality} by performing statistical inference on $\tx$. In particular, we estimate the limiting covariance matrix using Proposition \ref{prop:2}, and construct entrywise 95\% confidence intervals for $\tx$. We report the average iterate error, coverage rate over 200 runs, confidence interval length, and computational flops on 8 CUTEst problems~under 4 different variance levels $\sigma^2$. The results are summarized in Table \ref{tab:1}.

From the table, we observe that both first- and second-order derivative-based methods (DB-SSQP) generally achieve smaller iterate errors and shorter confidence interval lengths than their derivative-free counterparts (DF-SSQP), with comparable FLOPs, across 8 CUTEst problems and 4 noise levels. This suggests that, when available and reliable, derivative information should be used to compute the step direction. That said, in high-noise regimes ($\sigma^2\in\{0.1, 1\}$), second-order variants may fail to converge or may converge to a stationary point different from the package reference; thus, when second-order estimates are very noisy, incorporating curvature does not necessarily reduce the iterate error.

On the other hand, second-order information significantly improves coverage rate, bringing the~empirical rates closer to the nominal rate 95\%. In particular, for 6 out of 8 CUTEst problems, we observe many settings in which second-order DF- and DB-SSQP attain coverage rate much nearer 95\%, while the corresponding first-order methods exhibit noticeable over-coverage (near 100\%) or under-coverage (below 90\%).
These observations align with Theorem \ref{thm:normality}: local asymptotic normality of SSQP highlights the benefits of Hessian information; without it, the normality \eqref{equ:normality} fails to hold and the limiting covariance is only biasedly estimated, yielding asymptotically mis-calibrated confidence intervals.~\mbox{Notably},~on~problem \texttt{BT1}, DF-SSQP attains a much better coverage rate than DB-SSQP for both first- and second-order variants; and on the remaining 7 problems, DF-SSQP achieves coverage that is no worse than DB-SSQP. Taken together, the results indicate that for solution inference tasks, second-order DF-SSQP~can~be as reliable, and in some cases preferable, as second-order DB-SSQP in terms of coverage, even if DB-SSQP often delivers smaller errors and shorter interval lengths.

\include{table}

\section{Conclusion}\label{sec:6}

In this work, we proposed DF-SSQP (Algorithm \ref{Alg:DF-SSQP}), a derivative-free, fully stochastic method for solving the constrained stochastic optimization problem \eqref{Intro_StoProb}.
Our method leverages the simultaneous perturbation stochastic approximation (SPSA) technique, generalizes it to estimate both the objective gradient and the constraint Jacobian, and additionally employs an online debiasing, momentum-style strategy that properly aggregates past~gradients~(and Hessians, if local convergence is of interest) to reduce~the stochastic noise inherent in SPSA-based methods. The debiasing strategy avoids excessive memory~costs due to its simple running average scheme.~We established almost-sure global \mbox{convergence}~of~DF-SSQP by showing that the first-order (KKT) optimality conditions are asymptotically satisfied from any initialization.
Furthermore, we complemented the global analysis with local convergence guarantees: we~established the local convergence rate (in expectation) and proved that the rescaled iterates exhibit asymptotic normality. The limiting covariance matrix closely resembles the minimax optimal covariance achieved by derivative-based methods, albeit it is inflated due to the absence of derivative information.
This local result is particularly surprising and significant, not only because it illustrates the trade-off between computational efficiency and statistical efficiency, but also because DF-SSQP relies on highly correlated gradient~estimates due to the debiasing technique; unlike all existing methods that rely on conditionally independent gradient estimates.
Numerical experiments on a subset of benchmark nonlinear problems demonstrate the global and local performance of the proposed~method.

Several interesting avenues remain for future research. First, while our current analysis enables~statistical inference for the last iterate, establishing asymptotic normality for the averaged iterate remains an open problem. Second, it would be valuable to develop derivative-free SSQP algorithms that can~handle cases where the constraint Jacobians are rank-deficient. Finally, our implementation and analysis~assume exact solutions to the Newton system, which can be computationally expensive. Extending the method to allow inexact solutions to the quadratic subproblems could significantly reduce computational costs, though it remains unclear whether the global almost sure convergence and local asymptotic~normality properties of DF-SSQP would still hold under such approximations.

\section*{Acknowledgment}

The author would like to acknowledge Lvmin Wu for initial discussion of the work.

\bibliographystyle{my-plainnat}
\bibliography{ref}

\appendix
\numberwithin{equation}{section}
\makeatletter
\def\@seccntformat#1{%
	\expandafter\ifx\csname the#1\endcsname\thesection
	Appendix~\csname the#1\endcsname.\;
	\else\csname the#1\endcsname.\quad
	\fi}
\makeatother

\section{Preliminary Lemmas}

\begin{lemma}[\cite{Ruszczynski1980Feasible}, Lemma 1]\label{technical lemma:1}
Let $(\Omega,\mF,P)$ be a probability space and let $\{\mF_k\}$ be an increasing sequence of $\sigma$-algebras contained in $\mF$. Let $\{\boldsymbol \eta_{k}, \boldsymbol z_{k}\}$ be sequences of $\mF_k$-measurable $\mR^d$-valued random variables satisfying the relations
\begin{align}
&\boldsymbol{z}_{k+1}=\Pi_{Z}( (1-\rho_k)\boldsymbol z_{k}+\rho_{k}\boldsymbol \xi_{k}) ,\quad \boldsymbol z_{0}\in Z,\\
&\mE[\boldsymbol \xi_{k}\mid\mathcal{F}_{k}]=\boldsymbol \eta_{k}+\boldsymbol b_{k},
\end{align}
where $\rho_k\geq0$, the set $Z\subseteq \mathbb{R}^d$ is convex and closed, and $\Pi_{Z}(\cdot)$ is the projection onto the set $Z$. Suppose the following conditions hold:
\begin{enumerate}[label=(\alph*),topsep=2pt]
\setlength\itemsep{-0.1em}
\item all accumulation points of the sequence ${\boldsymbol \eta_k}$ belong to $Z$ almost surely;
\item there exists a constant $C$ such that $E[\|\boldsymbol \xi_k\|^2\mid\mF_k]\leq C$ for all $k\geq 0$;
\item $\sum_{k=0}^{\infty}\mE[\rho_{k}^{2}+\rho_{k}\|\boldsymbol b_{k}\|]<\infty$ and $\sum_{k=0}^{\infty}\rho_{k}=\infty$ almost surely;
\item $\|\boldsymbol \eta_{k+1}-\boldsymbol \eta_{k}\|/\rho_{k}\to 0$ almost surely.
\end{enumerate}
Then, we have $\boldsymbol z_{k}-\boldsymbol \eta_{k}\rightarrow0$ almost surely.	
	
\end{lemma}

\begin{lemma}[Adapted from {\cite[Lemma B.3]{Na2025Statistical}}]\label{technical lemma:2}
Let $\alpha_{k}=\iota_{1}(k+1)^{-p_{1}}$ and $\beta_{k}=\iota_{2}(k+1)^{-p_{2}}$ be two sequences with $\iota_{1},\iota_{2}, p_{1},p_{2}>0$. The following results hold.
\begin{enumerate}[topsep=1pt,itemsep=0em,partopsep=-2pt,parsep=1ex,label=(\alph*),beginpenalty=10000]
\item Let $\chi=0$ if $0<p_{2}<1$ and $\chi=-p_{1}/\iota_{2}$ if $p_{2}=1$. Then, as long as $\sum_{t=1}^{l}a_{t}+\chi>0$, we have
\begin{align*}
&\lim\limits_{k\to\infty}\frac{1}{\alpha_{k}}\sum_{i=0}^{k}\prod_{j=i+1}^{k}\prod_{t=1}^{l}\left(1-a_{t}\beta_{j}\right)\beta_{i}\alpha_{i}=\frac{1}{\sum_{t=1}^{l}a_{t}+\chi}, \\
& \lim\limits_{k\to\infty}\frac{1}{\alpha_{k}}\left\{\sum\limits_{i=0}^{k}\prod\limits_{j=i+1}^{k}\prod\limits_{t=1}^{l}\left(1-a_{t}\beta_{j}\right)\beta_{i}\alpha_{i}e_{i}+b\prod\limits_{j=0}^{k}\prod\limits_{t=1}^{l}\left(1-a_{t}\beta_{j}\right)\right\}=0, 
\end{align*}
where the second result holds for any constant $b$ and sequence $\{e_i\}$ such that $e_i \rightarrow 0$.
\item If $0<p_{2}<p_{1}\leq1$, then
\begin{equation*}
\lim_{k\to\infty}\frac{1}{\alpha_{k}}\sum_{i=0}^{k}\prod_{j=i+1}^{k}\left(1-\alpha_{j}\right)\left(1-\beta_{j}\right)\alpha_{i}\beta_{i}=1.
\end{equation*}
\end{enumerate}
	
\end{lemma}

\begin{lemma}\label{technical lemma:3}
Let $B\in\mR^{d\times d}$ and $A\in\mR^{m\times d}$. Suppose $AA^T\succeq \gamma_A I$, $\|B\| \leq \Upsilon_{B}$~for some constants $\gamma_A, \Upsilon_{B} > 0$, and $Z\in\mR^{d\times (d-m)}$ is a matrix whose columns are orthonormal and form the basis of $\text{Null}(A)$. Then,
\begin{equation*}
Z^TBZ\succeq \gamma_{RH} I \Longrightarrow \text{there exists } \delta = \delta(\gamma_{RH}, \gamma_A, \Upsilon_{A}) \text{ such that } B + \delta A^TA \succeq 0.5\gamma_{RH} I.
\end{equation*}	
\end{lemma}

\begin{proof}
For any $\bz\in \mR^d$, we decompose $\bz$ as
\begin{equation}\label{aaequ:2}
\bz=\bx+\by,  \quad \text{where} \quad \bx \in \text{Null}(A) \quad \text{and} \quad \by \in \text{Range}(A^T).
\end{equation}
Then, we can see that
\begin{align*}
&\bz^T(B+\delta A^TA-0.5\gamma_{RH} I)\bz\\
&\stackrel{\mathclap{\eqref{aaequ:2}}}{=}\; \bx^TB\bx+2\bx^TB\by+\by^TB\by+\delta\|A(\bx+\by)\|^2-0.5\gamma_{RH}(\|\bx\|^2+\|\by\|^2)\\
&\geq 0.5\gamma_{RH}\|\bx\|^2-2\Upsilon_{B}\|\bx\|\cdot\|\by\|-\Upsilon_{B}\|\by\|^2+\delta \gamma_A \|\by\|^2-0.5\gamma_{RH}\|\by\|^2\\
&=0.5\gamma_{RH}\rbr{\|\bx\|-\frac{2\Upsilon_{B}}{\gamma_{RH}}\|\by\|}^2+(\delta\gamma_A-\Upsilon_{B}-0.5\gamma_{RH}-\frac{2\Upsilon_{B}^2}{\gamma_{RH}})\|\by\|^2,
\end{align*}
where the inequality follows from $Z^TBZ\succeq \gamma_{RH} I$, $\|B\|\leq \Upsilon_{B}$ and $AA^T\succeq \gamma_A I$. Therefore, $B + \delta A^TA \succeq 0.5\gamma_{RH} I$ as long as $\delta\geq (\Upsilon_{B}+0.5\gamma_{RH}+2\Upsilon_{B}^2/\gamma_{RH})/\gamma_A$. 
\end{proof}

\section{Proofs of Section \ref{sec:3}}

\subsection{Proof of Lemma \ref{lemma:finite difference for gradient}}\label{appendix lemma:finite difference for gradient}

We use the objective gradient estimation as an example, while the same analysis applies to the constraint Jacobian. Our analysis is entrywise. Recall that for any vector $v$, $v^i$ denotes the $i$-th entry of $v$. For any $1\leq i\leq d$, we apply Taylor's expansion and have
\begin{align}\label{aequ:1}
& \mE[\hnabla F^i(\bx_k;\xi_k)-\nabla{f}_{k}^i\mid \mF_{k-1}] \nonumber \\
& \stackrel{\mathclap{\eqref{snequ:2}}}{=} \mE\left[\frac{F\left(\boldsymbol x_{k}+b_{k} \boldsymbol \Delta_{k};\xi_k\right)-F\left(\boldsymbol x_{k}-b_{k} \boldsymbol \Delta_{k};\xi_k\right)}{2 b_{k} \bDelta_{k}^{i}}-\nabla{f}_{k}^i\mid \mF_{k-1}\right] \nonumber\\
& = \mE\left[\frac{f\left(\boldsymbol x_{k}+b_{k} \boldsymbol \Delta_{k}\right)-f\left(\boldsymbol x_{k}-b_{k} \boldsymbol \Delta_{k}\right)}{2 b_{k} \bDelta_{k}^{i}}-\nabla{f}_{k}^i\mid \mF_{k-1}\right] \quad \text{(by Assumption \ref{ass:2-1})} \nonumber\\
& = \frac{1}{12}\mE\left[(b_k\bDelta_{k}^{i})^{-1}\sum_{i_1} \sum_{i_2} \sum_{i_3}b_k^3[\nabla^3f(\boldsymbol x_{k}^{+})+\nabla^3f(\boldsymbol x_{k}^{-})]^{i_1i_2i_3}\bDelta_{k}^{i_1} \bDelta_{k}^{i_2} \bDelta_{k}^{i_3}\mid \mF_{k-1}\right], 
\end{align}
where $\boldsymbol{x}_k^{\pm}$ are some points lying on the line segments between $\boldsymbol{x}_k$ and $\boldsymbol{x}_k \pm b_k \boldsymbol \Delta_k$, respectively, and the last equality also applies the symmetry condition on $\bDelta_k$ in Assumption \ref{ass:Delta}. By the boundedness of $\boldsymbol \Delta_k$ in Assumption \ref{ass:Delta} and boundedness of $\nabla^3f$ in Assumption \ref{ass:1-1}, we further have
\begin{equation}\label{aequ:2}
\mE[\hnabla F^i(\bx_k;\xi_k)-\nabla{f}_{k}^i\mid \mF_{k-1}] \stackrel{\eqref{aequ:1}}{=} O\left(\frac{b_k^2}{6} \sum_{i_1} \sum_{i_2} \sum_{i_3} \mathbb{E}\left|\frac{\bDelta_{k}^{i_1} \bDelta_{k}^{i_2} \bDelta_{k}^{i_3}}{\bDelta_{k}^i}\right|\right) = O(b_k^2).
\end{equation}
This completes the proof of the first part of the lemma. Now, we consider objective Hessian estimation, while noting that the same analysis applies directly to the constraint Hessian. For~any~$1\leq \ell_1, \ell_2\leq d$, we know 
\begin{equation*}
\mE\left[\frac{\delta \tilde{\nabla} F^{\ell_1}(\bx_k \pm b_k \bDelta_k;\xi_k)}{2b_k\bDelta_k^{\ell_2}}\mid \mF_{k-1}\right] \stackrel{\eqref{snequ:11}}{=}\mE\Bigg[\frac{\tilde{\nabla} F^{\ell_1}(\bx_k + b_k \bDelta_k;\xi_k)-\tilde{\nabla} F^{\ell_1}(\bx_k - b_k \bDelta_k;\xi_k)}{2b_{k}\bDelta_{k}^{\ell_2}}\mid \mF_{k-1}\Bigg].
\end{equation*}
Applying the definition \eqref{snequ:12} and following the same analysis as in \eqref{aequ:1} and \eqref{aequ:2}, we can have
\begin{equation*}
\mE[\tilde{\nabla} F^{\ell_1}(\bx_k \pm b_k \bDelta_k;\xi_k)\mid \mF_{k-1},\boldsymbol \Delta_{k}] = \nabla f^{\ell_1}(\boldsymbol x_{k}\pm b_{k}\boldsymbol \Delta_{k})+O(\tilde{b}_k^2).
\end{equation*}
Combining the above two displays and applying the Taylor's expansion, we obtain
\begin{align*}
& \mE\left[\frac{\delta \tilde{\nabla} F^{\ell_1}(\bx_k \pm b_k \bDelta_k;\xi_k)}{2b_k\bDelta_k^{\ell_2}} -\nabla^2 f^{\ell_1 \ell_2}_k\mid \mF_{k-1}\right] \\
& = \mE\left[\frac{\nabla f^{\ell_1}(\boldsymbol x_{k}+ b_{k}\boldsymbol \Delta_{k})-\nabla f^{\ell_1}(\boldsymbol x_{k} - b_{k}\boldsymbol \Delta_{k})}{2b_{k}\bDelta_{k}^{\ell_2}} -\nabla^2 f^{\ell_1 \ell_2}_k \mid \mF_{k-1}\right]+O(\tb_k^2/b_k)\\
& = \frac{1}{4}\mE\sbr{(b_k\bDelta_{k}^{\ell_2})^{-1}\sum_{i_1}\sum_{i_2}b_k^2[\nabla^2(\nabla f^{\ell_1})(\bx_k^+)+\nabla^2(\nabla f^{\ell_1})(\bx_k^-)]^{i_1i_2}\bDelta_{k}^{i_1} \bDelta_{k}^{i_2} \mid \mF_{k-1}} + O(\tb_k^2/b_k)\\
& = O(b_k + \tb_k^2/b_k),
\end{align*}
where we abuse the notation $\boldsymbol{x}_k^{\pm}$ in the second equality from \eqref{aequ:1} to let it denote some points lying on the line segments between $\boldsymbol{x}_k$ and $\boldsymbol{x}_k \pm b_k \boldsymbol \Delta_k$, and the second equality also applies Assumption~\ref{ass:Delta}.
The last equality is due to Assumptions \ref{ass:1-1} and \ref{ass:Delta}. This completes the proof.

\subsection{Proof of Lemma \ref{lemma:average almost sure}}\label{appendix lemma:average almost sure}

By Assumption \ref{ass:1-1}, let us denote $\Upsilon_{\nabla f}>0$ such that $\|\nabla f(\bx)\|\leq \Upsilon_{\nabla f}$, $\forall \bx\in \mathcal{X}$. We use $\bar{\boldsymbol{g}}_k$ as an example, while the same analysis applies to $\bar{G }_k$. We note that $\bar{\boldsymbol g}_k$ satisfies the following relations:$\quad\quad$
\begin{equation*}
\begin{aligned}
& \bar{\boldsymbol g}_k \stackrel{\eqref{snequ:3}}{=}(1-\beta_k) \bar{\boldsymbol g}_{k-1}+\beta_k\hat{\nabla}F(\bx_k;\xi_k),\\
& \mE[\hat{\nabla}F(\bx_k;\xi_k)\mid \mF_{k-1}]=\nabla{f}_k+O(b_k^2) \quad \text{(by Lemma \ref{lemma:finite difference for gradient})}.
\end{aligned}
\end{equation*}
We establish the almost sure convergence of $\bar{\boldsymbol g}_k$ by applying Lemma \ref{technical lemma:1}. We check the conditions in Lemma \ref{technical lemma:1}. Note that condition (a) in Lemma \ref{technical lemma:1} is trivially satisfied. For condition (b), we have$\quad$
\begin{align}\label{aequ:3}
\|\hat{\nabla}F(\bx_k;\xi_k)\|^2 \; & \stackrel{\mathclap{\eqref{snequ:2}}}{=} \; \nbr{ \frac{1}{2b_k}\int_{-b_k}^{b_k}\langle \nabla F(\boldsymbol x_{k}+s{\boldsymbol \Delta_ k};\xi_{k}),\boldsymbol \Delta_k \rangle\boldsymbol \Delta_k^{-1}\; ds}^{2} \nonumber \\
& \leq \|\boldsymbol \Delta_k^{-1}\|^2\cdot~\frac{1}{2b_k}\int_{-b_k}^{b_k}\|\langle \nabla F(\boldsymbol x_{k}+s{\boldsymbol \Delta_ k};\xi_{k}),\boldsymbol \Delta_k\rangle\|^{2}\; ds \nonumber\\
& \leq \|\boldsymbol \Delta_k^{-1}\|^2\|\boldsymbol \Delta_k\|^2 \frac{1}{2b_k}\int_{-b_k}^{b_k}\|\nabla F(\boldsymbol x_{k}+s{\boldsymbol \Delta_ k};\xi_{k})\|^{2}\; ds,
\end{align}
where the first inequality is due to the Jensen's inequality. For any $s\in[-b_k, b_k]$, we know from Assumption \ref{ass:2-1}\eqref{cond:grad:smo} with $r\geq 2$ in \eqref{cond:3.6} that
\begin{align}\label{aequ:B4}
& \mE[\|\nabla F(\boldsymbol x_{k}+s{\boldsymbol \Delta_ k};\xi_{k})\|^2\mid \mathcal{F}_{k-1}, \bDelta_k] \nonumber\\
& \leq  2\mE[\|\nabla F(\boldsymbol x_{k}+s{\boldsymbol \Delta_ k};\xi_{k}) - \nabla f(\bx_k+s\bDelta_k)\|^2\mid \mathcal{F}_{k-1}, \bDelta_k] + 2\|\nabla f(\bx_k+s\bDelta_k)\|^2 \nonumber\\
& \leq 2\{\mE[\|\nabla F(\boldsymbol x_{k}+s{\boldsymbol \Delta_ k};\xi_{k}) - \nabla f(\bx_k+s\bDelta_k)\|^r\mid \mathcal{F}_{k-1}, \bDelta_k]\}^{2/r} + 2\Upsilon_{\nabla f}^2 \nonumber\\
& \leq 2(\Upsilon_{m}^{2/r} + \Upsilon_{\nabla f}^2).
\end{align}
Combining \eqref{aequ:3} and \eqref{aequ:B4}, and applying Assumption \ref{ass:Delta}, we obtain
\begin{equation}\label{aequ:B5}
\mE[\|\hat{\nabla}F(\bx_k;\xi_k)\|^2 \mid \mathcal{F}_{k-1}] \leq 2d^2\kappa_{\bDelta_2}^2\kappa_{\bDelta_1}^{-2}(\Upsilon_{m}^{2/r} + \Upsilon_{\nabla f}^2),
\end{equation}
which verifies condition (b). Condition (c) is immediately satisfied under the conditions $p_2\in(0.5,1]$~and $p_2+2p_3>1$ in \eqref{cond:3.6} of the lemma. For condition (d), we note for the Lipschitz constant $\kappa_{\nabla f}>0$ that
\begin{align}\label{aequ:4}
\|\nabla f_{k+1}-\nabla f_k\| & \leq \kappa_{\nabla f}\|\boldsymbol x_{k+1}-\boldsymbol x_k\| \quad \text{(Lipschitz property)} \nonumber \\
&= \kappa_{\nabla f}\bar{\alpha}_k\|\tilde{\Delta}\boldsymbol x_k\| \stackrel{\eqref{def:acc2},\eqref{snequ:7}}{\leq} \kappa_{\nabla f}\rbr{\frac{\nu_{-1}\alpha_k}{\kappa_{\nabla c}} + \psi \alpha_k^p}\| \tilde{W}_k^{-1}\|(\|\barg_k\|+\|c_k\|).
\end{align}
By Assumption \ref{ass:1-1} and \citep[Lemma 1]{Na2022adaptive}, there exists a constant $\Upsilon_K >0$ such~that~$\|\tilde{W}_k^{-1}\| \leq \Upsilon_K$, $\forall k\geq 0$, and also $\|c_k\|\leq \kappa_{c}$. Furthermore, we follow the same analysis as in \eqref{aequ:3}, \eqref{aequ:B4}, \eqref{aequ:B5},~and obtain
\begin{align}\label{aequ:B7}
& \mE[\|\hat{\nabla}F(\bx_k;\xi_k)\|^r\mid \mF_{k-1}] \leq \mE\sbr{\|\bDelta_k^{-1}\|^r\|\bDelta_k\|^r\frac{1}{2b_k}\int_{-b_k}^{b_k}\|\nabla F(\bx_k+s\bDelta_k;\xi_k)\|^r ds \mid \mF_{k-1}} \nonumber\\
&\leq \mE\sbr{\|\bDelta_k^{-1}\|^r\|\bDelta_k\|^r\frac{2^{r-1}}{2b_k}\int_{-b_k}^{b_k} \rbr{\|\nabla F(\bx_k+s\bDelta_k;\xi_k) - \nabla f(\bx_k+s\bDelta_k)\|^r+\|\nabla f(\bx_k+s\bDelta_k)\|^r}ds\mid \mF_{k-1}} \nonumber\\
& \leq 2^{r-1}d^r\kappa_{\bDelta_2}^r\kappa_{\bDelta_1}^{-r}(\Upsilon_{m}+\Upsilon_{\nabla f}^r).
\end{align}
Thus, let us define
\begin{equation}\label{nequ:10}
\Upsilon_{\barg} \coloneqq \max\{\|\barg_{-1}\|^r,  2^{r-1}d^r\kappa_{\bDelta_2}^r\kappa_{\bDelta_1}^{-r}(\Upsilon_{m}+\Upsilon_{\nabla f}^r)\}.
\end{equation}
Then, we know $\mE[\|\barg_{-1}\|^r] = \|\barg_{-1}\|^r \leq \Upsilon_{\barg}$. For any $k\geq 0$, suppose $\mE[\|\barg_{k-1}\|^r]\leq \Upsilon_{\barg}$, then 
\begin{multline}\label{aequ:B8}
\mE[\|\barg_k\|^r] \leq \mE[((1-\beta_k)\|\barg_{k-1}\| + \beta_k\|\hat{\nabla}F(\bx_k;\xi_k)\|)^r]\\
\leq (1-\beta_k)\mE[\|\barg_{k-1}\|^r] + \beta_k\mE[\|\hat{\nabla}F(\bx_k;\xi_k)\|^r] \stackrel{\eqref{aequ:B7}}{\leq} \Upsilon_{\barg}.
\end{multline}
This shows $\mE[\|\barg_k\|^r]\leq \Upsilon_{\barg}$ for any $k\geq 0$. Combining the above display with \eqref{aequ:4}, and noting that $p\geq 1$, we obtain
\begin{align}\label{aequ:B9}
\mE\sbr{\sum_{k=0}^{\infty}\frac{\|\nabla f_{k+1}-\nabla f_k\|^r}{\beta_k^r}} & = \sum_{k=0}^{\infty}\frac{\mE[\|\nabla f_{k+1}-\nabla f_k\|^r]}{\beta_k^r} \nonumber\\
& \hskip-1cm \leq \sum_{k=0}^{\infty}\frac{\kappa_{\nabla f}^r(\frac{\nu_{-1}\alpha_k}{\kappa_{\nabla c}} + \psi \alpha_k^p)^r\Upsilon_{K}^r2^{r-1}(\mE[\|\barg_k\|^r]+\mE[\|c_k\|^r])}{\beta_k^r} \nonumber\\
& \hskip-1cm \leq \sum_{k=0}^{\infty}\frac{\kappa_{\nabla f}^r(\frac{\nu_{-1}}{\kappa_{\nabla c}} + \psi \alpha_k^{p-1})^r\Upsilon_{K}^r2^{r-1}(\Upsilon_{\barg}+\kappa_{c}^r)\alpha_k^r}{\beta_k^r} = \sum_{k=0}^{\infty}O\rbr{\frac{\alpha_k^r}{\beta_k^r}}<\infty,
\end{align}
where the first equality is due to Tonelli's theorem and the last inequality is due to $r(p_1-p_2)>1$ in \eqref{cond:3.6}. The above result immediately implies $\|\nabla f_{k+1} - \nabla f_k\|/\beta_k \rightarrow 0$ almost surely, which verifies~condition (d). By Lemma \ref{technical lemma:1}, we have $\bar{\boldsymbol g}_{k} - \nabla f_k \rightarrow \boldsymbol 0$ almost surely. 
The same analysis applies to $\barG_k$, and we complete the proof for the first part of the lemma. For the second part, we know for each sample~path, there exists $ K^{\star}_G>0 $ such that for any $k\geq K^\star_G$,
\begin{equation*}
\|\bar{G }_k \bar{G }_k^T  - G_kG_k^T\| \leq \min\{\kappa_{2,\tG} - \kappa_{2,G}, \kappa_{1,G} - \kappa_{1,\tG}\}.
\end{equation*}
By Weyl's inequality \cite[Theorem 4.3.1]{Horn1985Matrix}, we know $\kappa_{1,\tG}\cdot I \preceq \bar{G }_k \bar{G }_k^T \preceq \kappa_{2,\tG}\cdot I$. Since the modification $\delta_{k}^G $ is introduced to modify $\bar{G }_k$ to satisfy this condition, we know there is no need to apply $\delta_{k}^G $ for all $ k \geq K^{\star}_G$. This completes the proof.

\subsection{Proof of Lemma \ref{final fundamental lemma}}\label{appendix lemma:final fundamental lemma}

We use $\barg_k$ as an example, while the same analysis applies to $\barG_k$. We decompose $\barg_k-\nabla f_k$ as follows:
\begin{align}\label{aequ:5}
& \barg_k-\nabla f_k  \stackrel{\mathclap{\eqref{snequ:3}}}{=} \beta_k(\hnabla F(\bx_k;\xi_k)-\nabla f_k)+(1-\beta_k)\left(\barg_{k-1}-\nabla f_{k-1}\right)+(1-\beta_k)\left(\nabla f_{k-1}-\nabla f_k\right) \nonumber\\
& \stackrel{\mathclap{\eqref{snequ:3}}}{=} \beta_{k}(\hnabla F(\bx_k;\xi_k)-\nabla f_k)+(1-\beta_{k})\{\beta_{k-1}(\hnabla F(\bx_{k-1};\xi_{k-1})-\nabla f_{k-1}) \nonumber\\
&\quad +(1-\beta_{k-1})\left(\barg_{k-2}-\nabla f_{k-2}\right) +(1-\beta_{k-1}) (\nabla f_{k-2}-\nabla f_{k-1})\}+(1-\beta_{k}) (\nabla f_{k-1}-\nabla f_k) \nonumber\\
& = \cdots \nonumber\\
& = \sum_{i=0}^{k}\prod_{j=i+1}^{k}(1-\beta_{j})\beta_{i}(\hnabla F(\bx_i;\xi_i)-\nabla f_i)+\sum_{i=0}^{k}\prod_{j=i}^{k}(1-\beta_{j})(\nabla f_{i-1}-\nabla f_i) \nonumber\\
& = \sum_{i=0}^{k}\prod_{j=i+1}^{k}(1-\beta_{j})\beta_{i}(\hnabla F(\bx_i;\xi_i)-\mE[\hnabla F(\bx_i;\xi_i)\mid \mF_{i-1}]) \nonumber\\
& \quad +\sum_{i=0}^{k}\prod_{j=i+1}^{k}(1-\beta_{j})\beta_{i}(\mE[\hnabla F(\bx_i;\xi_i)\mid \mF_{i-1}]-\nabla f_i)\vspace{10 pt}+\sum_{i=0}^{k}\prod_{j=i}^{k}(1-\beta_{j})(\nabla f_{i-1}-\nabla f_i),
\end{align}
where we denote $\nabla f_{-1} = \barg_{-1}$ in the last two equalities for clarity. We now proceed to derive bounds~for each term in \eqref{aequ:5}. In particular, using the martingale difference property, we have
\begin{align*}
&\mE\left[\left\|\sum_{i=0}^{k}\prod_{j=i+1}^{k}(1-\beta_{j})\beta_{i}(\hnabla F(\bx_i;\xi_i)-\mE[\hnabla F(\bx_i;\xi_i)\mid \mF_{i-1}])\right\|^2\right]\\
&=\;\sum_{i=0}^{k}\left(\prod_{j=i+1}^{k}(1-\beta_{j})\right)^2\beta_{i}^2\mE\left[\left\|\hnabla F(\bx_i;\xi_i)-\mE[\hnabla F(\bx_i;\xi_i)\mid \mF_{i-1}]\right\|^2\right]\\
& \stackrel{\mathclap{\eqref{aequ:B5}}}{=}\; O\left(\sum_{i=0}^{k}\left(\prod_{j=i+1}^{k}(1-\beta_{j})\right)^2\beta_{i}^2 \right) = O(\beta_k)\quad \text{(by Lemma \ref{technical lemma:2})},
\end{align*}   
where the last inequality holds since if $p_2=1$, we have $2-1/\iota_2>0\Leftrightarrow \iota_2>0.5$ as in \eqref{cond:3.7}. For~the~second term in \eqref{aequ:5}, we have
\begin{align*}
& \left\|\sum_{i=0}^{k}\prod_{j=i+1}^{k}(1-\beta_{j})\beta_{i}(\mE[\hnabla F(\bx_i;\xi_i)\mid \mF_{i-1}]-\nabla f_i)\right\| \\
&\leq\sum_{i=0}^{k}\prod_{j=i+1}^{k}(1-\beta_{j})\beta_{i}\left\|\mE[\hnabla F(\bx_i;\xi_i)\mid \mF_{i-1}]-\nabla f_i\right\| = O\left(\sum_{i=0}^{k}\prod_{j=i+1}^{k}(1-\beta_{j})\beta_{i}b_i^2\right)\quad \text{(by Lemma \ref{lemma:finite difference for gradient})}\\
&=O(b_k^2)\quad \text{(by Lemma \ref{technical lemma:2})},
\end{align*}
where the last inequality holds since if $p_2=1$, we have $1-2p_3/\iota_2>0\Leftrightarrow p_3<0.5\iota_2$ as in \eqref{cond:3.7}. For~the third term in \eqref{aequ:5}, we have
\begin{align*}
& \mE\sbr{\nbr{\sum_{i=0}^{k}\prod_{j=i}^{k}(1-\beta_{j})(\nabla f_{i-1}-\nabla f_i)}^2} \leq \rbr{\sum_{i=0}^{k}\prod_{j=i}^{k}(1-\beta_j)\sqrt{\mE[\|\nabla f_{i-1}-\nabla f_i\|^2]}}^2\\
& \leq O\left(\cbr{\sum_{i=0}^{k}\prod_{j=i}^{k}(1-\beta_{j})\alpha_{i-1}}^2\right)\quad \text{(by the same analysis of \eqref{aequ:4}, \eqref{aequ:B7}, \eqref{aequ:B8}, \eqref{aequ:B9})}\\
& = O\left(\frac{\alpha_k^2}{\beta_k^2}\right)\quad \text{(by Lemma \ref{technical lemma:2})},
\end{align*}
where we set $\alpha_{-1} = \|\barg_{-1} - \nabla f_0\|$ in the last inequality and the last equality holds since $p_1>p_2$, and if $p_2=1$, $1 - (p_1-p_2)/\iota_2>0\Leftrightarrow p_1<p_2+\iota_{2}$ as in \eqref{cond:3.7}. Combining the above three displays~with~\eqref{aequ:5}, we know $\mE[\|\barg_k-\nabla f_k\|^2]=O(\beta_k+b_k^4+\alpha_k^2/\beta_k^2)$. The same analysis applies to $\barG_k$, thereby completing the proof.

\subsection{Proof of Lemma \ref{thm:convergence1}}\label{appendix thm:convergence1}

Let $k\geq 0$. For the result (a), we note that
\begin{equation*}
\tilde{G}_k \widetilde{\Delta} \boldsymbol{x}_{k} \stackrel{\eqref{decomposition}}{=} \tilde{G}_k(\boldsymbol u_k + \boldsymbol v_k) \stackrel{\eqref{decomposition}}{=} \tilde{G}_k \boldsymbol v_k \stackrel{\eqref{def:acc2}}{=} -c_k.
\end{equation*}
We apply Assumption \ref{ass:1-1} and have $\boldsymbol v_k = -\tilde{G}_k^T(\tilde{G}_k \tilde{G}_k^T)^{-1} c_k$. Thus, we obtain
\begin{equation}\label{aequ:6}
\|\boldsymbol v_{k}\| \leq \| \tilde{G}_k^{T} (\tilde{G}_k \tilde{G}_k^{T})^{-1} \| \|c_k\|\leq \frac{1}{\sqrt{\kappa_{1,\tG}}}\|c_k\| \quad \text{ and }\quad
\|\boldsymbol v_{k}\|^{2} \leq \frac{1}{{\kappa_{1,\tG}}}\|c_k\|^2\leq  \frac{\kappa_c}{{\kappa_{1,\tG}}}\|c_k\|.
\end{equation}
Thus, (a) holds with $\kappa_v=\max\{1/\sqrt{\kappa_{1,\tG}},\kappa_c/{\kappa_{1,\tG}}\}$. For the result (b), we note that
\begin{align}\label{aequ:7}
\tilde{\Delta} \boldsymbol{x}_k^T\tB_k\tilde{\Delta} \boldsymbol{x}_k & \stackrel{\mathclap{\eqref{decomposition}}}{=}\boldsymbol u_k^T\tB_k\boldsymbol u_k+ 2\boldsymbol u_k^T\tB_k\boldsymbol v_k+\boldsymbol v_k^T\tB_k\boldsymbol v_k \nonumber \\
&\geq \kappa_{1,\tB} \|\boldsymbol u_k\|^2-2\kappa_{2,\tB}\|\boldsymbol u_k\|\|\boldsymbol v_k\|-\kappa_{2,\tB}\|\boldsymbol v_k\|^2\quad \text{(by Assumption \ref{ass:1-1})} \nonumber \\
&\geq\rbr{\kappa_{1,\tB}-\frac{2\kappa_{2,\tB}}{\sqrt{\kappa_u}}-\frac{\kappa_{2,\tB}}{\kappa_u}}\|\boldsymbol u_k\|^2\quad\text{(by $\|\boldsymbol u_k\|^2 \geq \kappa _{u} \|\boldsymbol v_k\|^2$)}.
\end{align}
Thus, as long as $\kappa_{u}$ is large enough such that ${2\kappa_{2,\tB}}/{\sqrt{\kappa_u}}+{\kappa_{2,\tB}}/{\kappa_u}\leq {\kappa_{1,\tB}/}{2}$, the result (b) holds. For the result (c), we note that
\begin{equation*}
\Delta q(\tDelta\bx_k; \tau_k, \boldsymbol x_k,\bar{\boldsymbol g}_k,\tB_k) \stackrel{\eqref{reduction}}{\geq}  \frac{1}{2}\tau_k  \max\{\tilde{\Delta} \boldsymbol{x}_k^T\tB_k\tilde{\Delta} \boldsymbol{x}_k, 0\}+\sigma\|c_k\| .
\end{equation*}
If $\|\boldsymbol u_k\|^2 \geq \kappa _{u} \|\boldsymbol v_k\|^2$, we have
\begin{align*}
\Delta q(\tDelta\bx_k; \tau_k, \boldsymbol x_k,\bar{\boldsymbol g}_k,\tB_k) & \stackrel{\mathclap{\eqref{aequ:7}}}{\geq}\;  \frac{1}{4}\tau_k\kappa_{1,\tB}\|\boldsymbol u_k\|^2+\sigma\|c_k\|\\
& \geq\frac{\tau_k\kappa_u\kappa_{1,\tB}}{4(1+\kappa_u)}(\|\boldsymbol u_k\|^2+\|\boldsymbol v_k\|^2)+\sigma\|c_k\|\quad\text{(by $\|\boldsymbol u_k\|^2 \geq \kappa _{u} \|\boldsymbol v_k\|^2$)}\\
& \stackrel{\mathclap{\eqref{decomposition}}}{=} \; \frac{\tau_k\kappa_u\kappa_{1,\tB}}{4(1+\kappa_u)}\|\tDelta \bx_k\|^2+\sigma\|c_k\|.
\end{align*} 
Otherwise, we have 
\begin{align*}
\Delta q(\tDelta\bx_k; \tau_k, \boldsymbol x_k,\bar{\boldsymbol g}_k,\tB_k) & \geq\sigma\|c_k\|\\
& \stackrel{\mathclap{\eqref{aequ:6}}}{\geq} \frac{\sigma}{2\kappa_v(1+\kappa_u)}\|\tDelta \bx_k\|^2+\frac{\sigma}{2}\|c_k\|\quad\text{(by $\|\boldsymbol u_k\|^2 \leq \kappa _{u} \|\boldsymbol v_k\|^2$)}.
\end{align*} 
Combining the above two displays, we know (c) holds for $\kappa_q=\min\{{\kappa_u\kappa_{1,\tB}}/{4(1+\kappa_u)},\sigma/2\tau_{-1},{\sigma}/\{2\kappa_v\tau_{-1}(1+\kappa_u)\}\}$. This completes the proof.

\subsection{Proof of Lemma \ref{parameter stabilize}}\label{appendix parameter stabilize}

From the update of \eqref{def:acc4}, we know $ \tau_k < \tau_{k-1} $ if and only if  both $\bar{\boldsymbol g}_k^T\tilde{\Delta} \boldsymbol{x}_k+ \max\{\tilde{\Delta} \boldsymbol{x}_k^T\tB_k\tilde{\Delta} \boldsymbol{x}_k, 0\} > 0$~and $\tau_{k-1} ( \bar{\boldsymbol g}_k^T\tilde{\Delta} \boldsymbol{x}_k+ \max\{\tilde{\Delta} \boldsymbol{x}_k^T\tB_k\tilde{\Delta} \boldsymbol{x}_k, 0\} ) > (1-\sigma)\|c_k\|$. From \eqref{def:acc2}, we know
\begin{equation*}
\tB_k \tilde{\Delta} \boldsymbol{x}_k +  \tG_k^T \tilde{\Delta} \boldsymbol{\lambda}_k = -\bar{\boldsymbol g}_k - \tG_k^T \boldsymbol{\lambda}_k.
\end{equation*}
Multiplying both sides by $ \boldsymbol u_k^T $, we obtain  
\begin{equation}\label{aequ:8}
\boldsymbol u_k^T \tB_k (\boldsymbol u_k + \boldsymbol v_k) \stackrel{{\eqref{decomposition}}}{=} -\bar{\boldsymbol g}_k^T \boldsymbol u_k.
\end{equation} 
If $\tilde{\Delta} \boldsymbol{x}_k^T\tB_k\tilde{\Delta} \boldsymbol{x}_k\geq 0$, we have for some $\kappa_{\tau,1}>0$ that
\begin{align*}
\bar{\boldsymbol g}_k^T\tilde{\Delta} \boldsymbol{x}_k+ \max\{\tilde{\Delta} \boldsymbol{x}_k^T\tB_k\tilde{\Delta} \boldsymbol{x}_k, 0\} & \;\stackrel{\mathclap{\eqref{decomposition}}}{=}\; \bar{\boldsymbol g}_k^T(\boldsymbol u_k + \boldsymbol v_k) + (\boldsymbol u_k + \boldsymbol v_k)^T \tB_k (\boldsymbol u_k + \boldsymbol v_k) \\
&\; \stackrel{\mathclap{\eqref{aequ:8}}}{=}\; \bar{\boldsymbol g}_k^T \boldsymbol v_k + \boldsymbol v_k^T \tB_k \boldsymbol u_k + \boldsymbol v_k^T \tB_k \boldsymbol v_k \\
&\; \leq\; (\|\bar{\boldsymbol g}_k\| + \kappa_{2,\tB}\|\boldsymbol u_k\|)\|\boldsymbol v_k\| + \kappa_{2,\tB}\|\boldsymbol v_k\|^2 \quad \text{(by Assumption \ref{ass:1-1})}\\
&\; \leq \; \kappa_{\tau,1}\|c_k\|,
\end{align*}
where the existence of $\kappa_{\tau,1}$ in the last inequality is due to the boundedness of $\barg_k$ (similar to \eqref{aequ:B8}), the boundedness of $\tilde{\Delta} \boldsymbol{x}_k$ (hence $\bu_k$) in \eqref{aequ:4}, and Lemma \ref{thm:convergence1}(a). Otherwise $\tilde{\Delta} \boldsymbol{x}_k^T\tB_k\tilde{\Delta} \boldsymbol{x}_k< 0$, we~have~for some $\kappa_{\tau,2}>0$ that
\begin{align*}
\bar{\boldsymbol g}_k^T\tilde{\Delta} \boldsymbol{x}_k+ \max\{\tilde{\Delta} \boldsymbol{x}_k^T\tB_k\tilde{\Delta} \boldsymbol{x}_k, 0\}
&\stackrel{\mathclap{\eqref{decomposition}}}{=}\; \bar{\boldsymbol g}_k^T(\boldsymbol u_k + \boldsymbol v_k)\\
&\stackrel{\mathclap{\eqref{aequ:8}}}{=}\; \bar{\boldsymbol g}_k^T\boldsymbol v_k-\boldsymbol u_k^T \tB_k\boldsymbol u_k - \boldsymbol u_k^T \tB_k\boldsymbol v_k\\
& \leq\; (\|\bar{\boldsymbol g}_k\| + \kappa_{2,\tB}\|\boldsymbol u_k\|)\|\boldsymbol v_k\| \quad \text{(by Assumption \ref{ass:1-1})}\\
&\leq\; \kappa_{\tau,2}\|c_k\|,
\end{align*}
where the existence of $ \kappa_{\tau,2}$ in the last inequality follows from the same reasoning as $\kappa_{\tau,1}$. Combining~the above two displays, we know that, to have $\tau_k < \tau_{k-1}$, we must have $ \tau_{k-1} > (1-\sigma)/\max\{\kappa_{\tau,1},\kappa_{\tau,2}\} $. This, combined with the fact that Algorithm \ref{Alg:DF-SSQP} decreases $\tau_k$ by at least a constant factor whenever it is reduced, implies the existence of a (potentially random) $K^\star_{\tau}>0$ such that $\tau_k = \tau_{K^\star_{\tau}} \geq \tilde{\tau}=(1-\sigma)(1-\epsilon)/\max\{\kappa_{\tau,1},\kappa_{\tau,2}\}$ for all $ k\geq K^\star_{\tau}$. We now proceed to prove the stabilization of $\nu_k$. By Lemma \ref{thm:convergence1}(c) and the lower bound of $\tau_k$ demonstrated above, we have
\begin{equation*}
\nu_k^{\text{trial}}\stackrel{{\eqref{def:acc6}}}{=}\frac{\Delta q(\tDelta\bx_k; \tau_k, \boldsymbol x_k,\bar{\boldsymbol g}_k,\tB_k)}{\|\tilde{\Delta} \boldsymbol{x}_k\|^{2}}\geq\frac{\kappa _{q }\tau_k(\|\tDelta \bx_k\|^2+\|c_k\|)}{\|\tilde{\Delta} \boldsymbol{x}_k\|^{2}}\geq\kappa _{q}\tau_k\geq\kappa _{q}\tilde{\tau}.
\end{equation*}
This, combined with the fact that Algorithm \ref{Alg:DF-SSQP} decreases $\nu_k$ by at least a constant factor whenever it is reduced, implies the existence of a (potentially random) $K^\star_{\nu}>0$ such that $\nu_k$ stabilizes as $\nu_k = \nu_{K^\star_{\nu}} \geq \tilde{\nu}=(1-\epsilon)\kappa _{q}\tilde{\tau}$ for all $ k\geq K^\star_{\nu}$. Letting $K^\star_{\tau\nu}=\max\{K^\star_{\tau},K^\star_{\nu}\}$ completes the proof.

\subsection{Proof of Lemma \ref{convergence2}}\label{appendix convergence2}

By Assumption \ref{ass:3-1} and noting that $p\geq 1$ in \eqref{snequ:7}, we know there exists a (deterministic) $K^\star_1>0$ such that ${\nu_{-1}\alpha_k}/{\kappa_{\nabla c}} + \psi \alpha_k^p\leq 1$ for all $k\geq K_1^\star$. We further apply Lemmas \ref{lemma:average almost sure} and \ref{parameter stabilize}, and know that there exist (potentially random) $K^{\star}_G, K^\star_{\tau\nu}<\infty$ such that $\tilde{G}_k = \bar{G }_k $, $ \tau_k =\tau_{K^\star_{\tau\nu}}$, and $\nu_k=\nu_{K^\star_{\tau\nu}}$ for all $k\geq \max\{K^{\star}_G, K^\star_{\tau\nu}\}$.
To proceed, we first validate the well-definedness of $(\Delta \boldsymbol x_k,\Delta\boldsymbol \lambda_k)$. By Lemma \ref{technical lemma:3} and Assumption \ref{ass:1-1}, we know there exists $\delta = \delta(\kappa_{1,\tG}, \kappa_{1,\tB}, \kappa_{2,\tB})$ such that $\tB_k+\delta\bar G_k^T\bar G_k \succeq 0.5\kappa_{1,\tB}I$. Since we have from Lemma \ref{lemma:average almost sure} that $\barG_k- G_k\rightarrow \0$ as $k\rightarrow\infty$ almost surely, there exists (potentially random) $K^\star_2<\infty$ such that $\tB_k+\delta G_k^TG_k\succeq 0.25\kappa_{1,\tB}I$ for all $k\geq K^\star_2$, which implies $\tB_k\succeq 0.25\kappa_{1,\tB}I~ \text{in Null}(G_k)$. This result, combined with $\kappa_{1,G} I \preceq G_k G_k^T\preceq \kappa_{2,G} I$ in Assumption \ref{ass:1-1}, implies that $(\Delta \bx_k,\Delta\blambda_k)$ is well-defined. Furthermore, following the same analysis as in \eqref{aequ:4}, we have
\begin{equation}\label{aaequ:4}
\|W_k^{-1}\|\coloneqq\left\|\left(\begin{array}{cc}\tB_k & G_k^{T} \\ G_k & \boldsymbol 0\end{array}\right)^{-1}\right\|\leq \Upsilon_K \quad \text{ and }\quad \|\Delta \bx_k\|\leq \Upsilon_K(\Upsilon_{\nabla f}+\kappa_c),
\end{equation}
where we abuse the notation $\Upsilon_K$ in the analysis \eqref{aequ:4} to denote a common upper bound for $\tW_k^{-1}$ and $W_k^{-1}$, and $\Upsilon_{\nabla f}$ denotes the upper bound of $\nabla f$ in the analysis of \eqref{aequ:B4}. We now proceed to establish the convergence guarantee for $k\geq K^\star\coloneqq\max\{K^\star_1, K^\star_2, K^{\star}_G, K^\star_{\tau\nu}\}$. We have
\begin{align*}
& \phi_{\tau_{K^\star_{\tau\nu}}}(\boldsymbol x_{k}+\bar{\alpha}_k\tilde{\Delta}\boldsymbol x_{k})-\phi_{\tau_{K^\star_{\tau\nu}}}(\boldsymbol x_{k})\\
& = {\tau_{K^\star_{\tau\nu}}} f(\boldsymbol x_{k}+\bar{\alpha}_k\tilde{\Delta}\boldsymbol x_{k})+\|c(\boldsymbol x_{k}+\bar{\alpha}_k\tilde{\Delta}\boldsymbol x_{k})\|-{\tau_{K^\star_{\tau\nu}}} f(\boldsymbol x_{k})-\|c(\boldsymbol x_{k})\|\\
& \leq\bar\alpha_{k}{\tau_{K^\star_{\tau\nu}}}\nabla f_{k}^{T}\tilde{\Delta}\boldsymbol x_{k}+\|c_k+\bar{\alpha}_{k}G_k\tilde{\Delta}\boldsymbol x_{k}\| - \|c_k\| + \frac{{\tau_{K^\star_{\tau\nu}}}\kappa_{\nabla f}+\kappa_{\nabla c}}{2}\bar\alpha_{k}^{2}\|\tilde{\Delta}\boldsymbol x_{k}\|^{2} \quad \text{(Lipschitz property)}\\
& \leq\bar\alpha_k{\tau_{K^\star_{\tau\nu}}}\nabla f_{k}^{T}\tilde{\Delta}\boldsymbol x_{k}+\|c_k+\bar\alpha_k\bar{G}_{k}\tilde{\Delta}\boldsymbol x_{k}\|+\bar{\alpha}_{k}\|G_k-\bar{G}_{k}\|\|\tilde{\Delta}\boldsymbol x_{k}\| - \|c_k\|+\frac{{\tau_{K^\star_{\tau\nu}}}\kappa_{\nabla f}+\kappa_{\nabla c}}{2}\bar\alpha_k^2\|\tilde{\Delta}\boldsymbol x_{k}\|^{2}\\
& \stackrel{\mathclap{\eqref{def:acc2}}}{=}\bar\alpha_{k}{\tau_{K^\star_{\tau\nu}}}\nabla f_k^{T}\tilde{\Delta}\boldsymbol x_{k}+|1-\bar\alpha_{k}|\|c_k\| - \|c_k\| + \bar\alpha_{k}\|G_k-\bar{G}_{k}\| \|\tilde{\Delta}\boldsymbol x_{k}\| + \frac{{\tau_{K^\star_{\tau\nu}}}\kappa_{\nabla f}+\kappa_{\nabla c}}{2}\bar\alpha_{k}^{2}\|\tilde{\Delta}\boldsymbol x_{k}\|^{2}\\
& =\bar\alpha_{k}{\tau_{K^\star_{\tau\nu}}}\nabla f_k^{T}\tilde{\Delta}\boldsymbol x_{k}-\bar\alpha_{k}\|c_k\|+\bar\alpha_{k}\|G_k-\bar{G}_{k}\| \|\tilde{\Delta}\boldsymbol x_{k}\| + \frac{{\tau_{K^\star_{\tau\nu}}}\kappa_{\nabla f}+\kappa_{\nabla c}}{2}\bar\alpha_{k}^{2}\|\tilde{\Delta}\boldsymbol x_{k}\|^{2} \quad \text{(by $\bar {\alpha}_k\leq1$)}\\
& =\bar\alpha_{k}{\tau_{K^\star_{\tau\nu}}}\barg_k^{T}\tilde{\Delta}\boldsymbol x_{k}+\bar\alpha_{k}{\tau_{K^\star_{\tau\nu}}}(\nabla f_k-\barg_k)^T\tilde{\Delta}\boldsymbol x_{k}-\bar\alpha_{k}\|c_k\|+\bar\alpha_{k}\|G_k-\bar{G}_{k}\|\|\tilde{\Delta}\boldsymbol x_{k}\|+\frac{{\tau_{K^\star_{\tau\nu}}}\kappa_{\nabla f}+\kappa_{\nabla c}}{2}\bar\alpha_{k}^{2}\|\tilde{\Delta}\boldsymbol x_{k}\|^{2}\\
& \stackrel{\mathclap{\eqref{snequ:10}}}{\leq}-\bar\alpha_k\Delta q(\tDelta\bx_k;{\tau_{K^\star_{\tau\nu}}}, \bx_k, \barg_k, \tB_k)+\bar\alpha_{k}{\tau_{K^\star_{\tau\nu}}}(\nabla f_k-\barg_k)^T\tilde{\Delta}\boldsymbol x_{k}+\bar\alpha_{k}\|G_k-\bar{G}_{k}\|\|\tilde{\Delta}\boldsymbol x_{k}\|\\
&\quad +\frac{{\tau_{K^\star_{\tau\nu}}}\kappa_{\nabla f}+\kappa_{\nabla c}}{2}\bar\alpha_{k}^{2}\|\tilde{\Delta}\boldsymbol x_{k}\|^{2}\\
&\stackrel{\mathclap{\eqref{snequ:7}}}{\leq}-\frac{\nu_{K^\star_{\tau\nu}}\alpha_k}{{\tau_{K^\star_{\tau\nu}}}\kappa_{\nabla f}+\kappa_{\nabla c}}\Delta q(\tDelta\bx_k;{\tau_{K^\star_{\tau\nu}}}, \bx_k, \barg_k, \tB_k)+\rbr{\frac{\nu_{K^\star_{\tau\nu}}\alpha_k}{{\tau_{K^\star_{\tau\nu}}}\kappa_{\nabla f}+\kappa_{\nabla c}}+\psi \alpha_k^p}{\tau_{K^\star_{\tau\nu}}}\|(\nabla f_k-\barg_k)\|\|\tilde{\Delta}\boldsymbol x_{k}\|\\
&\quad +\rbr{\frac{\nu_{K^\star_{\tau\nu}}\alpha_k}{{\tau_{K^\star_{\tau\nu}}}\kappa_{\nabla f}+\kappa_{\nabla c}}+\psi \alpha_k^p}\|G_k-\bar{G}_{k}\|\|\tilde{\Delta}\boldsymbol x_{k}\|+\frac{{\tau_{K^\star_{\tau\nu}}}\kappa_{\nabla f}+\kappa_{\nabla c}}{2}\rbr{\frac{\nu_{K^\star_{\tau\nu}}\alpha_k}{{\tau_{K^\star_{\tau\nu}}}\kappa_{\nabla f}+\kappa_{\nabla c}}+\psi \alpha_k^p}^{2}\|\tilde{\Delta}\boldsymbol x_{k}\|^{2}.
\end{align*}
Taking conditional expectation $\mE[\cdot\mid\F_{k-1}]$ and subtracting $f_{\text{inf}}$ in Assumption \ref{ass:1-1} on both sides,~we~have
\begin{align}\label{aequ:9}
& \mE[\phi_{\tau_{K^\star_{\tau\nu}}}(\boldsymbol x_{k}+\bar{\alpha}_k\tilde{\Delta}\boldsymbol x_{k})- f_{\text{inf}}\mid \F_{k-1}]\leq \phi_{\tau_{K^\star_{\tau\nu}}}(\boldsymbol x_{k})- f_{\text{inf}}-\frac{\nu_{K^\star_{\tau\nu}}\alpha_k}{{\tau_{K^\star_{\tau\nu}}}\kappa_{\nabla f}+\kappa_{\nabla c}}\mE[\Delta q(\tDelta\bx_k;{\tau_{K^\star_{\tau\nu}}}, \bx_k, \barg_k, \tB_k)\mid \F_{k-1}] \nonumber\\
& \quad +\rbr{\frac{\nu_{K^\star_{\tau\nu}}\alpha_k}{{\tau_{K^\star_{\tau\nu}}}\kappa_{\nabla f}+\kappa_{\nabla c}}+\psi \alpha_k^p} \cbr{\tau_{K^\star_{\tau\nu}}\mE[\|\nabla f_k-\barg_k \|\|\tilde{\Delta}\boldsymbol x_{k}\|\mid \F_{k-1}] + \mE[\|G_k-\bar{G}_{k}\|\|\tilde{\Delta}\boldsymbol x_{k}\|\mid \F_{k-1}]} \nonumber\\
&\quad +\frac{{\tau_{K^\star_{\tau\nu}}}\kappa_{\nabla f}+\kappa_{\nabla c}}{2}\rbr{\frac{\nu_{K^\star_{\tau\nu}}\alpha_k}{{\tau_{K^\star_{\tau\nu}}}\kappa_{\nabla f}+\kappa_{\nabla c}}+\psi \alpha_k^p}^{2}\mE[\|\tilde{\Delta}\boldsymbol x_{k}\|^{2}\mid \F_{k-1}] \nonumber\\
& \leq \phi_{\tau_{K^\star_{\tau\nu}}}(\boldsymbol x_{k})- f_{\text{inf}} - \frac{\tilde{\nu}\alpha_k}{\tau_{-1}\kappa_{\nabla f}+\kappa_{\nabla c}}\mE[\Delta q(\tDelta\bx_k;{\tau_{K^\star_{\tau\nu}}}, \bx_k, \barg_k, \tB_k)\mid \F_{k-1}] \nonumber\\
& \quad +\rbr{\frac{\nu_{-1}\alpha_k}{\kappa_{\nabla c}}+\psi \alpha_k^p} \cbr{\tau_{-1}\mE[\|\nabla f_k-\barg_k \|\|\tilde{\Delta}\boldsymbol x_{k}\|\mid \F_{k-1}] + \mE[\|G_k-\bar{G}_{k}\|\|\tilde{\Delta}\boldsymbol x_{k}\|\mid \F_{k-1}]} \nonumber\\ 
&\quad +\frac{\tau_{-1}\kappa_{\nabla f}+\kappa_{\nabla c}}{2}\rbr{\frac{\nu_{-1}\alpha_k}{\kappa_{\nabla c}}+\psi \alpha_k^p}^{2}\mE[\|\tilde{\Delta}\boldsymbol x_{k}\|^{2}\mid \F_{k-1}],
\end{align}
where the last inequality utilizes Lemma \ref{parameter stabilize}. We now derive bounds for each positive conditional~expectation term in \eqref{aequ:9} so that we can apply Robbins-Siegmund theorem \citep{Robbins1985Convergence}. In particular, we have
\begin{align}\label{aequ:10}
& \mE\sbr{\sum_{k=0}^\infty\rbr{\frac{\nu_{-1}\alpha_k}{\kappa_{\nabla c}}+\psi \alpha_k^p} \tau_{-1}\mE[\|\nabla f_k-\barg_k \|\|\tilde{\Delta}\boldsymbol x_{k}\|\mid \mF_{k-1}]} \nonumber\\
& = \sum_{k=0}^\infty\rbr{\frac{\nu_{-1}\alpha_k}{\kappa_{\nabla c}}+\psi \alpha_k^p} \tau_{-1}\mE[\|\nabla f_k-\barg_k \|\|\tilde{\Delta}\boldsymbol x_{k}\|] \quad \text{(by Tonelli's theorem)} \nonumber\\
& \leq \sum_{k=0}^\infty\rbr{\frac{\nu_{-1}\alpha_k}{\kappa_{\nabla c}}+\psi \alpha_k^p} \tau_{-1} \sqrt{\mE[\|\nabla f_k-\barg_k \|^2]}\sqrt{\mE[\|\tilde{\Delta}\boldsymbol x_{k}\|^2]}\quad \text{(by Cauchy-Schwarz inequality)} \nonumber\\
& \stackrel{\mathclap{\eqref{nequ:10}}}{\leq} \sum_{k=0}^\infty\rbr{\frac{\nu_{-1}\alpha_k}{\kappa_{\nabla c}}+\psi \alpha_k^p} \tau_{-1}\Upsilon_{K}\sqrt{2}(\Upsilon_{\barg}^{1/r}+\kappa_c) \sqrt{\mE\left[\|\nabla f_k-\barg_k \|^2\right]} \quad \text{(by the same analysis of \eqref{aequ:B9})} \nonumber\\
& \leq \sum_{k=0}^\infty(\frac{\nu_{-1}\alpha_k}{\kappa_{\nabla c}}+\psi \alpha_k^p) \tau_{-1}\Upsilon_{K}(\Upsilon_{\barg}+\kappa_c) \left(\sqrt{\beta_k}+b_k^2+\frac{\alpha_k}{\beta_k}\right) \quad \text{(by Lemma \ref{final fundamental lemma})} \nonumber\\
& < \infty,
\end{align}
where the last inequality is ensured by $p\geq 1$, and $p_1+0.5p_2>1$, $p_1+2p_3>1$, $2p_1-p_2>1$, as~assumed in \eqref{cond:3.10} in the statement of the lemma. Therefore, we immediately have
\begin{equation}\label{aequ:11}
\mE\sbr{\sum_{k=K^\star}^\infty\rbr{\frac{\nu_{-1}\alpha_k}{\kappa_{\nabla c}}+\psi \alpha_k^p} \tau_{-1}\mE[\|\nabla f_k-\barg_k \|\|\tilde{\Delta}\boldsymbol x_{k}\|\mid \mF_{k-1}] }<\infty 
\end{equation}
and hence 
\begin{equation*}
\mE\sbr{\sum_{k=K^\star}^\infty(\frac{\nu_{-1}\alpha_k}{\kappa_{\nabla c}}+\psi \alpha_k^p) \tau_{-1}\mE[\|\nabla f_k-\barg_k \|\|\tilde{\Delta}\boldsymbol x_{k}\|\mid \mF_{k-1}] \mid \mF_{K^\star-1}}<\infty \text{\quad\quad almost surely}.
\end{equation*}
Following the same analysis as in \eqref{aequ:10} and \eqref{aequ:11}, we have 
\begin{equation*}
\begin{aligned}
&\mE\left[\sum_{k=K^\star}^\infty(\frac{\nu_{-1}\alpha_k}{\kappa_{\nabla c}}+\psi \alpha_k^p)\mE[\|G_k-\bar{G}_{k}\|\|\tilde{\Delta}\boldsymbol x_{k}\|\mid \F_{k-1}]\mid \F_{K^\star-1}\right] <\infty\quad \text{almost surely},\\
& \mE\left[\sum_{k=K^\star}^\infty (\frac{\nu_{-1}\alpha_k}{\kappa_{\nabla c}}+\psi \alpha_k^p)^{2}\mE[\|\tilde{\Delta}\boldsymbol x_{k}\|^{2}\mid \F_{k-1}]\mid \F_{K^\star-1}\right]<\infty \quad \text{almost surely}.
\end{aligned}
\end{equation*}
Combining the above two displays with \eqref{aequ:9}, we have from Robbins-Siegmund theorem \citep{Robbins1985Convergence} that 
\begin{multline}\label{aequ:13}
\sum_{k=K^{\star}}^{\infty}\alpha_k\mE\left[\mE[\Delta q(\tDelta\bx_k;{\tau_{K^\star_{\tau\nu}}}, \bx_k, \barg_k, \tB_k)\mid\F_{k-1}]\mid\F_{K^\star-1}\right]\\ =\sum_{k=K^{\star}}^{\infty}\alpha_k\mE[\Delta q(\tDelta\bx_k;{\tau_{K^\star_{\tau\nu}}}, \bx_k, \barg_k, \tB_k)\mid\F_{K^\star-1}]<\infty,
\end{multline}
which implies $P(\sum_{k=K^{\star}}^{\infty}\alpha_k\Delta q(\tDelta\bx_k;{\tau_{K^\star_{\tau\nu}}}, \bx_k, \barg_k, \tB_k)<\infty\mid\F_{K^\star-1})=1$. Since the result holds~for~any $\F_{K^\star-1}$, we integrate out the randomness of $\mF_{K^\star-1}$ and obtain 
\begin{equation*}
\sum_{k=K^{\star}}^{\infty}\alpha_k\Delta q(\tDelta\bx_k;{\tau_{K^\star_{\tau\nu}}}, \bx_k, \barg_k, \tB_k)<\infty \quad \text{almost surely}.
\end{equation*}
Utilizing $\sum_{k=K^{\star}}^\infty\alpha_k=\infty $ for any run of the algorithm, we know  $\liminf_{k \to \infty} \Delta q(\tDelta\bx_k;{\tau_{k}}, \bx_k, \barg_k, \tB_k)=0$ almost surely. Furthermore, by Lemma \ref{thm:convergence1}(c) and Lemma \ref{parameter stabilize}, we know that $\sum_{k=K^{\star}}^{\infty}\alpha_k(\|\tDelta \bx_k\|^2+\|c_k\|)<\infty$ almost surely. On the other hand, we note for $k\geq K^\star$ that
\begin{align*}
\|\tilde{\Delta} \boldsymbol{x}_k-\Delta \boldsymbol{x}_k\| & \stackrel{\mathclap{\eqref{def:acc2}}}{\leq}\left\|\left(\begin{array}{cc}\tB_k & \bar{G}_k^{T} \\ \bar{G}_k & \boldsymbol 0\end{array}\right)^{-1}\left(\begin{array}{c}\barg_k \\ c_k\end{array}\right)-\left(\begin{array}{cc}\tB_k & G_k^{T} \\ G_k & \boldsymbol 0\end{array}\right)^{-1}\left(\begin{array}{c}\nabla f_{k} \\ c_k\end{array}\right)\right\|\\
& \leq\Upsilon_K^2\left\| \left(\begin{array}{cc}\boldsymbol 0 & \bar{G}_k^{T}-G_k^{T} \\ \bar{G}_k-G_k & \boldsymbol 0\end{array}\right)\right\| \left\| \left(\begin{array}{c}\nabla f_k \\ c_k\end{array}\right)\right\| +\Upsilon_K\left\|\left(\begin{array}{c}\barg_k-\nabla f_k \\ \boldsymbol 0\end{array}\right) \right\|_2\\
& \leq 2\Upsilon_K^2(\Upsilon_{\nabla f}+\kappa_c)\|\bar{G}_k-G_k\| + \Upsilon_K\|\barg_k-\nabla f_k\|,
\end{align*}
where the last inequality is due to the boundedness of $\nabla f_k$ (cf. \eqref{aequ:B4}) and the boundedness of $c_k$ in Assumption \ref{ass:1-1}. Following the same analysis as in \eqref{aequ:10} and applying Lemma \ref{final fundamental lemma}, we have
\begin{equation*}
\mE\sbr{\sum_{k=0}^\infty \alpha_k(\|\bar{G}_k-G_k\|^2 + \|\barg_k-\nabla f_k\|^2)}<\infty.
\end{equation*}
The above two displays imply $\mE[\sum_{k=K^\star}^\infty\alpha_k\|\tilde{\Delta} \boldsymbol{x}_k-\Delta \boldsymbol{x}_k\|^2]<\infty$ and thus, $\sum_{k=K^\star}^\infty\alpha_k\|\tilde{\Delta} \boldsymbol{x}_k-\Delta \boldsymbol{x}_k\|^2<\infty$ almost surely. With this result and $\sum_{k=K^{\star}}^{\infty}\alpha_k(\|\tDelta \bx_k\|^2+\|c_k\|)<\infty$, we have almost surely
\begin{equation*}
\sum_{k=K^{\star}}^{\infty} \alpha_k(\|\Delta \bx_k\|^2 + \|c_k\|) \leq \sum_{k=K^{\star}}^{\infty}\alpha_k(2\|\tDelta \bx_k\|^2+\|c_k\|)+2\sum_{k=K^{\star}}^{\infty}\alpha_k\|\Delta \bx_k-\tDelta \bx_k\|^2<\infty.
\end{equation*}
Utilizing $\sum_{k=K^{\star}}^\infty\alpha_k=\infty $ for any run of the algorithm, we obtain $\liminf_{k \to \infty} (\|\Delta \bx_k\|^2+\|c_k\|)=0$ almost surely. This completes the proof.

\subsection{Proof of Theorem \ref{final convergence}}\label{appendix final convergence}

Let us consider $k\geq K^\star\coloneqq\max\{K^\star_1, K^\star_2, K^{\star}_G, K^\star_{\tau\nu}\}$ and define $\boldsymbol\lambda^{\text{sub}}_k=\boldsymbol\lambda_{k}+\Delta\boldsymbol\lambda_{k}$. By \eqref{def:acc2}, replacing $\barg_k$ with $\nabla f_k$ and $\tilde{G}_k$ with $G_k$, we have $\tB_k{\Delta}\boldsymbol x_k+G_k^T{\Delta}\boldsymbol \lambda_k=-{\nabla} f_k-G_k^T\boldsymbol \lambda_k$. By Assumption \ref{ass:1-1}, we~have
\begin{equation*}
\|{\nabla} f_k+G_k^T\boldsymbol \lambda^{\text{sub}}_k\| = \|\tB_k{\Delta}\boldsymbol x_k\| \leq\kappa_{2,\tB}\|{\Delta}\boldsymbol x_k\|.
\end{equation*}
By Lemma \ref{convergence2}, we know $\sum_{k=K^{\star}}^\infty\alpha_k(\|{\Delta} \boldsymbol{x}_k\|^2+\|c_k\|)<\infty$; thus, $\sum_{k=K^{\star}}^\infty\alpha_k(\|{\nabla} f_k+G_k^T\boldsymbol  \lambda^{\text{sub}}_k\|^2+\|c_k\|)<\infty$ almost surely. Furthermore, if we define $\blambda_k^{\star \text{true}}=-[G_kG_k^T]^{-1}G_k\nabla f_k$, which is indeed well-defined based on Assumption \ref{ass:1-1}, then 
\begin{equation}\label{aaequ:3}
\sum_{k=K^{\star}}^\infty\alpha_k(\|\nabla f_k+ G_k^T\blambda_k^{\star \text{true}}\|^2+\|c_k\|)\leq\sum_{k=K^{\star}}^\infty\alpha_k(\|{\nabla} f_k+G_k^T\boldsymbol  \lambda^{\text{sub}}_k\|^2+\|c_k\|) <\infty.
\end{equation}
Together with $\sum_{k=K^{\star}}^\infty\alpha_k=\infty$, we know almost surely
\begin{equation*}
\liminf_{k\to\infty} (\|\nabla f_k+ G_k^T\blambda_k^{\star \text{true}}\|^2+\|c_k\|)=0.
\end{equation*}
We claim $\lim_{k\to \infty} \|\nabla f_k+ G_k^T\blambda_k^{\star \text{true}}\|+\|c_k\|=0$, and use $\lim_{k\to \infty} \|\nabla f_k+ G_k^T\blambda_k^{\star \text{true}}\|=0$ as an example; the same analysis applies to $\|c_k\|$. Suppose $\limsup_{k\to\infty} [\|\nabla f_k+ G_k^T\blambda_k^{\star \text{true}}\|>0$. For such a run, we can find a sufficiently small number $\epsilon^\star>0$ and two infinite sequences $\{m_i\}$ and $\{n_i\}$ with $K^\star<m_i<n_i$, $\forall i\geq0$, such that
\begin{equation}\label{aequ:17}
\|\nabla f_{m_i}+ G_{m_i}^T\boldsymbol\lambda_{m_i}^{\star \text{true}}\|\geq2\epsilon^\star,\quad\|\nabla f_{n_i}+ G_{n_i}^T\boldsymbol\lambda_{n_i}^{\star \text{true}}\|<\epsilon^\star,\quad \left\|\nabla f_{k}+ G_k^T\blambda_k^{\star \text{true}}\right\|\geq\epsilon^\star\quad\text{for } k\in [m_i,n_i).
\end{equation}
Then, we have for some (potentially random) constant $\Upsilon>0$ that
\begin{align}\label{aequ:18}
\epsilon^\star  &\;\; \stackrel{\mathclap{\eqref{aequ:17}}}{\leq} \; \|(\nabla f_{m_i}+ G_{m_i}^T\boldsymbol\lambda_{m_i}^{\star \text{true}})\|-\|\nabla f_{n_i}+ G_{n_i}^T\boldsymbol\lambda_{n_i}^{\star \text{true}}\| \nonumber\\
& \;\; =\sum_{k=m_i}^{n_i-1}\left(\|\nabla f_{k}+ G_k^T\blambda_k^{\star \text{true}}\|-\|\nabla f_{k+1}+ G_{k+1}^T\boldsymbol\lambda_{k+1}^{\star \text{true}}\|\right) \nonumber\\
&\;\; \leq \sum_{k=m_i}^{n_i-1}\|\nabla f_{k}+ G_k^T\blambda_k^{\star \text{true}}-\nabla f_{k+1}- G_{k+1}^T\boldsymbol\lambda_{k+1}^{\star \text{true}}\| \nonumber\\
&\;\; \leq \sum_{k=m_i}^{n_i-1}(\|\nabla f_{k}-\nabla f_{k+1}\|+\|G_k-G_{k+1}\|\|\blambda_k^{\star \text{true}}\|+\|G_{k+1}\|\|\blambda_k^{\star \text{true}}-\blambda_{k+1}^{\star \text{true}}\|) \nonumber\\
&\;\; \stackrel{\mathclap{\eqref{snequ:7}}}{\leq} \Upsilon\sum_{k=m_i}^{n_i-1}(\frac{\nu_{-1}\alpha_k}{\kappa_{\nabla c}} + \psi \alpha_k^p),
\end{align}
where the existence of $\Upsilon$ in the last inequality is due to the same analysis as in \eqref{aequ:4} (note that $\barg_k$ is bounded for any particular run due to Lemma \ref{lemma:average almost sure} and boundedness of $\nabla f_k$ in Assumption \ref{ass:1-1}). Multiplying both sides of \eqref{aequ:18} by $(\epsilon^\star)^2$, we have
\begin{equation*}
(\epsilon^\star)^3\stackrel{{\eqref{aequ:17}}}{\leq} \Upsilon\sum_{k=m_i}^{n_i-1}\rbr{\frac{\nu_{-1}\alpha_k}{\kappa_{\nabla c}} + \psi \alpha_k^p}\left\|\nabla f_{k}+ G_k^T\blambda_k^{\star \text{true}}\right\|^2,
\end{equation*}
which implies that
\begin{equation*}
\infty\leq\sum_{i=0}^\infty\sum_{k=m_i}^{n_i-1}\rbr{\frac{\nu_{-1}\alpha_k}{\kappa_{\nabla c}} + \psi \alpha_k^p}\left\|\nabla f_{k}+ G_k^T\blambda_k^{\star \text{true}}\right\|^2\leq\sum_{k=K^\star}^\infty(\frac{\nu_{-1}\alpha_k}{\kappa_{\nabla c}} + \psi \alpha_k^p)\left\|\nabla f_{k}+ G_k^T\blambda_k^{\star \text{true}}\right\|^2\stackrel{\eqref{aaequ:3}}{<}\infty.
\end{equation*}
Here, the last inequality also uses the fact that $p\geq 1$. This leads to a contradiction. Thus, we obtain $\lim_{k\to \infty} \|\nabla f_k+ G_k^T\blambda_k^{\star \text{true}}\|+\|c_k\|=0$ almost surely. 
By Lemma \ref{lemma:average almost sure} and the definitions of $\blambda_k^{\star \text{true}}$ and $\blambda_k^\star$, we have $\blambda_k^{\star \text{true}}-\blambda_k^\star\rightarrow \0$ as $k\rightarrow\infty$ almost surely, which implies $\lim_{k \to \infty} \|\nabla f_k+ G_k^T\boldsymbol\lambda_k^\star\|+\|c_k\|=0$ almost surely. This completes the proof.

\section{Proofs of Section \ref{sec:4}}

\subsection{Proof of Lemma \ref{iterate convergence}}\label{appendix iterate convergence}

Recall from the proof of Theorem \ref{final convergence} that $\blambda^{\text{sub}}_k\coloneqq \blambda_{k}+\Delta\blambda_{k}$, where we use $(\Delta \bx_k,\Delta\blambda_k)$ to denote the solution of \eqref{def:acc2} but with $\barg_k$ replaced by $\nabla f_k$ and $\tilde{ G}_k$ replaced by $ G_k$. Let us define $\ttlambda^{\text{sub}}_k = \blambda_k + \tDelta\blambda_k$. By the proof of Lemma \ref{convergence2}, we know for any run of the algorithm, there exists a (potentially~random) $K^\star<\infty$ such that $(\Delta\bx_{k}, \Delta\blambda_{k})$ is well-defined (note that Lemma \ref{lemma:average almost sure} is applicable since \eqref{cond:4.4} implies \eqref{cond:3.6}). By \eqref{def:acc2}, we note for $k\geq K^\star$ that
\begin{equation*}
\begin{pmatrix}
\tB_k & G_k^T\\
G_k & \0
\end{pmatrix}\begin{pmatrix}
\Delta\bx_k\\
\blambda_k^{\text{sub}}
\end{pmatrix} = -\begin{pmatrix}
\nabla f_k\\
c_k
\end{pmatrix}\quad\quad \text{ and }\quad\quad \begin{pmatrix}
\tB_k & (G^\star)^T\\
G^\star & \0
\end{pmatrix}\begin{pmatrix}
\0\\
\tlambda
\end{pmatrix} = -\begin{pmatrix}
\nabla f^\star\\
\0
\end{pmatrix}.
\end{equation*}
Therefore, we have
\begin{align}\label{aaequ:5}
\nbr{\begin{pmatrix}
\Delta\bx_k\\
\blambda_k^\text{sub} - \tlambda
\end{pmatrix}} & = \nbr{\begin{pmatrix}
\tB_k & G_k^T\\
G_k & \0
\end{pmatrix}^{-1}\begin{pmatrix}
\nabla f_k\\
c_k
\end{pmatrix} - \begin{pmatrix}
\tB_k & (G^\star)^T\\
G^\star & \0
\end{pmatrix}^{-1}\begin{pmatrix}
\nabla f^\star\\
\0
\end{pmatrix}} \nonumber\\
& \stackrel{\mathclap{\eqref{aaequ:4}}}{\leq}\; \Upsilon_{K}\nbr{\begin{pmatrix}
\nabla f_k - \nabla f^\star\\
c_k
\end{pmatrix}} + \Upsilon_{K}^2\|\nabla f^\star\|\nbr{\begin{pmatrix}
\0 & G_k^T - (G^\star)^T\\
G_k - G^\star & \0
\end{pmatrix}} \nonumber\\
& \leq\; \Upsilon_{K} (\kappa_{\nabla f} + \kappa_c)\|\bx_k - \tx\| + 2\Upsilon_{K}^2\Upsilon_{\nabla f}\kappa_{\nabla c}\|\bx_k - \tx\|,
\end{align}
where in the last inequality, $\kappa_{\nabla f}, \kappa_{\nabla c}$ denote the Lipschitz constants of $\nabla f$ and $G = \nabla c$; $\Upsilon_{\nabla f}$ is the upper bound of $\nabla f$ over $\mX$ (cf. Appendix \ref{appendix lemma:average almost sure}); and we abuse the notation $\kappa_c$ from Assumption \ref{ass:1-1} to denote the Lipschitz constant of $c$ over $\mX$. Note that $\kappa_c$ always exists since $c$ has bounded Jacobian~as assumed in Assumption \ref{ass:1-1}.
Thus, we have from \eqref{aaequ:5} that $\blambda_k^{\text{sub}}\rightarrow\tlambda$ almost surely. Then, we characterize $\blambda_k^{\text{sub}} - \ttlambda_k^{\text{sub}}$. We have from \eqref{def:acc2} that
\begin{equation*}
\begin{pmatrix}
\tB_k & \tG_k^T\\
\tG_k & \0
\end{pmatrix}\begin{pmatrix}
\tDelta\bx_k\\
\tDelta\blambda_k
\end{pmatrix} = -\begin{pmatrix}
\barg_k+\tG_k^T\blambda_k\\
c_k
\end{pmatrix}\quad\;\text{ and }\quad\; \begin{pmatrix}
\tB_k & G_k^T\\
G_k & \0
\end{pmatrix}\begin{pmatrix}
\Delta\bx_k\\
\Delta\blambda_k
\end{pmatrix} = -\begin{pmatrix}
\nabla f_k + G_k^T\blambda_k\\
c_k
\end{pmatrix}.
\end{equation*}
Following the same derivations as in \eqref{aaequ:5} and applying Lemma \ref{lemma:average almost sure}, we immediately obtain $\|(\tDelta\bx_k - \Delta\bx_k, \tDelta\blambda_k-\Delta\blambda_k)\|\rightarrow 0$ as $k\rightarrow\infty$ almost surely; thus $\|\blambda_{k}^\text{sub} - \ttlambda_k^\text{sub}\|\rightarrow 0$. Combining the above convergence results, we know $\ttlambda_k^{\text{sub}}\rightarrow\tlambda$ as $k\rightarrow\infty$ almost surely. Finally, for any run of the algorithm and any $\epsilon>0$, we abuse the notation $K^\star$ to let $\baralpha_k \leq 1$ and $\|\ttlambda_k^{\text{sub}} - \tlambda\|_2\leq \epsilon$ for $k\geq K^\star$. Then, we know that, for any $k\geq K^\star$,
\begin{align*}
\|\blambda_{k+1}-\tlambda\| & = \|\blambda_k-\tlambda + \baralpha_k\tDelta\blambda_k\|  \leq  (1-\baralpha_k)\|\blambda_k-\tlambda\| + \baralpha_k\|\ttlambda_k^{\text{sub}} - \tlambda\|\\
& \leq \prod_{j=K^\star}^{k}(1-\baralpha_j)\|\blambda_{K^\star} - \tlambda\| + \sum_{i=\tK}^{k}\prod_{j=i+1}^{k}(1-\baralpha_j)\baralpha_i\|\ttlambda_i^{\text{sub}} - \tlambda\|\\
& \leq \prod_{j=K^\star}^{k}(1-\baralpha_j)\|\blambda_{K^\star} - \tlambda\| + \epsilon \sum_{i=\tK}^{k}\prod_{j=i+1}^{k}(1-\baralpha_j)\baralpha_i\\
& = \prod_{j=K^\star}^{k}(1-\baralpha_j)\|\blambda_{K^\star}- \tlambda\|  + \epsilon\{1 - \prod_{j=K^\star}^{k}(1-\baralpha_j)\}\\
& \leq \|\blambda_{K^\star}- \tlambda\| \exp\rbr{-\sum_{j=K^\star}^{k}\baralpha_k} + \epsilon,
\end{align*}
where the third inequality is due to the second inequality and induction. Noting that $\sum_{j=K^\star}^{\infty}\baralpha_k = \infty$ as $p_1\leq 1$, we know there exists $K^{\star\star}\geq K^\star$ such that $\|\blambda_{K^\star}- \tlambda\| \exp(-\sum_{j=K^\star}^{k}\baralpha_k)\leq \epsilon$ for all $k\geq K^{\star\star}$. This implies that $\|\blambda_{k+1} - \tlambda\|\leq 2\epsilon$ for all $k\geq K^{\star\star}$ and we complete the proof.

\subsection{Proof of Lemma \ref{lem:Hessian Convergence}}\label{appendix Hessian Convergence}

We note that by \eqref{snequ:5},
\begin{equation*}
\barB_k = \sum_{i=0}^{k}\prod_{j=i+1}^{k}(1-\beta_j)\beta_i\hnabla_{\bx}^2\mL_i + \prod_{i=0}^{k}(1-\beta_i)\barB_{-1}.
\end{equation*}
Without loss of generality, we suppose $\beta_k\leq 1$ for all $k\geq 0$ (otherwise, we just consider $k$ large enough). We obtain from the above display that
\begin{align}\label{cequ:1}
& \|\barB_k - \nabla_{\bx}^2\mL^\star\|  = \nbr{\sum_{i=0}^{k}\prod_{j=i+1}^{k}(1-\beta_j)\beta_i(\hnabla_{\bx}^2\mL_i - \nabla_{\bx}^2\mL^\star) + \prod_{i=0}^{k}(1-\beta_i)(\barB_{-1} - \nabla_{\bx}^2\mL^\star)}\nonumber\\
& \leq \nbr{\sum_{i=0}^{k}\prod_{j=i+1}^{k}(1-\beta_j)\beta_i(\hnabla_{\bx}^2\mL_i - \nabla_{\bx}^2\mL_i)} + \nbr{\sum_{i=0}^{k}\prod_{j=i+1}^{k}(1-\beta_j)\beta_i(\nabla_{\bx}^2\mL_i - \nabla_{\bx}^2\mL^\star)} \nonumber\\
&\quad + \prod_{i=0}^{k}(1-\beta_i)\|\barB_{-1} - \nabla_{\bx}^2\mL^\star\| \nonumber\\
& \leq \nbr{\sum_{i=0}^{k}\prod_{j=i+1}^{k}(1-\beta_j)\beta_i(\hnabla^2F(\bx_i;\xi_i) - \mE[\hnabla^2F(\bx_i;\xi_i)\mid \mF_{i-1}])} \nonumber\\
& \quad + \sum_{i=0}^{k}\prod_{j=i+1}^{k}(1-\beta_j)\beta_i\cdot\|\mE[\hnabla^2F(\bx_i;\xi_i)\mid \mF_{i-1}] -\nabla^2f_i\| \nonumber\\
& \quad + \sum_{l=1}^{m}\sum_{i=0}^{k}\prod_{j=i+1}^{k}(1-\beta_j)\beta_i\cdot |\blambda_i^l-(\tlambda)^l|\cdot\|\hnabla^2c_i^l - \nabla c_i^l\| \nonumber\\
& \quad + \sum_{l=1}^{m}|(\tlambda)^l|\cdot \nbr{\sum_{i=0}^{k}\prod_{j=i+1}^{k}(1-\beta_j)\beta_i(\hnabla^2c_i^l - \mE[\hnabla^2c_i^l\mid \mF_{i-1}])} \nonumber\\
& \quad + \sum_{l=1}^{m}|(\tlambda)^l|\cdot\sum_{i=0}^{k}\prod_{j=i+1}^{k}(1-\beta_j)\beta_i\cdot \|\mE[\hnabla^2c_i^l\mid \mF_{i-1}]-\nabla^2c_i^l\| \nonumber\\
& \quad + \sum_{i=0}^{k}\prod_{j=i+1}^{k}(1-\beta_j)\beta_i \|\nabla_{\bx}^2\mL_i - \nabla_{\bx}^2\mL^\star\| + \prod_{i=0}^{k}(1-\beta_i)\|\barB_{-1} - \nabla_{\bx}^2\mL^\star\| \nonumber\\
& \eqqcolon \I_1^k + \I_2^k + \I_3^k + \I_4^k + \I_5^k + \I_6^k + \I_7^k.
\end{align}
We analyze each term separately. We first present a generic result. For any sequence $e_i\rightarrow 0$ as $i\rightarrow\infty$, we have $\sum_{i=0}^{k}\prod_{j=i+1}^{k}(1-\beta_j)\beta_ie_i\rightarrow0$ as $k\rightarrow\infty$ as long as $\sum_{i=0}^{\infty}\beta_i = \infty$ (as implied by \eqref{cond:4.5}). In fact, for any $\epsilon>0$, there exists $i'>0$ such that $|e_i|\leq \epsilon$ for any $i\geq i'$. Thus, for $k\geq i'$, we have
\begin{align*}
\abr{\sum_{i=0}^{k}\prod_{j=i+1}^{k}(1-\beta_j)\beta_ie_i} & \leq \sum_{i=0}^{i'-1}\prod_{j=i+1}^{k}(1-\beta_j)\beta_i|e_i| + \sum_{i=i'}^{k}\prod_{j=i+1}^{k}(1-\beta_j)\beta_i|e_i|\\
& \leq \prod_{j=i'}^{k}(1-\beta_j)\cdot\sum_{i=0}^{i'-1}\prod_{j=i+1}^{i'-1}(1-\beta_j)\beta_i|e_i| + \epsilon \sum_{i=i'}^{k}\prod_{j=i+1}^{k}(1-\beta_j)\beta_i\\
& = \prod_{j=i'}^{k}(1-\beta_j)\cdot\sum_{i=0}^{i'-1}\prod_{j=i+1}^{i'-1}(1-\beta_j)\beta_i|e_i| + \epsilon\cbr{1-\prod_{j=i'}^{k}(1-\beta_j)}\\
& \leq \exp\rbr{-\sum_{j=i'}^{k}\beta_j}\cdot\sum_{i=0}^{i'-1}\prod_{j=i+1}^{i'-1}(1-\beta_j)\beta_i|e_i| + \epsilon.
\end{align*}
Since $\sum_{i=0}^{\infty}\beta_i=\infty$, we can find $k'\geq i'$ large enough such that $\exp(-\sum_{j=i'}^{k}\beta_j)\cdot\sum_{i=0}^{i'-1}\prod_{j=i+1}^{i'-1}(1-\beta_j)\beta_i|e_i|\leq \epsilon$ for any $k\geq k'$. Then, we obtain for $k\geq k'$ that
\begin{equation*}
\abr{\sum_{i=0}^{k}\prod_{j=i+1}^{k}(1-\beta_j)\beta_ie_i}\leq 2\epsilon.
\end{equation*}
This shows $\sum_{i=0}^{k}\prod_{j=i+1}^{k}(1-\beta_j)\beta_ie_i\rightarrow0$ as $k\rightarrow\infty$.~With this argument, we study each term~as~follows.
\vskip4pt
\noindent$\bullet$ For $\I_2^k, \I_5^k, \I_6^k$, we know from Lemmas \ref{lemma:finite difference for gradient} and \ref{iterate convergence} that $\|\mE[\hnabla^2F(\bx_i;\xi_i)\mid \mF_{i-1}] -\nabla^2f_i\|\rightarrow 0$, $\|\mE[\hnabla^2c_i^l\mid \mF_{i-1}]-\nabla^2c_i^l\|\rightarrow 0$, $\forall 1\leq l\leq m$, and $\nabla_{\bx}^2\mL_i - \nabla_{\bx}^2\mL^\star\rightarrow 0$ as $i\rightarrow\infty$ almost surely (where we use the conditions $p_3 > 0$ and $2p_4 - p_3 > 0$ from \eqref{cond:4.5}). Thus, $\I_2^k, \I_5^k, \I_6^k\rightarrow 0$ as $k\rightarrow\infty$ almost surely.

\vskip4pt
\noindent$\bullet$ For $\I_7^k$, we have $\prod_{i=0}^{k}(1-\beta_i)\leq \exp(-\sum_{i=0}^{k}\beta_i)\rightarrow 0$ as $k\rightarrow\infty$. Thus, $\I_7^k\rightarrow0$ as $k\rightarrow \infty$.

\vskip4pt
\noindent$\bullet$ For $\I_3^k$, we provide a deterministic upper bound on $\hnabla^2 c_i^l$ for any $1\leq l\leq m$. In particular, we~note~from the definition \eqref{snequ:4} that 
\begin{align}\label{cequ:5}
\hnabla^2 c^l_i & = \frac{\{c^l(\bx_i+b_i\bDelta_i+\tilde{b}_i\tbDelta_i) - c^l(\bx_i+b_i\bDelta_i)\} - \{c^l(\bx_i-b_i\bDelta_i+\tilde{b}_i\tbDelta_i) - c^l(\bx_i-b_i\bDelta_i)\}}{2b_i\tilde{b}_i} \nonumber\\
&\quad\times \frac{\bDelta_i^{-1}\tbDelta_i^{-T}+\tbDelta_i^{-1}\bDelta_i^{-T}}{2} \nonumber\\
& = \frac{1}{2b_i\tb_i}\int_{0}^{\tb_i}\int_{-b_i}^{b_i}\bDelta_i^T\nabla^2c^l(\bx_i+s_1\bDelta_i+s_2\tbDelta_i)\tbDelta_i\;ds_1ds_2\times \frac{\bDelta_i^{-1}\tbDelta_i^{-T}+\tbDelta_i^{-1}\bDelta_i^{-T}}{2}.
\end{align}
By the boundedness of $\nabla^2c^l$ over $\mX$ and Assumption \ref{ass:Delta}, we know there exists a deterministic constant $\Upsilon_{\hnabla^2 c}>0$ such that $\|\hnabla^2c^l_i\|\leq \Upsilon_{\hnabla^2 c}$ for any $i\geq 0$ and $1\leq l\leq m$. With this boundedness property~and the fact that $\blambda_i^l - (\tlambda)^l\rightarrow 0$ as $i\rightarrow\infty$, we know $\I_3^k\rightarrow0$ as $k\rightarrow\infty$ almost surely. 
\vskip4pt
\noindent$\bullet$ For $\I_4^k$, we apply Lemma \ref{technical lemma:2} and have
\begin{equation}\label{aequ:C4}	\sum_{i=0}^k\prod_{j=i+1}^{k}(1-\beta_j)^2\beta_i^2\mE[\|\hnabla^2c_i^l - \mE[\hnabla^2c_i^l\mid \mF_{i-1}]\|^2\mid \mF_{i-1}] = O(\beta_k)\rightarrow 0\quad \text{ as }\;\; k\rightarrow\infty.
\end{equation}
Thus, the martingale convergence theorem \cite[Theorem 2.18]{Hall2014Martingale} implies that $\I_4^k\rightarrow0$ as $k\rightarrow\infty$ almost surely.
\vskip4pt
\noindent$\bullet$ For $\I_1^k$, based on Assumption \ref{ass:4.2}, let us fix any $0<\delta'<\delta$ and let $K'>0$ be a deterministic index such that for any $\bx\in\{\bx: \|\bx-\tx\|\leq \delta'\}$ and for all $k\geq K'$, we have $\bx + s_1\bDelta + s_2\tbDelta\in\{\bx: \|\bx-\tx\|\leq \delta\}$ for any $s_1\in[-b_k, b_k]$, $s_2\in[0, \tilde{b}_k]$, and $\bDelta, \tbDelta\sim \P_{\bDelta}$. Note that such a $K'$ must exist due to Assumption~\ref{ass:Delta} and the fact that $b_k, \tb_k\rightarrow 0$. Then, we have
\begin{align*}
\I_1^k & = \nbr{\sum_{i=0}^{k}\prod_{j=i+1}^{k}(1-\beta_j)\beta_i(\hnabla^2F(\bx_i;\xi_i) - \mE[\hnabla^2F(\bx_i;\xi_i)\mid \mF_{i-1}])}\\
& \leq \sum_{i=0}^{k}\prod_{j=i+1}^{k}(1-\beta_j)\beta_i\|\hnabla^2F(\bx_i;\xi_i) - \mE[\hnabla^2F(\bx_i;\xi_i)\mid \mF_{i-1}]\|\cdot\1_{\|\bx_i - \tx\|>\delta'} \\
& \quad + \prod_{j=K'}^{k}(1-\beta_j)\sum_{i=0}^{K'-1}\prod_{j=i+1}^{K'-1}(1-\beta_j)\beta_i\|\hnabla^2F(\bx_i;\xi_i) - \mE[\hnabla^2F(\bx_i;\xi_i)\mid \mF_{i-1}]\|\cdot\1_{\|\bx_i - \tx\|\leq\delta'} \\
&\quad + \nbr{\sum_{i=K'}^{k}\prod_{j=i+1}^{k}(1-\beta_j)\beta_i(\hnabla^2F(\bx_i;\xi_i) - \mE[\hnabla^2F(\bx_i;\xi_i)\mid \mF_{i-1}])\cdot\1_{\|\bx_i - \tx\|\leq\delta'} }.
\end{align*}
The first term on the right-hand side converges to zero almost surely since $\bx_i-\tx\rightarrow 0$ as $i\rightarrow\infty$.~The second term converges to zero almost surely since $\prod_{j=K'}^{k}(1-\beta_j)\leq \exp(-\sum_{j=K'}^{k}\beta_j)\rightarrow 0$ as $k\rightarrow\infty$. The third term also converges to zero almost surely by following the same derivation as in \eqref{cequ:5} and applying Assumption \ref{ass:4.2} to show that $\mE[\|\hnabla^2 F(\bx_i;\xi_i)\|^2\mid \mF_{i-1}]$ is bounded for $\bx_i\in\mX\cap\{\bx:\|\bx-\tx\|\leq \delta'\}$, thereby obtaining \eqref{aequ:C4}, and then applying the martingale convergence theorem \cite[Theorem 2.18]{Hall2014Martingale}. Thus, we conclude that $\I_1^k\rightarrow0$ as $k\rightarrow\infty$ almost surely.

Combining the above arguments of $\I_1^k,\I_2^k,\I_3^k,\I_4^k,\I_5^k,\I_6^k,\I_7^k$ and plugging into \eqref{cequ:1}, we have shown that $\barB_k\rightarrow\nabla_{\bx}^2\mL^\star$ as $k\rightarrow\infty$ almost surely. 
For the second part of the statement, for each run of the algorithm with $k$ large enough, we know $\|\barB_k\|\leq \kappa_{1,\tB}$. In addition, we let $\tilde{Z}_k, Z^\star\in \mR^{d\times(d-m)}$ be the matrices whose columns are orthonormal and span the spaces of $\text{ker}(\tilde{G}_k)$, $\text{ker}(G^\star)$, respectively. Then, by Davis-Kahan $\sin(\theta)$ theorem \citep{Davis1970Rotation, Pensky2024Davis} and Lemma \ref{lemma:average almost sure}, we know
\begin{equation*}
\inf_{Q\in \mathcal{Q}_{d-m}}\|\tilde{Z}_k - Z^\star Q\| \leq 2\sqrt{2}\|\tilde{Z}_k\tilde{Z}_k^T - Z^\star(Z^\star)^T\|
= \|\tG_k^T(\tG_k\tG_k^T)^{-1}\tG_k - (G^\star)^T(G^\star(G^\star)^T)^{-1}G^\star\|\rightarrow 0,
\end{equation*}
where $\mathcal{Q}_{d-m}$ denotes the set of $(d-m)\times(d-m)$ orthonormal matrices. Thus, we obtain
\begin{equation*}
\lambda_{\min}(\tilde{Z}_k^T\barB_k\tilde{Z}_k) =  \lambda_{\min}(Q\tilde{Z}_k^T\barB_k\tilde{Z}_kQ^T)\rightarrow\lambda_{\min}((Z^\star)^T\nabla_{\bx}^2\mL^\star Z^\star),
\end{equation*}
which implies $\lambda_{\min}(\tilde{Z}_k^T\barB_k\tilde{Z}_k)\geq \kappa_{1,\tB}$ for large enough $k$. This completes the proof.

\subsection{Proof of Lemma \ref{lem:local:rate}}\label{pf:lem:local:rate}

To simplify the notation, we will just fix $\epsilon\in(0, 1 - 0.5/(\zeta\iota_{1})\1_{p_1=1})$ and denote $\tau_{k_0} = \tau_{k_0}(\epsilon)$. We use $\Upsilon_1, \Upsilon_2,\ldots$ to denote generic deterministic constants and may also use $O(\cdot)$ to ignore them. However, when they depend on $k_0$, we denote by $\Upsilon_i(k_0)$ for clarification and do not write $O(\cdot)$. In what follows,~we suppose $k_0$ is large enough (threshold index is deterministic) such that
\begin{equation}\label{equ:k_0:1}
\frac{\nu_{-1}}{\kappa_{\nabla c}}\alpha_k + \psi \alpha_k^p \leq 0.5\epsilon^5 \quad\;\; \forall k\geq k_0.
\end{equation}
To prove Lemma \ref{lem:local:rate}, we need two lemmas, which are proved in Appendices \ref{pf:lem:1} and \ref{pf:lem:2}.

\begin{lemma}\label{lem:1}
Under the conditions of Lemma \ref{lem:local:rate} and suppose \eqref{equ:k_0:1}, there exist constants $\Upsilon_1, \Upsilon_2(k_0)>0$ such that for any $k\geq k_0$,
\begin{align*}
&\hskip0.6cm  \mE\sbr{\|\bz_{k+1}\|^2 \1_{\tau_{k_0}>k+1}} \leq \mE\sbr{\cbr{1 - 2(1-\epsilon)\baralpha_k} \|\bz_k\|^2\1_{\tau_{k_0}>k+1}}  + \Upsilon_1\alpha_k\mE\sbr{\|\bnabla\mL_k - \nabla\mL_k\|^2\1_{\tau_{k_0}>k}},\\
& \mE\sbr{\|\bnabla\mL_{k+1} - \nabla\mL_{k+1}\|^2\1_{\tau_{k_0}>k+1}}  \leq \Upsilon_1(\beta_k + b_k^4)+\Upsilon_2(k_0)\exp\rbr{-\frac{2\iota_{2}k^{1-p_2}}{1-p_2}} \\
& \hskip4.8cm + \Upsilon_1\rbr{\sum_{i=k_0}^{k}\prod_{j=i+1}^{k}(1-\beta_j)\alpha_i\cbr{\mE[(\|\bnabla\mL_i - \nabla\mL_i\|^2 + \|\bz_i\|^2)\1_{\tau_{k_0}>i}]}^{1/2}}^2.
\end{align*}
\end{lemma}

\begin{lemma}\label{lem:2}
Under the conditions of Lemma \ref{lem:local:rate}, for any $q\geq 0$, there exists a deterministic integer $\bar{k}_0>0$ such that for any $k_0\geq \bar{k}_0$, there exists a constant $\Upsilon_3(k_0)$ such that
\begin{equation*}
\max\cbr{\mE[\|\bz_k\|^2 \1_{\tau_{k_0}>k}],\; \mE[\|\bnabla\mL_k - \nabla\mL_k\|^2\1_{\tau_{k_0}>k}]} \leq \Upsilon_3(k_0)\rbr{\beta_k + b_k^4 + \rbr{\alpha_k/\beta_k}^{2q}}\;\; \text{ for any } k\geq k_0.
\end{equation*}		

\end{lemma}

By Lemma \ref{lem:2}, we choose $q$ large enough such that $2q(p_1-p_2)>\min\{p_2, 4p_3\}$. Then, we have $(\alpha_k/\beta_k)^{2q} = o(\beta_k+b_k^4)$. This completes the proof.

\subsection{Proof of Lemma \ref{lem:1}}\label{pf:lem:1}

By Algorithm \ref{Alg:DF-SSQP}, we know for any fixed $\epsilon\in(0, 1 - 0.5/(\zeta\iota_{1})\1_{p_1=1})$ and $k\geq k_0$,
\begin{align}\label{nequ:1}
& \|\bz_{k+1}\|^2 \nonumber\\
& = \|\bz_k+ \baralpha_k(\tDelta\bx_k, \tDelta\blambda_k)\|^2 = \|\bz_k - \baralpha_k \tW_k^{-1}\bnabla\mL_k\|^2 = \|\bz_k - \baralpha_k \tW_k^{-1} \nabla\mL_k - \baralpha_k\tW_k^{-1}(\bnabla\mL_k - \nabla\mL_k)\|^2 \nonumber\\
& = \|\bz_k - \baralpha_k \tW_k^{-1} \nabla\mL_k\|^2 + \baralpha_k^2\|\tW_k^{-1}(\bnabla\mL_k - \nabla\mL_k)\|^2 - 2\baralpha_k\langle \bz_k - \baralpha_k \tW_k^{-1} \nabla\mL_k, \tW_k^{-1}(\bnabla\mL_k - \nabla\mL_k)\rangle \nonumber\\
& \leq (1+\epsilon\baralpha_k)\|\bz_k - \baralpha_k \tW_k^{-1} \nabla\mL_k\|^2 + (\baralpha_k^2+\baralpha_k/\epsilon)\|\tW_k^{-1}(\bnabla\mL_k - \nabla\mL_k)\|^2.
\end{align}
For the second term on the right-hand side, we apply the definition of $\tau_{k_0}$ in \eqref{equ:def:tau} and have
\begin{equation}\label{nequ:2}
\|\tW_k^{-1}(\bnabla\mL_k - \nabla\mL_k)\|^2\1_{\tau_{k_0}>k+1} \leq \|\tW_k^{-1}(\bnabla\mL_k - \nabla\mL_k)\|^2\1_{\tau_{k_0}>k} \leq \frac{\|\bnabla\mL_k - \nabla\mL_k\|^2\1_{\tau_{k_0}>k}}{\epsilon^2}.
\end{equation}
For the first term on the right-hand side, we have
\begin{align}\label{nequ:3}
\|\bz_k - & \baralpha_k \tW_k^{-1} \nabla\mL_k\|^2\1_{\tau_{k_0}>k+1} \nonumber \\
& = (\|\bz_k\|^2 - 2\baralpha_k\langle\bz_k,  \tW_k^{-1} \nabla\mL_k\rangle + \baralpha_k^2\|\tW_k^{-1} \nabla\mL_k\|^2)\1_{\tau_{k_0}>k+1} \nonumber\\
& = (1-2\baralpha_k)\|\bz_k\|^2\1_{\tau_{k_0}>k+1} + \baralpha_k ( 2\langle\bz_k, \bz_k - \tW_k^{-1} \nabla\mL_k \rangle + \baralpha_k\|\tW_k^{-1} \nabla\mL_k\|^2)\1_{\tau_{k_0}>k+1} \nonumber \\
& \stackrel{\mathclap{\eqref{equ:def:tau}}}{\leq}(1-2\baralpha_k)\|\bz_k\|^2\1_{\tau_{k_0}>k+1} + \baralpha_k\rbr{\frac{2}{\epsilon}\|\bz_k\| \|\nabla\mL_k - \tW_k\bz_k\| + \frac{\baralpha_k}{\epsilon^2}\|\nabla\mL_k\|^2}\1_{\tau_{k_0}>k+1} \nonumber\\
& \stackrel{\mathclap{\eqref{equ:def:tau}}}{\leq} (1-2\baralpha_k)\|\bz_k\|^2\1_{\tau_{k_0}>k+1} + \baralpha_k\rbr{0.5\epsilon\|\bz_k\|^2+\frac{\baralpha_k}{\epsilon^4}\|\bz_k\|^2}\1_{\tau_{k_0}>k+1} \nonumber\\
& \stackrel{\mathclap{\eqref{equ:k_0:1}}}{\leq} (1-(2-\epsilon)\baralpha_k)\|\bz_k\|^2\1_{\tau_{k_0}>k+1}.
\end{align}
Combining \eqref{nequ:1}, \eqref{nequ:2}, \eqref{nequ:3} and applying \eqref{equ:k_0:1}, we obtain
\begin{multline*}
\|\bz_{k+1}\|^2\1_{\tau_{k_0}>k+1} \leq (1+\epsilon\baralpha_k)(1-(2-\epsilon)\baralpha_k)\|\bz_k\|^2\1_{\tau_{k_0}>k+1} + \rbr{0.5\epsilon^3+\frac{1}{\epsilon^3}}\baralpha_k\|\bnabla\mL_k - \nabla\mL_k\|^2\1_{\tau_{k_0}>k}\\
\stackrel{\mathclap{\eqref{snequ:7}}}{\leq} \cbr{1-2(1-\epsilon)\baralpha_k}\|\bz_k\|^2\1_{\tau_{k_0}>k+1} + \rbr{0.5\epsilon^3+\frac{1}{\epsilon^3}}\rbr{\frac{\nu_{-1}}{\kappa_{\nabla c}}\alpha_k + \psi \alpha_k^p}\|\bnabla\mL_k - \nabla\mL_k\|^2\1_{\tau_{k_0}>k}.
\end{multline*}
This completes the proof of the first part of the result by taking expectation on both sides and setting $\Upsilon_1$ large enough. For the second part of the result, we apply \eqref{equ:def:tau} and note that, for $k_0\leq k <\tau_{k_0}-1$,
\begin{align}\label{nequ:4}
& \bnabla_{\bx} \mL_{k+1} - \nabla_{\bx} \mL_{k+1} \nonumber\\
& = \barg_{k+1} - \nabla f_{k+1} + (\tG_{k+1} - G_{k+1})^T\blambda_{k+1} = \barg_{k+1} - \nabla f_{k+1} + (\barG_{k+1} - G_{k+1})^T\blambda_{k+1} \nonumber\\
& \stackrel{\mathclap{\eqref{snequ:3}}}{=} \beta_{k+1}(\hnabla F(\bx_{k+1}; \xi_{k+1}) - \nabla f_{k+1}) + (1-\beta_{k+1})(\barg_k - \nabla f_k) + (1-\beta_{k+1})(\nabla f_k - \nabla f_{k+1}) \nonumber\\
& \quad + \cbr{\beta_{k+1}(\hnabla c_{k+1} - G_{k+1}) + (1-\beta_{k+1})(\barG_k - G_k) + (1-\beta_{k+1})(G_k - G_{k+1})}^T\blambda_{k+1} \nonumber\\
& \stackrel{\mathclap{\eqref{aequ:5}}}{=}\; \sum_{i=0}^{k+1}\prod_{j=i+1}^{k+1}(1-\beta_j)\beta_i \rbr{\hnabla F(\bx_i; \xi_i) - \mE[\hnabla F(\bx_i; \xi_i)\mid \mF_{i-1}]} \nonumber\\
& \quad + \sum_{i=0}^{k+1}\prod_{j=i+1}^{k+1}(1-\beta_j)\beta_i\rbr{\mE[\hnabla F(\bx_i; \xi_i)\mid \mF_{i-1}] - \nabla f_i} + \sum_{i=0}^{k+1}\prod_{j=i}^{k+1}(1-\beta_j)(\nabla f_{i-1} - \nabla f_i) \nonumber\\
& \quad + \sum_{i=0}^{k+1}\prod_{j=i+1}^{k+1}(1-\beta_j)\beta_i(\hnabla c_i - \mE[\hnabla c_i\mid \mF_{i-1}])^T\blambda_{k+1} + \sum_{i=0}^{k+1}\prod_{j=i+1}^{k+1}(1-\beta_j)\beta_i(\mE[\hnabla c_i\mid \mF_{i-1}] - G_i)^T\blambda_{k+1} \nonumber\\
& \quad + \sum_{i=0}^{k+1}\prod_{j=i}^{k+1}(1-\beta_j)(G_{i-1} - G_i)^T\blambda_{k+1} \eqqcolon \J_1^k + \J_2^k + \J_3^k + \J_4^k + \J_5^k + \J_6^k.
\end{align}
We provide the upper bounds for the terms $\J_1^k$, $\J_2^k$, $\J_3^k$, while the terms $\J_4^k, \J_5^k, \J_6^k$ can be proved~in the same way by noting that $\|\blambda_{k+1}\|^2\1_{\tau_{k_0}>k+1}\leq 1/\epsilon$. For $\J_1^k$, we apply Lemma \ref{technical lemma:2} and have
\begin{align}\label{nequ:5}
\mE[\|\J_1^k\|^2\1_{\tau_{k_0}>k+1}] & \leq \mE[\|\J_1^k\|^2] = \sum_{i=0}^{k+1}\prod_{j=i+1}^{k+1}(1-\beta_j)^2\beta_i^2\mE[\|\hnabla F(\bx_i; \xi_i) - \mE[\hnabla F(\bx_i; \xi_i)\mid \mF_{i-1}]\|^2] \nonumber\\
& \stackrel{\mathclap{\eqref{aequ:B5}}}{\leq} \;  O\rbr{\sum_{i=0}^{k+1}\prod_{j=i+1}^{k+1}(1-\beta_j)^2\beta_i^2}  = O(\beta_k).
\end{align}
For $\J_2^k$, we apply Lemmas \ref{lemma:finite difference for gradient} and \ref{technical lemma:2} and have
\begin{equation}\label{nequ:9}
\mE[\|\J_2^k\|^2\1_{\tau_{k_0}>k+1}]\leq \mE[\|\J_2^k\|^2] = O\rbr{\cbr{\sum_{i=0}^{k+1}\prod_{j=i+1}^{k+1}(1-\beta_j)\beta_ib_i^2}^2} = O(b_k^4).
\end{equation}
For $\J_3^k$, we have
\begin{align}\label{nequ:6}
\|\J_3^k\|^2\1_{\tau_{k_0}>k+1} & \leq \nbr{\sum_{i=0}^{k+1}\prod_{j=i}^{k+1}(1-\beta_j)(\nabla f_{i-1} - \nabla f_i)}^2\1_{\tau_{k_0}>k+1} \nonumber\\
& \stackrel{\mathclap{\eqref{aequ:4}}}{\leq} \kappa_{\nabla f}^2\rbr{\sum_{i=0}^{k+1}\prod_{j=i}^{k+1}(1-\beta_j)\|\bx_i - \bx_{i-1}\|}^2\1_{\tau_{k_0}>k+1} \quad (\text{Lipschitz continuity})\nonumber \\
& = \kappa_{\nabla f}^2\rbr{\sum_{i=0}^{k+1}\prod_{j=i}^{k+1}(1-\beta_j)\baralpha_{i-1}\|\tDelta\bx_{i-1}\|}^2\1_{\tau_{k_0}>k+1}.
\end{align}
We separate the sum on the right-hand side into two parts, $i=0$ to $k_0$ and $i=k_0+1$ to $k+1$. In particular, for the first part, there exists a constant $\Upsilon_2(k_0)>0$ depending on $k_0$ such that
\begin{align}\label{nequ:7}
& \mE\sbr{\rbr{\sum_{i=0}^{k_0}\prod_{j=i}^{k+1}(1-\beta_j)\baralpha_{i-1}\|\tDelta\bx_{i-1}\|}^2\1_{\tau_{k_0}>k+1}}\stackrel{\eqref{aequ:B9}}{\leq} \Upsilon_2(k_0)\prod_{j=k_0}^{k+1}(1-\beta_j)^2 \leq \Upsilon_2(k_0)\exp\rbr{-2\sum_{j = k_0}^{k+1}\beta_j} \nonumber\\
& \leq \Upsilon_2(k_0)\exp\rbr{-\int_{k_0}^{k+2}\frac{2\iota_2}{(j+1)^{p_2}}dj} \leq \Upsilon_2(k_0)\exp\rbr{\frac{2\iota_2(k_0+1)^{1-p_2}}{1-p_2}}\exp\rbr{-\frac{2\iota_2k^{1-p_2}}{1-p_2}}.
\end{align}
For the second part, there exists a constant $\Upsilon_3>0$ such that
\begin{align}\label{nequ:8}
& \mE\sbr{\rbr{\sum_{i=k_0+1}^{k+1}\prod_{j=i}^{k+1}(1-\beta_j)\baralpha_{i-1}\|\tDelta\bx_{i-1}\|}^2\1_{\tau_{k_0}>k+1}} = \mE\sbr{\rbr{\sum_{i=k_0}^{k}\prod_{j=i+1}^{k+1}(1-\beta_j)\baralpha_i\|\tDelta\bx_i\|}^2\1_{\tau_{k_0}>k+1}} \nonumber \\
& \leq \mE\sbr{\rbr{\sum_{i=k_0}^{k}\prod_{j=i+1}^{k}(1-\beta_j)\baralpha_i\|\tDelta\bx_i\|}^2\1_{\tau_{k_0}>k+1}} \stackrel{\substack{\eqref{snequ:7}\\\eqref{equ:def:tau}}}{\leq} \frac{\Upsilon_3}{\epsilon^2}\mE\sbr{\rbr{\sum_{i=k_0}^{k}\prod_{j=i+1}^{k}(1-\beta_j)\alpha_i\|\bnabla\mL_i\|}^2\1_{\tau_{k_0}>k+1}} \nonumber\\
& \leq \frac{\Upsilon_3}{\epsilon^2}\mE\sbr{\rbr{\sum_{i=k_0}^{k}\prod_{j=i+1}^{k}(1-\beta_j)\alpha_i\|\bnabla\mL_i\|\1_{\tau_{k_0}>i}}^2}\leq \frac{\Upsilon_3}{\epsilon^2}\rbr{\sum_{i=k_0}^{k}\prod_{j=i+1}^{k}(1-\beta_j)\alpha_i \cbr{\mE[\|\bnabla\mL_i\|^2\1_{\tau_{k_0}>i}]}^{1/2} }^2 \nonumber\\
& \leq \frac{2\Upsilon_3}{\epsilon^2}\rbr{\sum_{i=k_0}^{k}\prod_{j=i+1}^{k}(1-\beta_j)\alpha_i \cbr{\mE[(\|\bnabla\mL_i-\nabla\mL_i\|^2 + \|\nabla\mL_i\|^2)\1_{\tau_{k_0}>i}]}^{1/2} }^2 \nonumber\\
& \stackrel{\mathclap{\eqref{equ:def:tau}}}{\leq}\; \frac{2\Upsilon_3}{\epsilon^4}\rbr{\sum_{i=k_0}^{k}\prod_{j=i+1}^{k}(1-\beta_j)\alpha_i \cbr{\mE[(\|\bnabla\mL_i-\nabla\mL_i\|^2 + \|\bz_i\|^2)\1_{\tau_{k_0}>i}]}^{1/2} }^2.
\end{align}
Combining \eqref{nequ:4}, \eqref{nequ:5}, \eqref{nequ:9}, \eqref{nequ:6}, \eqref{nequ:7}, \eqref{nequ:8}, and noting that $\|\bnabla\mL_{k+1} - \nabla\mL_{k+1}\| = \|\bnabla_{\bx}\mL_{k+1} - \nabla_{\bx}\mL_{k+1}\|$, we complete the proof of the second part of the result.

\subsection{Proof of Lemma \ref{lem:2}}\label{pf:lem:2}

We prove the statement by induction. Recall that $\epsilon\in(0, 1 - 0.5/(\zeta\iota_{1})\1_{p_1=1})$ is fixed and we denote~$\tau_{k_0} = \tau_{k_0}(\epsilon)$. We have $\mE[\|\bz_k\|^2 \1_{\tau_{k_0}>k}]\leq \epsilon^4$ and
\begin{multline*}
\mE\sbr{\|\bnabla\mL_k - \nabla\mL_k\|^2\1_{\tau_{k_0}>k}} = \mE\sbr{\|\bnabla_{\bx}\mL_k - \nabla_{\bx}\mL_k\|^2\1_{\tau_{k_0}>k}} \\
\leq 2\rbr{\mE\sbr{\|\bnabla_{\bx}\mL_k\|^2\1_{\tau_{k_0}>k}} + \mE\sbr{\|\nabla_{\bx}\mL_k\|^2\1_{\tau_{k_0}>k}}} \stackrel{\eqref{equ:def:tau}}{\leq} 2\rbr{\mE\sbr{\|\bnabla_{\bx}\mL_k\|^2\1_{\tau_{k_0}>k}} + \epsilon^2}.
\end{multline*}
Thus, to prove the result for $q=0$, it suffices to show $\mE[\|\bnabla_{\bx}\mL_k\|^2\1_{\tau_{k_0}>k}]$ is upper bounded. In fact,~we note that
\begin{equation*}
\mE\sbr{\|\bnabla_{\bx}\mL_k\|^2\1_{\tau_{k_0}>k}} = \mE\sbr{\|\barg_k + \barG_k^T\blambda_k\|^2\1_{\tau_{k_0}>k}} \leq 2\rbr{\mE\sbr{\|\barg_k\|^2} + \frac{1}{\epsilon^2}\mE\sbr{\|\barG_k\|^2}}.
\end{equation*}
By \eqref{aequ:B8} we know $\mE[\|\barg_k\|^2]\leq \Upsilon_{\barg}$ for all $k\geq 0$ while the term $\mE[\|\barG_k\|_2^2]$ can be proved in the same way. Thus, combining the above two displays, we know the~result holds for $q=0$. Suppose the result holds for $q\geq 0$, we aim to establish the result for $q+1$. We apply Lemma \ref{lem:1} and obtain for some~constants $\Upsilon_1(k_0), \Upsilon_2(k_0), \Upsilon_3(k_0)>0$ that for any $k\geq k_0$,
\begin{align}\label{nequ:11}
&\mE\sbr{\|\bnabla\mL_{k+1} - \nabla\mL_{k+1}\|^2\1_{\tau_{k_0}>k+1}} \nonumber\\ 
&\hskip1.8cm \leq \Upsilon_1(k_0)\rbr{\beta_k+b_k^4 + \cbr{\sum_{i=k_0}^{k}\prod_{j=i+1}^{k}(1-\beta_j)\alpha_i\rbr{\sqrt{\beta_i}+b_i^2+\rbr{\frac{\alpha_i}{\beta_i}}^q} }^2 } \nonumber\\
&\hskip1.8cm \leq \Upsilon_1(k_0)\rbr{\beta_k+b_k^4 + \cbr{\sum_{i=0}^{k}\prod_{j=i+1}^{k}(1-\beta_j)\alpha_i\rbr{\sqrt{\beta_i}+b_i^2+\rbr{\frac{\alpha_i}{\beta_i}}^q} }^2 } \nonumber\\
& \hskip1.8cm \leq \Upsilon_2(k_0)\rbr{\beta_k+b_k^4+\frac{\alpha_k^2}{\beta_k^2}\cbr{\beta_k + b_k^4 + \rbr{\frac{\alpha_k}{\beta_k}}^{2q} } } \quad\text{ (Lemma \ref{technical lemma:2})} \nonumber\\
& \hskip1.8cm \leq \Upsilon_3(k_0)\rbr{\beta_k+b_k^4+\rbr{\alpha_k/\beta_k}^{2(q+1)} } .
\end{align}
In addition, by Lemma \ref{lem:1}, we also have for some constant $\Upsilon_4>0$ such that for any $k\geq k_0$,
\begin{align*}
& \mE\sbr{\|\bz_{k+1}\|^2 \1_{\tau_{k_0}>k+1}} \\
& \leq \mE\sbr{\cbr{1 - 2(1-\epsilon)\baralpha_k} \|\bz_k\|^2\1_{\tau_{k_0}>k+1}}  + \Upsilon_4\alpha_k\mE\sbr{\|\bnabla\mL_k - \nabla\mL_k\|^2\1_{\tau_{k_0}>k}}\\
& \leq \mE\sbr{\cbr{1 - \frac{2(1-\epsilon)\nu_k\alpha_k}{\tau_k\kappa_{\nabla f} + \kappa_{\nabla c}}} \|\bz_k\|^2\1_{\tau_{k_0}>k+1}} + \Upsilon_4\alpha_k\mE\sbr{\|\bnabla\mL_k - \nabla\mL_k\|^2\1_{\tau_{k_0}>k}}\\
& \leq \cbr{1 - 2(1-\epsilon)\zeta\alpha_k}\mE\sbr{\|\bz_k\|^2\1_{\tau_{k_0}>k}}  + \Upsilon_4\alpha_k\mE\sbr{\|\bnabla\mL_k - \nabla\mL_k\|^2\1_{\tau_{k_0}>k}},
\end{align*}
where the last inequality uses the fact that $2(1-\epsilon)\zeta\alpha_k<1$ (it holds for $k$ large enough with deterministic threshold index). Applying the above inequality recursively with the bound in \eqref{nequ:11}, we know for some constant $\Upsilon_5(k_0)>0$,
\begin{align*}
\mE\sbr{\|\bz_{k+1}\|^2 \1_{\tau_{k_0}>k+1}} & \leq \Upsilon_5(k_0)\sum_{i=k_0}^{k}\prod_{j=i+1}^{k}\cbr{1-2\zeta(1-\epsilon)\alpha_j}\alpha_i\rbr{\beta_i+b_i^4+\rbr{\frac{\alpha_i}{\beta_i}}^{2(q+1)}}\\
& \leq \Upsilon_5(k_0)\sum_{i=0}^{k}\prod_{j=i+1}^{k}\cbr{1-2\zeta(1-\epsilon)\alpha_j}\alpha_i\rbr{\beta_i+b_i^4+\rbr{\frac{\alpha_i}{\beta_i}}^{2(q+1)}}.
\end{align*}
By Lemma \ref{technical lemma:2} and the condition $2\zeta\iota_1(1-\epsilon)>1$ when $p_1=1$, we know
\begin{equation*}
\sum_{i=0}^{k}\prod_{j=i+1}^{k}\cbr{1-2\zeta(1-\epsilon)\alpha_j}\alpha_i \beta_i = O(\beta_k).
\end{equation*}
Without loss of generality, we suppose $\beta_k = o(b_k^4+(\alpha_k/\beta_k)^{2(q+1)})$; otherwise the result is trivial.~Then, Lemma \ref{technical lemma:2} also leads to
\begin{equation*}
\sum_{i=0}^{k}\prod_{j=i+1}^{k}\cbr{1-2\zeta(1-\epsilon)\alpha_j}\alpha_i \rbr{b_i^4+\rbr{\alpha_i/\beta_i}^{2(q+1)}} = O\rbr{b_k^4 + \rbr{\alpha_k/\beta_k}^{2(q+1)}}.
\end{equation*}
Combining the above three displays, we obtain
\begin{equation}\label{nequ:12}
\mE\sbr{\|\bz_{k+1}\|^2 \1_{\tau_{k_0}>k+1}} \leq \Upsilon_6(k_0)\rbr{\beta_k+b_k^4+\rbr{\alpha_k/\beta_k}^{2(q+1)} }.
\end{equation}
Combining \eqref{nequ:11} and \eqref{nequ:12}, we prove that the result holds for $q+1$. This completes the induction step and concludes the proof.

\subsection{Proof of Lemma \ref{lem:local:rate:Hessian}}\label{pf:lem:local:rate:Hessian}

For notational conciseness, we follow Appendix \ref{pf:lem:local:rate} and use $\Upsilon_1, \Upsilon_2,\ldots$ to denote generic deterministic constants. We note that
\begin{multline}\label{nequ:20}
\|\tW_k - W^\star\|^2\1_{\tau_{k_0}>k} \stackrel{\eqref{def:acc2}, \eqref{equ:def:tau}}{=} \nbr{\begin{pmatrix}
\barB_k-\nabla_{\bx}\mL^\star & (\barG_k-G^\star)^T\\
\barG_k - G^\star & \0
\end{pmatrix}}^2\1_{\tau_{k_0}>k} \\ 
\leq 2\|\barB_k-\nabla_{\bx}\mL^\star\|^2\1_{\tau_{k_0}>k} + 2\|\barG_k - G^\star\|^2\1_{\tau_{k_0}>k}.
\end{multline}
We bound $\|\barB_k-\nabla_{\bx}\mL^\star\|^2$ as an example, while the bound of $\|\barG_k-G^\star\|^2$ can be derived in the same~way with only fewer terms, resulting in the same upper bound. We have
\begin{align*}
\barB_k & =  (1-\beta_k)\barB_{k-1} + \beta_k\hnabla_{\bx}^2\mL_k = (1-\beta_k)\barB_{k-1} + \beta_k\rbr{\hnabla^2 F(\bx_k;\xi_k)+\sum_{j=1}^{m}\blambda_k^j\hnabla^2 c^j_k}\\
& = \sum_{h=0}^{k}\prod_{l=h+1}^{k}(1-\beta_l)\beta_h\rbr{\hnabla^2 F(\bx_h;\xi_h)+\sum_{j=1}^{m}\blambda_h^j\hnabla^2 c^j_h} + \prod_{h=0}^{k}(1-\beta_h)\barB_{-1}.
\end{align*}
With the above expression, we have a similar decomposition to \eqref{cequ:1} and obtain
\begin{align*}
& \barB_k - \nabla_{\bx}^2\mL^\star = \sum_{h=0}^{k}\prod_{l=h+1}^{k}(1-\beta_l)\beta_h\rbr{\hnabla^2 F(\bx_h;\xi_h)-\mE[\hnabla^2 F(\bx_h;\xi_h)\mid \mF_{h-1}]} \\
& + \sum_{h=0}^{k}\prod_{l=h+1}^{k}(1-\beta_l)\beta_h\rbr{\mE[\hnabla^2 F(\bx_h;\xi_h)\mid \mF_{h-1}]- \nabla^2 f_h}  + \sum_{h=0}^{k}\prod_{l=h+1}^{k}(1-\beta_l)\beta_h(\nabla^2 f_h - \nabla^2f^\star) \\
& + \sum_{h=0}^{k}\prod_{l=h+1}^{k}(1-\beta_l)\beta_h \rbr{\sum_{j=1}^{m}\blambda_h^j - (\blambda^\star)^j}\hnabla^2 c^j_h  + \sum_{h=0}^{k}\prod_{l=h+1}^{k}(1-\beta_l)\beta_h \sum_{j=1}^{m}(\blambda^\star)^j \rbr{\hnabla^2 c^j_h - \mE[\hnabla^2 c^j_h\mid \mF_{h-1}]} \\
& + \sum_{h=0}^{k}\prod_{l=h+1}^{k}(1-\beta_l)\beta_h \sum_{j=1}^{m}(\blambda^\star)^j \rbr{\mE[\hnabla^2 c^j_h\mid \mF_{h-1}] - \nabla^2c^j_h}  + \sum_{h=0}^{k}\prod_{l=h+1}^{k}(1-\beta_l)\beta_h \sum_{j=1}^{m}(\blambda^\star)^j \rbr{\nabla^2c^j_h - (\nabla^2 c^j)^\star} \\
& + \prod_{h=0}^{k}(1-\beta_h)(\barB_{-1}-\nabla_{\bx}^2\mL^\star) \eqqcolon \K_1^k + \K_2^k + \K_3^k + \K_4^k + \K_5^k + \K_6^k + \K_7^k + \K_8^k.
\end{align*}
We establish the bounds for $\K_1^k$, $\K_2^k$, $\K_3^k$, $\K_4^k$, while the bounds of $\K_5^k$, $\K_6^k$, $\K_7^k$ can be derived similarly to those of $\K_1^k$, $\K_2^k$, $\K_3^k$, and $\|\K_8^k\|^2 = O(\prod_{h=0}^{k}(1-\beta_h)^2)\leq \exp(-2\sum_{h=0}^{k}\beta_h) = o(\beta_k)$ by \eqref{nequ:7} only contributes to the higher-order error. 

\vskip4pt
\noindent$\bullet$ For $\K_1^k$, we know from the proof of Lemma \ref{lem:Hessian Convergence} in Appendix~\ref{appendix Hessian Convergence} that there exist $0<\delta'<\delta$, a deterministic threshold $K'>0$, and a constant $\Upsilon_1>0$ such that for any $k\geq K'$ and any $\bx_k\in\mX_{\delta'}\coloneqq \mX\cap\{\bx:\|\bx-\tx\|\leq \delta'\}$, we have $\mE[\|\hnabla^2 F(\bx_k;\xi_k)\|^2\mid \mF_{k-1}]\leq \Upsilon_1$. With this property,~we~separate~$\K_1^k$ into three terms:
\begin{align*}
\K_1^k & = \sum_{h=0}^{k}\prod_{l=h+1}^{k}(1-\beta_l)\beta_h\rbr{\hnabla^2 F(\bx_h;\xi_h)-\mE[\hnabla^2 F(\bx_h;\xi_h)\mid \mF_{h-1}]}\1_{\bx_h\notin\mX_{\delta'}} \\
& \quad + \sum_{h=0}^{K'-1}\prod_{l=h+1}^{k}(1-\beta_l)\beta_h\rbr{\hnabla^2 F(\bx_h;\xi_h)-\mE[\hnabla^2 F(\bx_h;\xi_h)\mid \mF_{h-1}]}\1_{\bx_h\in\mX_{\delta'}}\\
& \quad + \sum_{h=K'}^{k}\prod_{l=h+1}^{k}(1-\beta_l)\beta_h\rbr{\hnabla^2 F(\bx_h;\xi_h)-\mE[\hnabla^2 F(\bx_h;\xi_h)\mid \mF_{h-1}]}\1_{\bx_h\in\mX_{\delta'}} \eqqcolon \K_{1,1}^k + \K_{1,2}^k + \K_{1,3}^k.
\end{align*}
For $\K_{1,1}^k$, since $\bx_h\in \mX$ and $\bx_h\rightarrow\tx$ as $h\rightarrow\infty$ almost surely (cf. Assumption \ref{ass:4.1}), we know for any run of the sequence~$\{\bx_h\}$, there exist a (potentially random) $\tilde{h}<\infty$ and a constant $\Upsilon_2(\tilde{h})>0$~such~that
\begin{align*}
\|\K_{1,1}^k\| & = \nbr{\sum_{h=0}^{\tilde{h}}\prod_{l=h+1}^{k}(1-\beta_l)\beta_h\rbr{\hnabla^2 F(\bx_h;\xi_h)-\mE[\hnabla^2 F(\bx_h;\xi_h)\mid \mF_{h-1}]}\1_{\bx_h\notin\mX_{\delta'}} }\\
& \leq \sum_{h=0}^{\tilde{h}}\prod_{l=h+1}^{k}(1-\beta_l)\beta_h\nbr{\hnabla^2 F(\bx_h;\xi_h)-\mE[\hnabla^2 F(\bx_h;\xi_h)\mid \mF_{h-1}]}\1_{\bx_h\notin\mX_{\delta'}}\\
& = \sum_{h=0}^{\tilde{h}}\prod_{l=h+1}^{\tilde{h}}(1-\beta_l)\beta_h\nbr{\hnabla^2 F(\bx_h;\xi_h)-\mE[\hnabla^2 F(\bx_h;\xi_h)\mid \mF_{h-1}]}\1_{\bx_h\notin\mX_{\delta'}}\prod_{l=\tilde{h}+1}^{k}(1-\beta_l)\\
& \stackrel{\mathclap{\eqref{nequ:7}}}{\leq}\; \Upsilon_2(\tilde{h})\exp\rbr{-\frac{\iota_{2}k^{1-p_2}}{1-p_2}}.
\end{align*}
This implies that
\begin{equation*}
P\rbr{\bigcap_{M=0}^\infty\bigcap_{K=0}^\infty\mA_{M, K}}\coloneqq P\rbr{\bigcap_{M=0}^\infty\bigcap_{K=0}^\infty\bigcup_{k\geq K}\cbr{k\|\K_{1,1}^k\| \geq M}} = 0.
\end{equation*}
Since $\mA_{M+1,K+1}\subseteq\mA_{M,K}$, we have $\lim_{M\rightarrow\infty,K\rightarrow\infty}P(\mA_{M,K}) = 0$. Thus, for any $\epsilon>0$ there exist $M(\epsilon)$ and $K(\epsilon)$ such that for any $k\geq K(\epsilon)$, $P(k\|\K_{1,1}^k\| \geq M(\epsilon))\leq \epsilon$. This means that $\|\K_{1,1}^k\| = O_p(1/k)$. Following the same analysis, we also obtain $\|\K_{1,2}^k\|=O_p(1/k)$. For $\K_{1,3}^k$, we apply the martingale difference property (noting that $\1_{\bx_h\in\mX_{\delta'}}$ is $\mF_{h-1}$-measurable) and the bounded second moment condition, and obtain
\begin{equation*}
\mE[\|\K_{1,3}^k\|^2] \leq O\rbr{\sum_{h=K'}^{k}\prod_{l=h+1}^{k}(1-\beta_l)^2\beta_h^2}\leq O\rbr{\sum_{h=0}^{k}\prod_{l=h+1}^{k}(1-\beta_l)^2\beta_h^2} = O(\beta_k),
\end{equation*}
where the last equality is due to Lemma \ref{technical lemma:2}. Combining the results of $\K_{1,1}^k$, $\K_{1,2}^k$, $\K_{1,3}^k$, we have
\begin{equation}\label{nequ:16}
\|\K_1^k\|^2\1_{\tau_{k_0}>k} \leq \|\K_1^k\|^2 = O_p(\beta_k).
\end{equation}
\noindent$\bullet$ For $\K_2^k$, we use $p_4>0.5p_3$, apply Lemmas \ref{lemma:finite difference for gradient} and \ref{technical lemma:2}, and have
\begin{equation}\label{nequ:17}
\mE[\|\K_2^k\|^2\1_{\tau_{k_0}>k}]\leq O\rbr{\cbr{\sum_{h=0}^{k}\prod_{l=h+1}^{k}(1-\beta_l)\beta_h(b_h + \tilde{b}_h^2/b_h)}^2} = O(b_k^2 + \tilde{b}_k^4/b_k^2).
\end{equation}
\noindent$\bullet$ For $\K_3^k$, we have
\begin{align*}
& \|\K_3^k\| \1_{\tau_{k_0}>k} \leq \sum_{h=0}^{k}\prod_{l=h+1}^{k}(1-\beta_l)\beta_h\|\nabla^2 f_h - \nabla^2f^\star\| \cdot \1_{\tau_{k_0}>k}\\
& \leq \sum_{h=0}^{k_0-1}\prod_{l=h+1}^{k}(1-\beta_l)\beta_h\|\nabla^2 f_h - \nabla^2f^\star\| + \sum_{h=k_0}^{k}\prod_{l=h+1}^{k}(1-\beta_l)\beta_h\|\nabla^2 f_h - \nabla^2f^\star\| \cdot \1_{\tau_{k_0}>h} \eqqcolon \K_{3,1}^k + \K_{3,2}^k.
\end{align*}
By the same analysis as in $\K_{1,1}^k$, we know $\K_{3,1}^k = O_p(1/k)$. For $\J_{3,2}^k$, we apply the Lipschitz continuity condition and Lemmas \ref{lem:local:rate} and \ref{technical lemma:2}, and have for some constants $\Upsilon_3>0, \Upsilon_4(k_0)>0$, $\Upsilon_5(k_0)>0$,
\begin{align*}
\mE[(\K_{3,2}^k)^2] & \leq \Upsilon_3 \rbr{\sum_{h=k_0}^{k}\prod_{l=h+1}^{k}(1-\beta_l)\beta_h\{\mE[\|\bx_h-\tx\|^2\1_{\tau_{k_0}>h}]\}^{1/2}}^2\\
& \leq \Upsilon_4(k_0)\rbr{\sum_{h=k_0}^{k}\prod_{l=h+1}^{k}(1-\beta_l)\beta_h(\sqrt{\beta_h}+b_h^2)}^2 \leq \Upsilon_5(k_0)(\beta_k+b_k^4).
\end{align*}
Combining the above two displays, we have
\begin{equation}\label{nequ:18}
\|\K_3^k\|^2\1_{\tau_{k_0}>k} = O_p(\beta_k+b_k^4).
\end{equation}
\noindent$\bullet$ For $\K_4^k$, we apply \eqref{cequ:5} and follow the same analysis as $\J_3^k$. We obtain for some constant $\Upsilon_6>0$~that 
\begin{align}\label{nequ:19}
\|\K_4^k\|^2\1_{\tau_{k_0}>k} & \leq \Upsilon_6\cbr{\sum_{h=0}^{k}\prod_{l=h+1}^{k}(1-\beta_l)\beta_h\|\blambda_h-\tlambda\|\1_{\tau_{k_0}>k}}^2 \nonumber \\
& \leq \Upsilon_6\cbr{\rbr{\sum_{h=0}^{k_0-1} + \sum_{h=k_0}^{k}}\prod_{l=h+1}^{k}(1-\beta_l)\beta_h\|\blambda_h-\tlambda\|\1_{\tau_{k_0}>h}}^2 =  O_p\rbr{\beta_k+b_k^4}.
\end{align}
Combining \eqref{nequ:16}, \eqref{nequ:17}, \eqref{nequ:18}, \eqref{nequ:19}, ignoring higher-order error terms, and establishing the same bounds for $\J_5^k$, $\J_6^k$, $\J_7^k$, we obtain
\begin{equation*}
\|\barB_k-\nabla_{\bx}\mL^\star\|^2\1_{\tau_{k_0}>k} = O_p\rbr{\beta_k+b_k^2 + \tilde{b}_k^4/b_k^2}.
\end{equation*}
Following the same analysis, we can derive the same bound for $\|\barG_k-G^\star\|^2$. Plugging into \eqref{nequ:20},~we complete the proof.

\subsection{Proof of Theorem \ref{thm:normality}}\label{pf:thm:normality}

To streamline the proof, we first present a generic lemma, which is proved in Appendix \ref{pf:lem:4}. We will apply this lemma to various terms that appear throughout the proof.

\begin{lemma}\label{lem:4}
Consider a sequence of random variables $\{X_k\}_{k=0}^\infty$ and a sequence of events $\{\mA_k\}_{k=0}^\infty$. Let $\tau_{k_0} = \inf\{k\geq k_0: \mA_k \text{ happens}\}$ be the first index $k$ after $k_0$ such that $\mA_k$ happens. Suppose that for each realization of the sequence, there exists a (potentially random) $\tilde{k}_0<\infty$ such that $\tau_{\tilde{k}_0} = \infty$ (in other words, $\mA_k$ will finally not happen almost surely). 
Also, for the sequence $\alpha_k = \iota_1/(k+1)^{p_1}$ with $p_1\in(0,1]$, suppose there exists a deterministic $\bar{k}_0>0$ such that for any fixed $k_0\geq \bar{k}_0$, $X_k\1_{\tau_{k_0}>k} = o_p(\sqrt{\alpha_k})$.~Then, for any constant $\zeta>0$ satisfying $\zeta\iota_1>0.5$ when $p_1=1$, we have 
\begin{equation*}
\sum_{i=0}^{k}\prod_{j=i+1}^{k}(1-\zeta\alpha_j)\alpha_iX_i = o_p(\sqrt{\alpha_k}).
\end{equation*}	
	
\end{lemma}

We first decompose the primal-dual error term of Algorithm \ref{Alg:DF-SSQP}. We have
\begin{align*}
& \bz_{k+1} = \bz_k - \baralpha_k \tW_k^{-1}\bnabla \mL_k  = \bz_k - \zeta\alpha_k \tW_k^{-1}\bnabla \mL_k - (\baralpha_k-\zeta\alpha_k)\tW_k^{-1}\bnabla \mL_k \\
& = (1 - \zeta\alpha_k) \bz_k -\zeta\alpha_k\tW_k^{-1}(\nabla\mL_k - \tW_k\bz_k) - \zeta\alpha_k\tW_k^{-1}(\bnabla\mL_k - \nabla\mL_k) - (\baralpha_k-\zeta\alpha_k)\tW_k^{-1}\bnabla \mL_k\\
& = (1 - \zeta\alpha_k) \bz_k - \zeta\alpha_k\tW_k^{-1}(\nabla\mL_k - W^\star\bz_k) - \zeta\alpha_k\tW_k^{-1}(W^\star-\tW_k)\bz_k - \zeta\alpha_k(W^\star)^{-1}(\bnabla\mL_k - \nabla\mL_k)\\
& \quad  - \zeta\alpha_k(\tW_k^{-1} - (W^\star)^{-1})(\bnabla\mL_k - \nabla\mL_k) - (\baralpha_k-\zeta\alpha_k)\tW_k^{-1}\bnabla \mL_k\\
& = \prod_{i=0}^{k}(1-\zeta\alpha_i)\bz_0 - \sum_{i=0}^{k}\prod_{j=i+1}^{k}(1-\zeta\alpha_j)\zeta\alpha_i\cbr{\tW_i^{-1}(\nabla\mL_i - W^\star\bz_i) + \tW_i^{-1}(W^\star-\tW_i)\bz_i }\\
& \quad - \sum_{i=0}^{k}\prod_{j=i+1}^{k}(1-\zeta\alpha_j)\zeta\alpha_i\cbr{(\tW_i^{-1} - (W^\star)^{-1})(\bnabla\mL_i - \nabla\mL_i) + \frac{\baralpha_i-\zeta\alpha_i}{\zeta\alpha_i}\tW_i^{-1}\bnabla \mL_i}\\
& \quad - \sum_{i=0}^{k}\prod_{j=i+1}^{k}(1-\zeta\alpha_j)\zeta\alpha_i(W^\star)^{-1}(\bnabla\mL_i - \nabla\mL_i) \eqqcolon \C_1^k - \C_2^k - \C_3^k - \C_4^k.
\end{align*}
In the following proof, we choose the (deterministic) constant $\epsilon$ to be sufficiently small such that, for each run of the algorithm, there exists a (potentially random) $\tilde{k}_0 < \infty$ satisfying $\tau_{\tilde{k}_0}(\epsilon) = \infty$, where~$\tau_{k_0}(\epsilon)$ is defined in \eqref{equ:def:tau}. 
This $\epsilon$ exists because, for each run of the algorithm:
\begin{enumerate}[label=(\alph*),topsep=0pt]
\setlength\itemsep{0.0em}
\item $\|\bz_k\|>\epsilon^2$ and $\|(\bx_k, \blambda_k)\|>1/\epsilon$ will finally not happen since \eqref{cond:4.8} implies \eqref{cond:4.4}, and Lemma \ref{iterate convergence} shows that $\bz_k\rightarrow0$ as $k\rightarrow\infty$ almost surely.
\item $\|\tilde{W}_k^{-1}\|>1/\epsilon$, $\delta_k^G\neq \0$, and $\delta_k^B\neq \0$ will finally not happen since \eqref{cond:4.8} implies \eqref{cond:3.6} and \eqref{cond:4.5},~and Lemmas \ref{lemma:average almost sure} and \ref{lem:Hessian Convergence} show that $\tilde{W}_k\rightarrow W^\star=\nabla^2\mL^\star$ as $k\rightarrow\infty$ almost surely.
\item $\|\nabla\mL_k - \tilde{W}_k\bz_k\|>0.25\epsilon^2\|\bz_k\|$, $\|\nabla\mL_k\|>\|\bz_k\|/\epsilon$, and $\|\nabla\mL_k - W^\star\bz_k\|>\|\bz_k\|^2/\epsilon$ will~\mbox{finally}~not~happen since $\nabla^2\mL$ is Lipschitz continuous near $(\tx,\tlambda)$ by Assumption \ref{ass:1-1} and $\tilde{W}_k\rightarrow W^\star$.
\item $\nu_k/(\tau_k\kappa_{\nabla f}+\kappa_{\nabla c})\neq \zeta$ will finally not happen by Assumption \ref{ass:4.3}.
\end{enumerate}
\vskip4pt
\noindent$\bullet$ For $\C_1^k$, we follow \eqref{nequ:7}, apply $\zeta\iota_{1}>0.5$ when $p_1=1$, and have $\C_1^k = o(\sqrt{\alpha_k})$.
\vskip4pt
\noindent$\bullet$ For $\C_2^k$, we apply Lemmas \ref{lem:local:rate} and \ref{lem:local:rate:Hessian}, and have for $k\geq k_0$,
\begin{align*}
\|\tW_k^{-1}(\nabla\mL_k - W^\star\bz_k) & + \tW_k^{-1}(W^\star-\tW_k)\bz_k\|\1_{\tau_{k_0}>k}  \stackrel{\mathclap{\eqref{equ:def:tau}}}{\leq} \frac{1}{\epsilon^2}\|\bz_k\|^2 \1_{\tau_{k_0}>k} + \frac{1}{\epsilon}\|\tW_k-W^\star\| \|\bz_k\| \1_{\tau_{k_0}>k} \\
& \leq O_p\rbr{\beta_k+b_k^4} + O_p\rbr{\cbr{\sqrt{\beta_k}+b_k+\tilde{b}_k^2/b_k}\cbr{\sqrt{\beta_k}+b_k^2}}\\
& = O_p\rbr{\beta_k + \sqrt{\beta_k}b_k+b_k^3+\sqrt{\beta_k}\tb_k^2/b_k+\tb_k^2b_k}.
\end{align*}
We note that
\begin{multline*}
O_p\rbr{\beta_k + \sqrt{\beta_k}b_k+b_k^3+\sqrt{\beta_k}\tb_k^2/b_k+\tb_k^2b_k} = o_p(\sqrt{\alpha_k})\\
\Longleftarrow \min\{p_2, 0.5p_2+p_3, 3p_3, 0.5p_2+2p_4-p_3, 2p_4+p_3\}>0.5p_1 \Longleftarrow \eqref{cond:4.8}.
\end{multline*}
Thus, we apply Lemma \ref{lem:4} and have $\C_2^k = o_p(\sqrt{\alpha_k})$.
\vskip4pt
\noindent$\bullet$ For $\C_3^k$, we have
\begin{align*}
& \nbr{(\tW_k^{-1} - (W^\star)^{-1})(\bnabla\mL_k - \nabla\mL_k) + \frac{\baralpha_k-\zeta\alpha_k}{\zeta\alpha_k}\tW_k^{-1}\bnabla \mL_k}\1_{\tau_{k_0}>k}\\
&\stackrel{\mathclap{\eqref{equ:def:tau}}}{\leq} \|\tW_k^{-1} - (W^\star)^{-1}\| \|\bnabla\mL_k - \nabla\mL_k\| \1_{\tau_{k_0}>k} + \frac{\psi\alpha_k^{p-1}}{\epsilon\zeta}\|\bnabla\mL_k\|\1_{\tau_{k_0}>k}\\
& \stackrel{\mathclap{\eqref{equ:def:tau}}}{\leq} \frac{\|(W^\star)^{-1}\|}{\epsilon} \|\tW_k - W^\star\| \|\bnabla\mL_k - \nabla\mL_k\|\1_{\tau_{k_0}>k} + \frac{\psi\alpha_k^{p-1}}{\epsilon\zeta}\cbr{\|\bnabla\mL_k-\nabla\mL_k\|+\frac{\|\bz_k\|}{\epsilon}}\1_{\tau_{k_0}>k}\\
& = O_p\rbr{\cbr{\sqrt{\beta_k}+b_k+\tilde{b}_k^2/b_k}\cbr{\sqrt{\beta_k}+b_k^2}} + O_p\rbr{\alpha_k^{p-1}\rbr{\sqrt{\beta_k}+b_k^2}}\\
& = O_p\rbr{\beta_k + \sqrt{\beta_k}b_k+b_k^3+\sqrt{\beta_k}\tb_k^2/b_k+\tb_k^2b_k + \alpha_k^{p-1}\sqrt{\beta_k} + \alpha_k^{p-1}b_k^2}.
\end{align*}
We note that
\begin{align*}
& O_p\rbr{\beta_k + \sqrt{\beta_k}b_k+b_k^3+\sqrt{\beta_k}\tb_k^2/b_k+\tb_k^2b_k + \alpha_k^{p-1}\sqrt{\beta_k} + \alpha_k^{p-1}b_k^2} = o_p(\sqrt{\alpha_k})\\
& \Longleftarrow \min\{p_2, 0.5p_2+p_3, 3p_3, 0.5p_2+2p_4-p_3, 2p_4+p_3, (p-1)p_1+0.5p_2, (p-1)p_1+2p_3\}>0.5p_1\\
& \Longleftarrow \eqref{cond:4.8}.
\end{align*}
Thus, we apply Lemma \ref{lem:4} again and have $\C_3^k = o_p(\sqrt{\alpha_k})$.
\vskip4pt
\noindent$\bullet$ For $\C_4^k$, let us define $\hnabla_{\bx}\mL(\bx,\blambda;\xi) \coloneqq \hnabla F(\bx;\xi)+\hnabla c(\bx)^T\blambda$. We have
\begin{align*}
& \C_4^k  = \sum_{i=0}^{k}\prod_{j=i+1}^{k}(1-\zeta\alpha_j)\zeta\alpha_i(W^\star)^{-1}(\bnabla\mL_i - \nabla\mL_i) =  \sum_{i=0}^{k}\prod_{j=i+1}^{k}(1-\zeta\alpha_j)\zeta\alpha_i(W^\star)^{-1}\begin{pmatrix}
\bnabla_{\bx}\mL_i - \nabla_{\bx}\mL_i\\
\0
\end{pmatrix}\\
& \stackrel{\mathclap{\eqref{nequ:4}}}{=}\;\; \sum_{i=0}^{k}\prod_{j=i+1}^{k}(1-\zeta\alpha_j)\zeta\alpha_i\cdot\sum_{h=0}^{i}\prod_{l=h}^{i}(1-\beta_l)(W^\star)^{-1}\begin{pmatrix}
\nabla_{\bx}\mL(\bx_{h-1},\blambda_i) - \nabla_{\bx}\mL(\bx_h,\blambda_i)\\
\0
\end{pmatrix}\\
& \quad + \sum_{i=0}^{k}\prod_{j=i+1}^{k}(1-\zeta\alpha_j)\zeta\alpha_i\cdot\sum_{h=0}^{i}\prod_{l=h+1}^{i}(1-\beta_l)\beta_h(W^\star)^{-1}\begin{pmatrix}
\hnabla_{\bx}\mL(\bx_h,\blambda_i;\xi_h) - \nabla_{\bx}\mL(\bx_h,\blambda_i)\\
\0
\end{pmatrix}\\
& = \sum_{i=0}^{k}\prod_{j=i+1}^{k}(1-\zeta\alpha_j)\zeta\alpha_i\cdot\sum_{h=0}^{i}\prod_{l=h}^{i}(1-\beta_l)(W^\star)^{-1}\begin{pmatrix}
\nabla_{\bx}\mL(\bx_{h-1},\blambda_i) - \nabla_{\bx}\mL(\bx_h,\blambda_i)\\
\0
\end{pmatrix}\\
& \quad + \sum_{i=0}^{k}\prod_{j=i+1}^{k}(1-\zeta\alpha_j)\zeta\alpha_i\cdot\sum_{h=0}^{i}\prod_{l=h+1}^{i}(1-\beta_l)\beta_h(W^\star)^{-1}\begin{pmatrix}
(\hnabla c_h - \mE[\hnabla c_h\mid \mF_{h-1}])^T(\blambda_i-\tlambda)\\
\0
\end{pmatrix}\\
& \quad + \sum_{i=0}^{k}\prod_{j=i+1}^{k}(1-\zeta\alpha_j)\zeta\alpha_i\cdot\sum_{h=0}^{i}\prod_{l=h+1}^{i}(1-\beta_l)\beta_h(W^\star)^{-1}\begin{pmatrix}
(\mE[\hnabla c_h\mid \mF_{h-1}] - G_h)^T(\blambda_i-\tlambda)\\
\0
\end{pmatrix}\\
& \quad + \sum_{i=0}^{k}\prod_{j=i+1}^{k}(1-\zeta\alpha_j)\zeta\alpha_i\cdot\sum_{h=0}^{i}\prod_{l=h+1}^{i}(1-\beta_l)\beta_h(W^\star)^{-1}\begin{pmatrix}
\mE[\hnabla_{\bx}\mL(\bx_h,\tlambda;\xi_h)\mid \mF_{h-1}] - \nabla_{\bx}\mL(\bx_h,\tlambda)\\
\0
\end{pmatrix}\\
& \quad + \sum_{i=0}^{k}\prod_{j=i+1}^{k}(1-\zeta\alpha_j)\zeta\alpha_i\cdot\sum_{h=0}^{i}\prod_{l=h+1}^{i}(1-\beta_l)\beta_h(W^\star)^{-1}\begin{pmatrix}
\hnabla_{\bx}\mL(\bx_h,\tlambda;\xi_h) - \mE[\hnabla_{\bx}\mL(\bx_h,\tlambda;\xi_h)\mid \mF_{h-1}]\\
\0
\end{pmatrix}\\
& \eqqcolon \C_{4,1}^k + \C_{4,2}^k + \C_{4,3}^k + \C_{4,4}^k + \C_{4,5}^k.
\end{align*}
\noindent$\bullet$$\bullet$ For $\C_{4,1}^k$, we know from Lemma \ref{lem:4} that it suffices to show
\begin{equation*}
\sum_{h=0}^{k}\prod_{l=h}^{k}(1-\beta_l)(W^\star)^{-1}\begin{pmatrix}
\nabla_{\bx}\mL(\bx_{h-1},\blambda_k) - \nabla_{\bx}\mL(\bx_h,\blambda_k)\\
\0
\end{pmatrix}\cdot\1_{\tau_{k_0}>k} = o_p(\sqrt{\alpha_k}).
\end{equation*}
Since $\|\blambda_k\|\leq 1/\epsilon$ when $\tau_{k_0}>k$, we apply Lemma \ref{lem:4} again and know that the above display is~implied~by
\begin{equation*}
(\|\nabla f_{k-1}-\nabla f_k\| + \|G_{k-1}-G_k\|)\1_{\tau_{k_0}>k} = o_p(\beta_k\sqrt{\alpha_k}).
\end{equation*}
By the Lipschitz continuity of $\nabla f$ and $G$, we know for $k\geq k_0+1$ that
\begin{align*}
& (\|\nabla f_{k-1}-\nabla f_k\| + \|G_{k-1}-G_k\|)\1_{\tau_{k_0}>k} \leq (\kappa_{\nabla f} + \kappa_{\nabla c})\|\bx_{k-1}-\bx_k\|\1_{\tau_{k_0}>k-1} \\
& \leq (\kappa_{\nabla f} + \kappa_{\nabla c})\baralpha_{k-1}\|\tDelta\bx_{k-1}\|\1_{\tau_{k_0}>k-1} \stackrel{\mathclap{\eqref{equ:def:tau}}}{\leq} \frac{1}{\epsilon}(\kappa_{\nabla f} + \kappa_{\nabla c})(\zeta\alpha_{k-1}+\psi\alpha_{k-1}^p)\|\bnabla\mL_{k-1}\|\1_{\tau_{k_0}>k-1}\\
& \leq \frac{1}{\epsilon}(\kappa_{\nabla f} + \kappa_{\nabla c})(\zeta\alpha_{k-1}+\psi\alpha_{k-1}^p)(\|\bnabla\mL_{k-1}-\nabla\mL_{k-1}\| + \|\nabla\mL_{k-1}\|)\1_{\tau_{k_0}>k-1} \\
& \stackrel{\mathclap{\eqref{equ:def:tau}}}{\leq} \frac{1}{\epsilon}(\kappa_{\nabla f} + \kappa_{\nabla c})(\zeta\alpha_{k-1}+\psi\alpha_{k-1}^p)\rbr{\|\bnabla\mL_{k-1}-\nabla\mL_{k-1}\| + \frac{\|\bz_{k-1}\|}{\epsilon}}\1_{\tau_{k_0}>k-1} \\ 
& = O_p\rbr{\alpha_k\rbr{\sqrt{\beta_k}+b_k^2}},
\end{align*}
where the last equality is due to Lemma \ref{lem:local:rate}. We note that
\begin{equation*}
O_p\rbr{\alpha_k\rbr{\sqrt{\beta_k}+b_k^2}} = o_p(\beta_k\sqrt{\alpha_k})\Longleftarrow \min\{p_1+0.5p_2, p_1+2p_3\}>p_2+0.5p_1\Longleftarrow \eqref{cond:4.8}.
\end{equation*}
Thus, we obtain $\C_{4,1}^k = o_p(\sqrt{\alpha_k})$.
\vskip4pt
\noindent$\bullet$$\bullet$ For $\C_{4,2}^k$, we still apply Lemma \ref{lem:4}. We have
\begin{align*}
& \nbr{\sum_{h=0}^{k}\prod_{l=h+1}^{k}(1-\beta_l)\beta_h(W^\star)^{-1}\begin{pmatrix}
(\hnabla c_h - \mE[\hnabla c_h\mid \mF_{h-1}])^T(\blambda_k-\tlambda)\\
\0
\end{pmatrix}}\1_{\tau_{k_0}>k}\\
& \leq \nbr{\sum_{h=0}^{k}\prod_{l=h+1}^{k}(1-\beta_l)\beta_h(W^\star)^{-1}\begin{pmatrix}
\hnabla c_h - \mE[\hnabla c_h\mid \mF_{h-1}]\\
\0
\end{pmatrix}}^2\1_{\tau_{k_0}>k} + \|\blambda_k-\tlambda\|^2\1_{\tau_{k_0}>k}\\
& = O_p(\beta_k+b_k^4) \stackrel{\eqref{cond:4.8}}{=}o_p(\sqrt{\alpha_k}),
\end{align*}
where the second equality from the last is due to \eqref{nequ:5} and Lemma \ref{lem:local:rate}. Thus, Lemma \ref{lem:4} suggests~that $\C_{4,2}^k = o_p(\sqrt{\alpha_k})$.
\vskip4pt
\noindent$\bullet$$\bullet$ For $\C_{4,3}^k$, we have a similar derivation. In particular, we note that
\begin{align*}
&\nbr{\sum_{h=0}^{k}\prod_{l=h+1}^{k}(1-\beta_l)\beta_h(W^\star)^{-1}\begin{pmatrix}
(\mE[\hnabla c_h\mid \mF_{h-1}] - G_h)^T(\blambda_k-\tlambda)\\
\0
\end{pmatrix}}\1_{\tau_{k_0}>k} \\
& \leq \nbr{\sum_{h=0}^{k}\prod_{l=h+1}^{k}(1-\beta_l)\beta_h(W^\star)^{-1}\begin{pmatrix}
\mE[\hnabla c_h\mid \mF_{h-1}] - G_h\\
\0
\end{pmatrix}}\|\blambda_k-\tlambda\|\1_{\tau_{k_0}>k} \\
& =O_p(b_k^2(\sqrt{\beta_k}+b_k^2))\stackrel{\eqref{cond:4.8}}{=}o_p(\sqrt{\alpha_k}),
\end{align*}
where the second equality from the last is due to Lemmas \ref{lemma:finite difference for gradient} and \ref{lem:local:rate} and \eqref{nequ:9}. Thus, Lemma \ref{lem:4}~suggests that $\C_{4,3}^k = o_p(\sqrt{\alpha_k})$.
\vskip4pt
\noindent$\bullet$$\bullet$ For $\C_{4,4}^k$, we apply Lemma \ref{lemma:finite difference for gradient} and have
\begin{equation*}
\nbr{\sum_{h=0}^{k}\prod_{l=h+1}^{k}(1-\beta_l)\beta_h(W^\star)^{-1}\begin{pmatrix}
\mE[\hnabla_{\bx}\mL(\bx_h,\tlambda;\xi_h)\mid \mF_{h-1}] - \nabla_{\bx}\mL(\bx_h,\tlambda)\\
\0
\end{pmatrix}}\1_{\tau_{k_0}>k} = O(b_k^2) \stackrel{\eqref{cond:4.8}}{=} o(\sqrt{\alpha_k}).
\end{equation*}
Thus, Lemma \ref{lem:4}~suggests that $\C_{4,4}^k = o_p(\sqrt{\alpha_k})$.
\vskip4pt
\noindent$\bullet$$\bullet$ For $\C_{4,5}^k$, we aim to show
\begin{equation}\label{dequ:4}
1/\sqrt{\zeta\alpha_k}\cdot\C_{4,5}^k\stackrel{d}{\longrightarrow}\mN(\0,\omega\cdot (W^\star)^{-1}\tOmega(W^\star)^{-1}).
\end{equation}
We have
\begin{align}\label{nequ:25}
\C_{4,5}^k & = \sum_{i=0}^{k}\sum_{h=0}^{i}\prod_{j=i+1}^{k}(1-\zeta\alpha_j)\zeta\alpha_i\prod_{l=h+1}^{i}(1-\beta_l)\beta_h(W^\star)^{-1}\begin{pmatrix}
\hnabla_{\bx}\mL(\bx_h,\tlambda;\xi_h) - \mE[\hnabla_{\bx}\mL(\bx_h,\tlambda;\xi_h)\mid \mF_{h-1}]\\
\0
\end{pmatrix} \nonumber\\
& = \sum_{h=0}^{k}\sum_{i=h}^{k}\prod_{j=i+1}^{k}(1-\zeta\alpha_j)\zeta\alpha_i\prod_{l=h+1}^{i}(1-\beta_l)\beta_h(W^\star)^{-1}\begin{pmatrix}
\hnabla_{\bx}\mL(\bx_h,\tlambda;\xi_h) - \mE[\hnabla_{\bx}\mL(\bx_h,\tlambda;\xi_h)\mid \mF_{h-1}]\\
\0
\end{pmatrix} \nonumber\\
& \eqqcolon \sum_{h=0}^{k} a_{h,k}\cdot\bphi_{h}.
\end{align}
We claim that $\mE[\bphi_h\bphi_h^T\mid \mF_{h-1}]\rightarrow (W^\star)^{-1}\tOmega(W^\star)^{-1}$ as $h\rightarrow\infty$ almost surely. In fact, we have
\begin{align*}
& \mE\sbr{(\hnabla_{\bx}\mL(\bx_h,\tlambda;\xi_h) - \mE[\hnabla_{\bx}\mL(\bx_h,\tlambda;\xi_h)\mid \mF_{h-1}])(\hnabla_{\bx}\mL(\bx_h,\tlambda;\xi_h) - \mE[\hnabla_{\bx}\mL(\bx_h,\tlambda;\xi_h)\mid \mF_{h-1}])^T\mid \mF_{h-1}} \\
& = \mE\sbr{\hnabla_{\bx}\mL(\bx_h,\tlambda;\xi_h)\hnabla_{\bx}^T\mL(\bx_h,\tlambda;\xi_h)\mid \mF_{h-1}} - \mE[\hnabla_{\bx}\mL(\bx_h,\tlambda;\xi_h)\mid \mF_{h-1}]\mE[\hnabla_{\bx}\mL(\bx_h,\tlambda;\xi_h)\mid \mF_{h-1}]^T.
\end{align*}
Since $\mE[\hnabla_{\bx}\mL(\bx_h,\tlambda;\xi_h)\mid \mF_{h-1}]\rightarrow \nabla_{\bx}\mL^\star=\0$ as $h\rightarrow\infty$ by Lemma \ref{lemma:finite difference for gradient}, we only consider the first~term. We have
\begin{align}\label{nequ:24}
& \mE\sbr{\hnabla_{\bx}\mL(\bx_h,\tlambda;\xi_h)\hnabla_{\bx}^T\mL(\bx_h,\tlambda;\xi_h)\mid \mF_{h-1}} \nonumber\\
& = \mE\sbr{(\hnabla F(\bx_h;\xi_h)+\hnabla^Tc(\bx_h)\tlambda)(\hnabla F(\bx_h;\xi_h)+\hnabla^Tc(\bx_h)\tlambda)^T\mid \mF_{h-1}} \nonumber\\
& = \frac{1}{4b_h^2}\mE\sbr{\bDelta_h^{-1}
\cbr{\delta\rbr{F(\bx_h\pm b_h\bDelta_h;\xi_h)+c^T(\bx_h\pm b_h\bDelta_h)\tlambda} }^2 \bDelta_h^{-T}\mid \mF_{h-1}} \nonumber\\
& = \frac{1}{4b_h^2}\mE\sbr{\bDelta_h^{-1}\bDelta_h^T\int_{-b_h}^{b_h}\int_{-b_h}^{b_h}\nabla_{\bx}\mL(\bx_h+s_1\bDelta_h,\tlambda;\xi_h)\nabla_{\bx}^T\mL(\bx_h+s_2\bDelta_h,\tlambda;\xi_h)ds_1ds_2\bDelta_h\bDelta_h^{-T}\mid \mF_{h-1} },
\end{align}
where in the second equality, we follow the definition in \eqref{snequ:11} and define
\begin{multline*}
\delta\rbr{F(\bx_h\pm b_h\bDelta_h;\xi_h)+c^T(\bx_h\pm b_h\bDelta_h)\tlambda} \coloneqq \rbr{F(\bx_h+ b_h\bDelta_h;\xi_h)+c^T(\bx_h+ b_h\bDelta_h)\tlambda}\\
- \rbr{F(\bx_h- b_h\bDelta_h;\xi_h)+c^T(\bx_h- b_h\bDelta_h)\tlambda}.
\end{multline*}
For \eqref{nequ:24}, we first condition on both $\bx_h$ and $\bDelta_h$, and focus on the integrand. 
For each run of the~algorithm, we consider $h$ to be sufficiently large (with a potentially random threshold index) such that $\bx_h \in \{\bx : \|\bx - \tx\| \leq \delta'\}$, where $\delta' \in (0, \delta)$ is chosen to ensure that $\bx + s \bDelta \in \{\bx : \|\bx - \tx\| \leq \delta\}$~for~any $s \in [-b_h, b_h]$ and $\bDelta \sim \P_{\bDelta}$. For the above $\bx_h$ and any $-b_h\leq s_1, s_2\leq b_h$, we have
\begin{align}\label{dequ:1}
& \mE\sbr{\nabla_{\bx}\mL(\bx_h+s_1\bDelta_h,\tlambda;\xi_h)\nabla_{\bx}^T\mL(\bx_h+s_2\bDelta_h,\tlambda;\xi_h) \mid \bx_h,\bDelta_h} - \mE\sbr{\nabla_{\bx}\mL(\tx,\tlambda;\xi)\nabla_{\bx}^T\mL(\tx,\tlambda;\xi)} \nonumber\\
& = \mE\sbr{\nabla F(\bx_h+s_1\bDelta_h;\xi_h)\nabla^T F(\bx_h+s_2\bDelta_h;\xi_h) - \nabla F(\tx;\xi_h)\nabla^T F(\tx;\xi_h)\mid \bx_h,\bDelta_h} \nonumber\\
& \quad + \nabla f(\bx_h+s_1\bDelta_h)(\tlambda)^TG(\bx_h+s_2\bDelta_h) - \nabla f^\star (\tlambda)^TG^\star \nonumber\\
& \quad + G^T(\bx_h+s_1\bDelta_h)\tlambda\nabla^T f(\bx_h+s_2\bDelta_h) - (G^\star)^T\tlambda\nabla^T f^\star \nonumber\\
& \quad + G^T(\bx_h+s_1\bDelta_h)\tlambda(\tlambda)^TG(\bx_h+s_2\bDelta_h) - (G^\star)^T\tlambda(\tlambda)^TG^\star.
\end{align}
For the first term in \eqref{dequ:1}, we can further bound it as
\begin{align*}
& \nbr{\mE\sbr{\nabla F(\bx_h+s_1\bDelta_h;\xi_h)\nabla^T F(\bx_h+s_2\bDelta_h;\xi_h) - \nabla F(\tx;\xi_h)\nabla^T F(\tx;\xi_h)\mid \bx_h,\bDelta_h}}\\
& \leq \mE\sbr{\|\nabla F(\bx_h+s_1\bDelta_h;\xi_h) - \nabla F(\tx;\xi_h)\|\cdot\|\nabla F(\bx_h+s_2\bDelta_h;\xi_h) - \nabla F(\tx;\xi_h)\| \mid \bx_h,\bDelta_h}\\
& \quad + \mE\sbr{\|\nabla F(\bx_h+s_1\bDelta_h;\xi_h) - \nabla F(\tx;\xi_h)\|\cdot\|\nabla F(\tx;\xi_h)\| \mid \bx_h,\bDelta_h}\\
& \quad + \mE\sbr{\|\nabla F(\bx_h+s_2\bDelta_h;\xi_h) - \nabla F(\tx;\xi_h)\|\cdot\|\nabla F(\tx;\xi_h)\| \mid \bx_h,\bDelta_h}\\
& \leq \prod_{q=1}^{2}\cbr{\mE\sbr{\|\nabla F(\bx_h+s_q\bDelta_h;\xi_h) - \nabla F(\tx;\xi_h)\|^2 \mid \bx_h,\bDelta_h}}^{1/2} \\
&\quad + \cbr{\mE[\|\nabla F(\tx;\xi_h)\|^2]}^{1/2}\cdot\sum_{q=1}^{2}\cbr{\mE\sbr{\|\nabla F(\bx_h+s_q\bDelta_h;\xi_h) - \nabla F(\tx;\xi_h)\|^2 \mid \bx_h,\bDelta_h}}^{1/2}. 
\end{align*}
Note from Assumptions \ref{ass:4.2} and \ref{ass:Delta} that for $q=1,2$,
\begin{align*}
& \mE\sbr{\nbr{\nabla F(\bx_h+s_q\bDelta_h;\xi_h)-\nabla F(\tx;\xi_h)}^2\mid \bx_h, \bDelta_h} \\
& = \mE\sbr{\nbr{\int_{0}^1 \nabla^2 F(\bx_h+s_q\bDelta_h + t(\bx_h+s_q\bDelta_h-\tx);\xi_h)(\bx_h+s_q\bDelta_h-\tx) dt }^2\mid \bx_h, \bDelta_h}\\
& \leq \int_{0}^{1}\mE[\|\nabla^2 F(\bx_h+s_q\bDelta_h + t(\bx_h+s_q\bDelta_h-\tx);\xi_h)\|^2\mid \bx_h, \bDelta_h] dt\cdot \|\bx_h+s_q\bDelta_h-\tx\|^2\\
& = O(\|\bx_h-\tx\|^2 + b_h^2) \rightarrow 0\quad \text{ as }\;\; h\rightarrow \infty.
\end{align*}
The above two displays imply almost surely,
\begin{multline*}
\max_{-b_h\leq s_1,s_2\leq b_h}\big\{\mE\sbr{\nabla F(\bx_h+s_1\bDelta_h;\xi_h)\nabla^T F(\bx_h+s_2\bDelta_h;\xi_h)\mid \bx_h,\bDelta_h} \\ - \mE[\nabla F(\tx;\xi)\nabla^T F(\tx;\xi)]\big\}\rightarrow 0\quad \text{ as }\;\; h\rightarrow \infty.
\end{multline*}
For the second, third, and fourth terms in \eqref{dequ:1}, it is trivial to verify that they achieve the same almost sure convergence as the above display, due to the Lipschitz continuity of $\nabla f$ and $G$ and the fact that $|s_1|, |s_2|\leq b_h\rightarrow \infty$. Therefore, we combine \eqref{nequ:24} and \eqref{dequ:1} and obtain almost surely,
\begin{align}\label{dequ:5}
& \mE\sbr{\hnabla_{\bx}\mL(\bx_h,\tlambda;\xi_h)\hnabla_{\bx}^T\mL(\bx_h,\tlambda;\xi_h)\mid \mF_{h-1}} \nonumber\\
&\hskip3cm \longrightarrow \mE\sbr{\bDelta^{-1}\bDelta^T\mE\sbr{\nabla_{\bx}\mL(\tx,\tlambda;\xi)\nabla_{\bx}^T\mL(\tx,\tlambda;\xi)}\bDelta\bDelta^{-T}} \nonumber\\
& \hskip3cm \quad = \mE\sbr{\bDelta^{-1}\bDelta^T\text{Cov}\rbr{\nabla_{\bx}\mL(\tx,\tlambda;\xi)}\bDelta\bDelta^{-T}} \nonumber\\
& \hskip3cm\quad = \mE\sbr{\bDelta^{-1}\bDelta^T\text{Cov}\rbr{\nabla F(\tx;\xi)}\bDelta\bDelta^{-T}}.
\end{align}
This, together with \eqref{nequ:24} and the definition of $\bphi_h$ in \eqref{nequ:25}, implies $\mE[\bphi_h \bphi_h^T \mid \mF_{h-1}] \rightarrow (W^\star)^{-1} \tOmega (W^\star)^{-1}$ as $h \rightarrow \infty$ almost surely. With this result, we then analyze the conditional variance process. We have
\begin{align*}
& \frac{1}{\zeta\alpha_k}\sum_{h=0}^{k} a_{h,k}^2\mE[\bphi_h\bphi_h^T\mid \mF_{h-1}]\\
& = \frac{1}{\zeta\alpha_k} \sum_{h=0}^{k}\sum_{i=h}^k\sum_{i'=h}^{k}\prod_{j=i+1}^{k}(1-\zeta\alpha_j)\zeta\alpha_i\prod_{l=h+1}^{i}(1-\beta_l)\beta_h\prod_{j'=i'+1}^{k}(1-\zeta\alpha_{j'})\zeta\alpha_{i'}\prod_{l'=h+1}^{i'}(1-\beta_{l'})\beta_h\mE[\bphi_h\bphi_h^T\mid \mF_{h-1}]\\
& = \frac{1}{\zeta\alpha_k}\sum_{i=0}^{k}\sum_{i'=0}^{k}\prod_{j=i+1}^{k}(1-\zeta\alpha_j)\zeta\alpha_i\prod_{j'=i'+1}^{k}(1-\zeta\alpha_{j'})\zeta\alpha_{i'}\sum_{h=0}^{\min\{i,i'\}}\prod_{l=h+1}^{i}(1-\beta_l)\prod_{l'=h+1}^{i'}(1-\beta_{l'})\beta_h^2\mE[\bphi_h\bphi_h^T\mid \mF_{h-1}]\\
& = \frac{2}{\zeta\alpha_k}\sum_{i=0}^{k}\sum_{i'=0}^{i}\prod_{j=i+1}^{k}(1-\zeta\alpha_j)\zeta\alpha_i\prod_{j'=i'+1}^{k}(1-\zeta\alpha_{j'})\zeta\alpha_{i'}\sum_{h=0}^{i'}\prod_{l=h+1}^{i}(1-\beta_l)\prod_{l'=h+1}^{i'}(1-\beta_{l'})\beta_h^2\mE[\bphi_h\bphi_h^T\mid \mF_{h-1}]\\
& \quad - \frac{1}{\zeta\alpha_k}\sum_{i=0}^{k}\prod_{j=i+1}^{k}(1-\zeta\alpha_j)^2\zeta^2\alpha_i^2\sum_{h=0}^{i}\prod_{l=h+1}^{i}(1-\beta_l)^2\beta_h^2\mE[\bphi_h\bphi_h^T\mid \mF_{h-1}]\\
& = \frac{2}{\zeta\alpha_k}\sum_{i=0}^{k}\prod_{j=i+1}^{k}(1-\zeta\alpha_j)^2\zeta\alpha_i\sum_{i'=0}^{i}\prod_{j'=i'+1}^{i}(1-\zeta\alpha_{j'})(1-\beta_{j'})\zeta\alpha_{i'}\sum_{h=0}^{i'}\prod_{l'=h+1}^{i'}(1-\beta_{l'})^2\beta_h^2\mE[\bphi_h\bphi_h^T\mid \mF_{h-1}]\\
& \quad - \frac{1}{\zeta\alpha_k}\sum_{i=0}^{k}\prod_{j=i+1}^{k}(1-\zeta\alpha_j)^2\zeta^2\alpha_i^2\sum_{h=0}^{i}\prod_{l=h+1}^{i}(1-\beta_l)^2\beta_h^2\mE[\bphi_h\bphi_h^T\mid \mF_{h-1}].
\end{align*}
We apply Lemma \ref{technical lemma:2} and note that
\begin{align*}
& \lim\limits_{i\rightarrow\infty}\frac{1}{\beta_i}\sum_{h=0}^{i}\prod_{l=h+1}^{i}(1-\beta_l)^2\beta_h^2\mE[\bphi_h\bphi_h^T\mid \mF_{h-1}] = \frac{1}{2}(W^\star)^{-1}\tOmega(W^\star)^{-1},\\
& \lim\limits_{i\rightarrow \infty}\frac{1}{\zeta\alpha_i}\sum_{i'=0}^{i}\prod_{j'=i'+1}^{i}(1-\zeta\alpha_{j'})(1-\beta_{j'})\zeta\alpha_{i'}\beta_{i'} = 1,\\
& \lim\limits_{k\rightarrow\infty}\frac{1}{\zeta\alpha_k} \sum_{i=0}^{k}\prod_{j=i+1}^{k}(1-\zeta\alpha_j)^2\zeta^2\alpha_i^2 = \omega \coloneqq \begin{cases}
0.5, & \text{if }\; p_1\in(0,1),\\
\frac{\zeta\iota_{1}}{2\zeta\iota_{1}-1}, & \text{if }\; p_1=1,
\end{cases}\\
& \lim\limits_{k\rightarrow\infty}\frac{1}{\zeta\alpha_k} \sum_{i=0}^{k}\prod_{j=i+1}^{k}(1-\zeta\alpha_j)^2\zeta^2\alpha_i^2\beta_i = 0.
\end{align*}
Combining the above two displays, we obtain almost surely,
\begin{equation}\label{dequ:2}
\lim\limits_{k\rightarrow\infty}\frac{1}{\zeta\alpha_k}\sum_{h=0}^{k} a_{h,k}^2\mE[\bphi_h\bphi_h^T\mid \mF_{h-1}] = \omega\cdot(W^\star)^{-1}\tOmega(W^\star)^{-1}.
\end{equation}
Next, we verify the Lindeberg condition. We aim to show that for any $\epsilon>0$,
\begin{equation}\label{dequ:3}
\lim\limits_{k\rightarrow\infty}\frac{1}{\alpha_k}\sum_{h=0}^{k}a_{h,k}^2\mE\sbr{\|\bphi_h\|^2 1_{\|a_{h,k}\bphi_h\|\geq \epsilon\sqrt{\alpha_k}} \mid \mF_{h-1}} \leq \lim\limits_{k\rightarrow\infty}\frac{1}{\epsilon\alpha_k^{1.5}}\sum_{h=0}^{k}a_{h,k}^3\mE[\|\bphi_h\|^3\mid \mF_{h-1}] = 0.
\end{equation}
Since $r\geq 3$ in \eqref{cond:4.8}, we know from \eqref{aequ:B7} that $\mE[\|\bphi_h\|^3\mid \mF_{h-1}]$ is uniformly bounded. Thus, it suffices~to show $\sum_{h=0}^{k}a_{h,k}^3=o(\alpha_k^{1.5})$. We have
\begin{align*}
\sum_{h=0}^{k}a_{h,k}^3 & = \sum_{h=0}^k\sum_{i=h}^{k}\sum_{i'=h}^{k}\sum_{i''=h}^{k}\prod_{j=i+1}^{k}(1-\zeta\alpha_j)\zeta\alpha_i\prod_{l=h+1}^{i}(1-\beta_l)\beta_h\prod_{j'=i'+1}^{k}(1-\zeta\alpha_{j'})\zeta\alpha_{i'}\prod_{l'=h+1}^{i'}(1-\beta_{l'})\beta_h\cdot\\
&\quad\quad \prod_{j''=i''+1}^{k}(1-\zeta\alpha_{j''})\zeta\alpha_{i''}\prod_{l''=h+1}^{i''}(1-\beta_{l''})\beta_h\\
& = \sum_{i=0}^{k}\sum_{i'=0}^{k}\sum_{i''=0}^{k}\prod_{j=i+1}^{k}(1-\zeta\alpha_j)\zeta\alpha_i\prod_{j'=i'+1}^{k}(1-\zeta\alpha_{j'})\zeta\alpha_{i'}\prod_{j''=i''+1}^{k}(1-\zeta\alpha_{j''})\zeta\alpha_{i''}\cdot\\
& \quad\quad \sum_{h=0}^{\min\{i,i',i''\}}\prod_{l=h+1}^{i}(1-\beta_l)\prod_{l'=h+1}^{i'}(1-\beta_{l'})\prod_{l''=h+1}^{i''}(1-\beta_{l''}) \beta_h^3\\
& \leq 6\sum_{i=0}^{k}\sum_{i'=0}^{i}\sum_{i''=0}^{i'}\prod_{j=i+1}^{k}(1-\zeta\alpha_j)\zeta\alpha_i\prod_{j'=i'+1}^{k}(1-\zeta\alpha_{j'})\zeta\alpha_{i'}\prod_{j''=i''+1}^{k}(1-\zeta\alpha_{j''})\zeta\alpha_{i''}\cdot\\
& \quad\quad\quad \sum_{h=0}^{i''}\prod_{l=h+1}^{i}(1-\beta_l)\prod_{l'=h+1}^{i'}(1-\beta_{l'})\prod_{l''=h+1}^{i''}(1-\beta_{l''}) \beta_h^3 \quad\quad (i\geq i'\geq i'')\\
& = 6\sum_{i=0}^{k}\prod_{j=i+1}^{k}(1-\zeta\alpha_j)^3\zeta\alpha_i\sum_{i'=0}^{i}\prod_{j'=i'+1}^{i}(1-\zeta\alpha_{j'})^2(1-\beta_{j'})\zeta\alpha_{i'}\cdot\\
&\quad\quad\quad \sum_{i''=0}^{i'}\prod_{j''=i''+1}^{i'}(1-\zeta\alpha_{j''})(1-\beta_{j''})^2\zeta\alpha_{i''}\sum_{h=0}^{i''}\prod_{l''=h+1}^{i''}(1-\beta_{l''})^3 \beta_h^3.
\end{align*}
We apply Lemma \ref{technical lemma:2} and \cite[Lemma B.3(b)]{Na2025Statistical} and note that
\begin{align*}
&\lim\limits_{i''\rightarrow\infty} \frac{1}{\beta_{i''}^2}\sum_{h=0}^{i''}\prod_{l''=h+1}^{i''}(1-\beta_{l''})^3 \beta_h^3 = \frac{1}{3},\\
& \lim\limits_{i'\rightarrow\infty} \frac{1}{\zeta\alpha_{i'}\beta_{i'}}\sum_{i''=0}^{i'}\prod_{j''=i''+1}^{i'}(1-\zeta\alpha_{j''})(1-\beta_{j''})^2\zeta\alpha_{i''}\beta_{i''}^2 = \frac{1}{2},\\
& \lim\limits_{i\rightarrow\infty} \frac{1}{(\zeta\alpha_i)^2}\sum_{i'=0}^{i}\prod_{j'=i'+1}^{i}(1-\zeta\alpha_{j'})^2(1-\beta_{j'})(\zeta\alpha_{i'})^2\beta_{i'} = 1,\\
& \lim\limits_{k\rightarrow\infty}\frac{1}{(\zeta\alpha_k)^{1.5}}\sum_{i=0}^{k}\prod_{j=i+1}^{k}(1-\zeta\alpha_j)^3(\zeta\alpha_i)^3 = 0,
\end{align*}
where the last equality applies $\zeta\iota_{1}>0.5$ when $p_1=1$. Thus, we have $\sum_{h=0}^{k}a_{h,k}^3 = o(\alpha_k^{1.5})$. By the central limit theorem of martingale arrays \cite[Corollary 3.1]{Hall2014Martingale}, the results \eqref{dequ:2}~and \eqref{dequ:3} lead to \eqref{dequ:4}.

Finally, we combine the result of $\C_{4,5}^k$ in \eqref{dequ:4} with all the results of $\C_1^k, \C_2^k, \C_3^k, \C_{4,1}^k, \C_{4,2}^k, \C_{4,3}^k, \C_{4,4}^k$, for which we have shown that each is of order $o_p(\sqrt{\alpha_k})$. We obtain
\begin{equation*}
1/\sqrt{\zeta\alpha_k}\cdot(\bx_k-\tx, \blambda_k-\tlambda)\stackrel{d}{\longrightarrow}\mN(\0,\omega\cdot (W^\star)^{-1}\tOmega(W^\star)^{-1}).
\end{equation*}
Noting that $\baralpha_k/(\zeta\alpha_k)\rightarrow 1$ almost surely and applying Slutsky's theorem, we complete the proof.

\subsection{Proof of Lemma \ref{lem:4}}\label{pf:lem:4}

We aim to show that for any $\epsilon, \delta>0$, there exists $K = K(\epsilon, \delta)>0$ such that for any $k\geq K(\epsilon, \delta)$,
\begin{equation}\label{cequ:22}
P\rbr{\frac{1}{\sqrt{\alpha_k}}\abr{\sum_{i=0}^{k}\prod_{j=i+1}^{k}(1-\zeta\alpha_j)\alpha_iX_i}\geq \epsilon}\leq \delta.
\end{equation}
For the above fixed $\epsilon, \delta>0$, we know from $P(\cup_{k_0=0}^{\infty}\{\tau_{k_0} =\infty\}) = 1$ that
\begin{align*}
P\rbr{\bigcap_{k_0=0}^{\infty}\mB_{k_0}}\coloneqq P\rbr{\bigcap_{k_0=0}^{\infty}\bigcup_{k_0'\geq k_0}\bigcup_{k\geq k_0'}\cbr{\frac{1}{\sqrt{\alpha_k}}\sum_{i=k_0'}^{k}\prod_{j=i+1}^{k}|(1-\zeta\alpha_j)\alpha_iX_i|\1_{\tau_{k_0}\leq i} \geq \frac{\epsilon}{3} } } = 0.
\end{align*}
Since
\begin{align*}
\mB_{k_0+1} & = \bigcup_{k_0'\geq k_0+1}\bigcup_{k\geq k_0'}\cbr{\frac{1}{\sqrt{\alpha_k}}\sum_{i=k_0'}^{k}\prod_{j=i+1}^{k}|(1-\zeta\alpha_j)\alpha_iX_i|\1_{\tau_{k_0+1}\leq i} \geq \frac{\epsilon}{3} } \\
& \subseteq \bigcup_{k_0'\geq k_0+1}\bigcup_{k\geq k_0'}\cbr{\frac{1}{\sqrt{\alpha_k}}\sum_{i=k_0'}^{k}\prod_{j=i+1}^{k}|(1-\zeta\alpha_j)\alpha_iX_i|\1_{\tau_{k_0}\leq i} \geq \frac{\epsilon}{3} } \quad(\text{ since } \tau_{k_0+1}\geq \tau_{k_0})\\
& \subseteq \bigcup_{k_0'\geq k_0}\bigcup_{k\geq k_0'}\cbr{\frac{1}{\sqrt{\alpha_k}}\sum_{i=k_0'}^{k}\prod_{j=i+1}^{k}|(1-\zeta\alpha_j)\alpha_iX_i|\1_{\tau_{k_0}\leq i} \geq \frac{\epsilon}{3} } = \mB_{k_0},
\end{align*}
the above two displays imply that $\lim_{k_0\rightarrow\infty}P(\mB_{k_0}) = 0$. Thus, there exists $k_0(\delta)\geq \bar{k}_0$ such that~for~any $k\geq k_0(\delta)$,
\begin{align}\label{nequ:13}
& P\rbr{\frac{1}{\sqrt{\alpha_k}}\abr{\sum_{i=k_0(\delta)}^{k}\prod_{j=i+1}^{k}(1-\zeta\alpha_j)\alpha_iX_i\1_{\tau_{k_0(\delta)}\leq i}}\geq \frac{\epsilon}{3} } \nonumber\\
& \leq P\rbr{\frac{1}{\sqrt{\alpha_k}}\sum_{i=k_0(\delta)}^{k}\prod_{j=i+1}^{k}|(1-\zeta\alpha_j)\alpha_iX_i|\1_{\tau_{k_0(\delta)}\leq i}\geq \frac{\epsilon}{3}} \nonumber\\
& \leq P\rbr{\bigcup_{k\geq k_0(\delta)}\cbr{\frac{1}{\sqrt{\alpha_k}}\sum_{i=k_0(\delta)}^{k}\prod_{j=i+1}^{k}|(1-\zeta\alpha_j)\alpha_iX_i|\1_{\tau_{k_0(\delta)}\leq i} \geq \frac{\epsilon}{3} }} \leq P(\mB_{k_0(\delta)}) \leq \frac{\delta}{3}. 
\end{align}
For the above $k_0(\delta)$ fixed, we apply Lemma \ref{technical lemma:2} and have
\begin{equation*}
\frac{1}{\sqrt{\alpha_k}} \sum_{i=k_0(\delta)}^{k}\prod_{j=i+1}^{k}(1-\zeta\alpha_j)\alpha_iX_i\1_{\tau_{k_0(\delta)}>i} = o_p\rbr{\frac{1}{\sqrt{\alpha_k}}\sum_{i=k_0(\delta)}^{k}\prod_{j=i+1}^{k}(1-\zeta\alpha_j)\alpha_i^{1.5} } = o_p(1).
\end{equation*}
Thus, there exists $K^1 = K^1(\epsilon, \delta)\geq k_0(\delta)$ such that for any $k\geq K^1(\epsilon, \delta)$,
\begin{equation}\label{nequ:14}
P\rbr{\frac{1}{\sqrt{\alpha_k}}\abr{\sum_{i=k_0(\delta)}^{k}\prod_{j=i+1}^{k}(1-\zeta\alpha_j)\alpha_iX_i\1_{\tau_{k_0(\delta)}>i}} \geq \frac{\epsilon}{3} }\leq \frac{\delta}{3}.
\end{equation}
Finally, we note that with probability 1,
\begin{align*}
\frac{1}{\sqrt{\alpha_k}}\abr{\sum_{i=0}^{k_0(\delta)-1}\prod_{j=i+1}^{k}(1-\zeta\alpha_j)\alpha_iX_i} & \leq \frac{1}{\sqrt{\alpha_k}}\sum_{i=0}^{k_0(\delta)-1}\prod_{j=i+1}^{k_0(\delta)-1}\abr{(1-\zeta\alpha_j)\alpha_iX_i}\cdot\prod_{j=k_0(\delta)}^{k}\abr{1-\zeta\alpha_j} \\ & \stackrel{\eqref{nequ:7}}{\longrightarrow} 0 \quad\;\; \text{ as }\;\; k\rightarrow\infty.
\end{align*}
This implies that
\begin{equation*}
P\rbr{\bigcap_{k\geq k_0(\delta)}\mC_k} \coloneqq P\rbr{\bigcap_{k\geq k_0(\delta)}\bigcup_{k'\geq k}\cbr{\frac{1}{\sqrt{\alpha_k}}\abr{\sum_{i=0}^{k_0(\delta)-1}\prod_{j=i+1}^{k'}(1-\zeta\alpha_j)\alpha_iX_i}\geq \frac{\epsilon}{3}}} = 0.
\end{equation*}
Since $\mC_{k+1}\subseteq \mC_k$, we have $\lim_{k\rightarrow\infty}P(\mC_k) = 0$. Thus, there exists $K^2(\epsilon, \delta)\geq k_0(\delta)$ such that for any $k\geq K^2(\epsilon, \delta)$,
\begin{equation}\label{nequ:15}
P\rbr{\frac{1}{\sqrt{\alpha_k}}\abr{\sum_{i=0}^{k_0(\delta)-1}\prod_{j=i+1}^{k}(1-\zeta\alpha_j)\alpha_iX_i}\geq \frac{\epsilon}{3}} \leq P(\mC_k) \leq \frac{\delta}{3}.
\end{equation}
Combining \eqref{nequ:13}, \eqref{nequ:14}, \eqref{nequ:15}, and letting $K(\epsilon, \delta) \coloneqq \max\{K^1(\epsilon, \delta), K^2(\epsilon, \delta)\}$, we have \mbox{$\forall k\geq K(\epsilon, \delta)$},
\begin{align*}
& P\rbr{\frac{1}{\sqrt{\alpha_k}}\abr{\sum_{i=0}^{k}\prod_{j=i+1}^{k}(1-\zeta\alpha_j)\alpha_iX_i}\geq \epsilon}\\
& \leq P\rbr{\frac{1}{\sqrt{\alpha_k}}\abr{\sum_{i=0}^{k_0(\delta)-1}\prod_{j=i+1}^{k}(1-\zeta\alpha_j)\alpha_iX_i}\geq \frac{\epsilon}{3}} + P\rbr{\frac{1}{\sqrt{\alpha_k}}\abr{\sum_{i=k_0(\delta)}^{k}\prod_{j=i+1}^{k}(1-\zeta\alpha_j)\alpha_iX_i\1_{\tau_{k_0(\delta)}>i}} \geq \frac{\epsilon}{3} }\\
& \quad + P\rbr{\frac{1}{\sqrt{\alpha_k}}\abr{\sum_{i=k_0(\delta)}^{k}\prod_{j=i+1}^{k}(1-\zeta\alpha_j)\alpha_iX_i\1_{\tau_{k_0(\delta)}\leq i}} \geq \frac{\epsilon}{3} } \leq \frac{\delta}{3}+\frac{\delta}{3}+\frac{\delta}{3}= \delta.
\end{align*}
This verifies \eqref{cequ:22} and completes the proof.

\subsection{Proof of Proposition \ref{prop:1}}\label{pf:prop:1}

By the definition of $\bSigma^\star$, $\bSigma^\star_{op}$ and $\tOmega$, we note that
\begin{align*}
\bSigma^\star-\bSigma^\star_{op} & = (W^\star)^{-1}\rbr{\tOmega - \diag\rbr{\text{Cov}(\nabla F(\tx;\xi)), \0}} (W^\star)^{-1}\\
& = (W^\star)^{-1}\diag\rbr{\mE\sbr{\bDelta^{-1}\bDelta^{T}\text{Cov}(\nabla F(\tx;\xi))\bDelta\bDelta^{-T}}-\text{Cov}(\nabla F(\tx;\xi)),\; \0 }(W^\star)^{-1}\\
& = (W^\star)^{-1}\diag\rbr{\mE\sbr{\rbr{\bDelta^{-1}\bDelta^{T} - I} \text{Cov}(\nabla F(\tx;\xi))\rbr{\bDelta\bDelta^{-T} - I}},\; \0 }(W^\star)^{-1}\succeq \0,
\end{align*}
where the third equality is due to $\mE[\bDelta^{-1}\bDelta^{T}] = \mE[\bDelta\bDelta^{-T}] = I$ by Assumption \ref{ass:Delta}. For the second~part of the result, we follow the above result and have
\begin{align*}
\|\bSigma^\star-\bSigma^\star_{op}\| & \geq \frac{1}{\|W^\star\|^2}\nbr{\diag\rbr{\mE\sbr{\rbr{\bDelta^{-1}\bDelta^{T} - I} \text{Cov}(\nabla F(\tx;\xi))\rbr{\bDelta\bDelta^{-T} - I}},\; \0 }}\\
& = \frac{1}{\|W^\star\|^2} \nbr{\mE\sbr{\rbr{\bDelta^{-1}\bDelta^{T} - I} \text{Cov}(\nabla F(\tx;\xi))\rbr{\bDelta\bDelta^{-T} - I}}}\\
& \geq \frac{\lambda_{\min}(\text{Cov}(\nabla F(\tx;\xi)))}{\|W^\star\|^2}\nbr{\mE\sbr{\rbr{\bDelta^{-1}\bDelta^{T} - I}\rbr{\bDelta\bDelta^{-T} - I}}}\\
& = \frac{\lambda_{\min}(\text{Cov}(\nabla F(\tx;\xi)))}{\|W^\star\|^2}\nbr{\mE[\bDelta^T\bDelta\cdot \bDelta^{-1}\bDelta^{-T}] -I}\\
& = \frac{\lambda_{\min}(\text{Cov}(\nabla F(\tx;\xi)))}{\|W^\star\|^2}\cdot(d-1)\mE[\bDelta^2]\mE[\frac{1}{\bDelta^2}]\quad (\text{by Assumption \ref{ass:Delta}}).
\end{align*}
On the other hand, we also have
\begin{align*}
\|\bSigma^\star-\bSigma^\star_{op}\| & \leq \|(W^\star)^{-1}\|^2 \nbr{\diag\rbr{\mE\sbr{\rbr{\bDelta^{-1}\bDelta^{T} - I} \text{Cov}(\nabla F(\tx;\xi))\rbr{\bDelta\bDelta^{-T} - I}},\; \0 }}\\
&=  \|(W^\star)^{-1}\|^2\nbr{\mE\sbr{\rbr{\bDelta^{-1}\bDelta^{T} - I} \text{Cov}(\nabla F(\tx;\xi))\rbr{\bDelta\bDelta^{-T} - I}}}\\
& \leq \|(W^\star)^{-1}\|^2\lambda_{\max}(\text{Cov}(\nabla F(\tx;\xi))) \nbr{\mE\sbr{\rbr{\bDelta^{-1}\bDelta^{T} - I}\rbr{\bDelta\bDelta^{-T} - I}}}\\
& = \|(W^\star)^{-1}\|^2\lambda_{\max}(\text{Cov}(\nabla F(\tx;\xi))) \nbr{\mE[\bDelta^T\bDelta\cdot \bDelta^{-1}\bDelta^{-T}] -I}\\
& = \|(W^\star)^{-1}\|^2\lambda_{\max}(\text{Cov}(\nabla F(\tx;\xi)))\cdot(d-1)\mE[\bDelta^2]\mE[\frac{1}{\bDelta^2}].
\end{align*}
This completes the proof.

\subsection{Proof of Proposition \ref{prop:2}}\label{pf:prop:2}

By Lemmas \ref{lemma:average almost sure} and \ref{lem:Hessian Convergence}, we know $\tW_k\rightarrow W^\star$ as $k\rightarrow\infty$ almost surely. Thus, it suffices to show 
\begin{multline*}
\frac{1}{k+1}\sum_{t=0}^{k}\rbr{\hnabla F(\bx_t;\xi_t)+\hnabla^T c(\bx_t)\blambda_t}\rbr{\hnabla F(\bx_t;\xi_t)+\hnabla^T c(\bx_t)\blambda_t}^T\\ \longrightarrow \mE\sbr{\bDelta^{-1}\bDelta^T\text{Cov}\rbr{\nabla F(\tx;\xi)}\bDelta\bDelta^{-T}}\quad \text{ as }\;\; k\rightarrow \infty\;\; \text{ almost surely}.
\end{multline*}
Recall from the proof of $\C_4^k$ in Appendix \ref{pf:thm:normality} that we define $\hnabla_{\bx}\mL(\bx_t,\blambda_t;\xi_t) \coloneqq \hnabla F(\bx_t;\xi_t)+\hnabla^T c(\bx_t)\blambda_t$.
Since $r\geq 4$, we apply \eqref{aequ:B7} and the strong law of large number for square integrable martingales~\cite[Theorem 1.3.15]{Duflo1997Random}, and know that
\begin{equation}\label{eequ:1}
\frac{1}{k+1}\sum_{t=0}^{k}\rbr{\hnabla_{\bx}\mL(\bx_t,\blambda_t;\xi_t)\hnabla_{\bx}^T\mL(\bx_t,\blambda_t;\xi_t) - \mE\sbr{\hnabla_{\bx}\mL(\bx_t,\blambda_t;\xi_t)\hnabla_{\bx}^T\mL(\bx_t,\blambda_t;\xi_t) \mid \mF_{t-1}}}\rightarrow 0
\end{equation}
as $k\rightarrow\infty$ almost surely. Furthermore, we have
\begin{align*}
& \mE\sbr{\hnabla_{\bx}\mL(\bx_t,\blambda_t;\xi_t)\hnabla_{\bx}^T\mL(\bx_t,\blambda_t;\xi_t) \mid \mF_{t-1}} - \mE\sbr{\hnabla_{\bx}\mL(\bx_t,\tlambda;\xi_t)\hnabla_{\bx}^T\mL(\bx_t,\tlambda;\xi_t) \mid \mF_{t-1}}\\
& = \mE[\hnabla^T c(\bx_t)(\blambda_t-\tlambda)\hnabla^T f(\bx_t)\mid \mF_{t-1}] + \mE[\hnabla f(\bx_t)(\blambda_t-\tlambda)^T\hnabla c(\bx_t)\mid \mF_{t-1}]\\
&\quad + \mE[\hnabla^T c(\bx_t)(\blambda_t-\tlambda)(\blambda_t-\tlambda)^T\hnabla c(\bx_t)\mid \mF_{t-1}]\\
& = O(\|\blambda_t-\tlambda\| + \|\blambda_t-\tlambda\|^2) \rightarrow 0 \quad \text{ as }\;\; t\rightarrow \infty\;\; \text{ almost surely},
\end{align*}
where the second equality is due to the boundedness of $\hnabla f(\bx_t)$ and $\hnabla c(\bx_t)$, which is as shown in \eqref{aequ:3}. Therefore, the Stolz–Cesa\`ro theorem suggests that
\begin{equation*}
\frac{1}{k+1}\sum_{t=0}^{k}\rbr{\mE\sbr{\hnabla_{\bx}\mL(\bx_t,\blambda_t;\xi_t)\hnabla_{\bx}^T\mL(\bx_t,\blambda_t;\xi_t) \mid \mF_{t-1}} - \mE\sbr{\hnabla_{\bx}\mL(\bx_t,\tlambda;\xi_t)\hnabla_{\bx}^T\mL(\bx_t,\tlambda;\xi_t) \mid \mF_{t-1}}} \rightarrow 0
\end{equation*}
as $k\rightarrow \infty$ almost surely. Finally, applying \eqref{dequ:5} and the Stolz–Cesa\`ro theorem again, we obtain
\begin{equation*}
\frac{1}{k+1}\sum_{t=0}^{k}\mE\sbr{\hnabla_{\bx}\mL(\bx_t,\tlambda;\xi_t)\hnabla_{\bx}^T\mL(\bx_t,\tlambda;\xi_t) \mid \mF_{t-1}} \rightarrow \mE\sbr{\bDelta^{-1}\bDelta^T\text{Cov}\rbr{\nabla F(\tx;\xi)}\bDelta\bDelta^{-T}}
\end{equation*}
as $k\rightarrow \infty$ almost surely. Combining the above two displays with \eqref{eequ:1}, we complete the proof.

\end{document}

%% file: table.tex
\begin{table}[t]
\centering
\small
\setlength{\tabcolsep}{4pt}
\renewcommand{\arraystretch}{0.95}
\begin{adjustbox}{max width=\textwidth, max totalheight=0.9\textheight, keepaspectratio}
\begin{tabular}{ccc|cccc|cccc}
\hline
\multirow{2}{*}{\textbf{Prob}} & \multirow{2}{*}{$\sigma^2$} & \multirow{2}{*}{\textbf{Hess}} & \multicolumn{4}{c|}{\textbf{Derivative-Free SSQP}} & \multicolumn{4}{c}{\textbf{Derivative-Based SSQP}}\\
\cline{4-11}
& & & 
\textbf{Err $(10^{-4})$} & \textbf{Cov $(100\%)$} & \textbf{Len $(10^{-2})$} & \textbf{FLOPs} & \textbf{Err $(10^{-4})$} & \textbf{Cov $(100\%)$} & \textbf{Len $(10^{-2})$} & \textbf{FLOPs}\\
\hline
\multirow{8}{*}{MARATOS} & \multirow{2}{*}{$10^{-4}$} 
& Id & 6.54 & 93.50 & 0.15 & 31.80 & 1.13 & 94.00 & 0.03 & 33.00 \\ & 
& Hess & 6.45 & 92.50 & 0.15 & 42.20 & 1.10 & 94.50 & 0.03 & 45.00 \\
\cline{4-11}
& \multirow{2}{*}{$10^{-2}$} 
& Id & 64.99 & 93.00 & 1.46 & 31.80 & 11.95 & 95.00 & 0.26 & 33.00 \\ &
& Hess & 62.06 & 93.00 & 1.47 & 42.20 & 10.12 & 97.50 & 0.26 & 45.00 \\
\cline{4-11}
& \multirow{2}{*}{$10^{-1}$} 
& Id & 217.15 & 91.50 & 4.61 & 31.80 & 36.24 & 95.00 & 0.82 & 33.00 \\ &
& Hess & 192.85 & 92.00 & 4.64 & 42.20 & 32.98 & 97.00 & 0.82 & 45.00 \\
\cline{4-11}
& \multirow{2}{*}{$1$} 
& Id & 633.13 & 95.00 & 14.55 & 31.80 & 105.72 & 94.50 & 2.61 & 33.00 \\ &
& Hess & 610.65 & 95.00 & 14.77 & 42.20 & 109.65 & 93.00 & 2.61 & 45.00 \\
\hline
\multirow{8}{*}{HS48} & \multirow{2}{*}{$10^{-4}$} 
& Id & 7.75 & 99.40 & 0.11 & 371.01 & 0.98 & 99.70 & 0.01 & 378.01 \\ & 
& Hess & 5.01 & {\red 94.70} & 0.05 & 454.01 & 0.64 & {\red 95.10} & 0.01 & 453.01 \\
\cline{4-11}
& \multirow{2}{*}{$10^{-2}$} 
& Id & 82.23 & 99.30 & 1.13 & 371.01 & 8.68 & 99.80 & 0.14 & 378.01 \\ &
& Hess & 51.78 & {\red 94.30} & 0.48 & 454.01 & 6.58 & {\red 94.10} & 0.06 & 453.01 \\
\cline{4-11}
& \multirow{2}{*}{$10^{-1}$} 
& Id & 253.12 & 99.20 & 3.56 & 371.01 & 29.26 & 99.60 & 0.45 & 378.01 \\ &
& Hess & 180.68 & 91.00 & 1.48 & 454.01 & 19.24 & {\red 95.50} & 0.18 & 453.01 \\
\cline{4-11}
& \multirow{2}{*}{$1$} 
& Id & 811.96 & 99.30 & 11.26 & 371.01 & 91.78 & 99.30 & 1.42 & 378.01 \\ & 
& Hess & 577.03 & {\red 93.90} & 4.70 & 454.01 & 63.69 & {\red 96.40} & 0.57 & 453.01 \\
\hline
\multirow{8}{*}{BT9} & \multirow{2}{*}{$10^{-4}$} 
& Id & 7.40 & 98.25 & 0.15 & 235.20 & 1.18 & 99.25 & 0.03 & 240.00 \\ &
& Hess & 5.12 & {\red 95.50} & 0.08 & 289.60 & 0.80 & {\red 96.75} & 0.01 & 288.01 \\
\cline{4-11}
& \multirow{2}{*}{$10^{-2}$} 
& Id & 66.83 & 100.00 & 1.46 & 235.20 & 11.39 & 99.25 & 0.26 & 240.00 \\ &
& Hess & 7933.70 & {\red 94.30} & 0.89 & 289.60 & 83.67 & {\red 95.57} & 0.13 & 288.01 \\
\cline{4-11}
& \multirow{2}{*}{$10^{-1}$} 
& Id & 236.10 & 98.50 & 4.59 & 235.20 & 36.02 & 99.25 & 0.82 & 240.00 \\ &
& Hess & $\slash$ & 84.81 & 8.76 & 289.60 & $\slash$ & 88.58 & 9.57 & 288.01 \\
\cline{4-11}
& \multirow{2}{*}{$1$} 
& Id & 769.04 & {\blue 95.50} & 14.07 & 235.20 & 124.44 & 99.25 & 2.60 & 240.00 \\ &
& Hess & $\slash$ & 57.69 & 57.84 & 289.60 & $\slash$ & 58.04 & 16.11 & 288.01 \\
\hline
\multirow{8}{*}{BYRDSPHR} & \multirow{2}{*}{$10^{-4}$}
& Id & 8.94 & 83.50 & 0.10 & 137.00 & 1.11 & 83.50 & 0.01 & 140.00 \\ &
& Hess & 14.39 & 88.50 & 0.22 & 168.80 & 1.82 & 92.00 & 0.03 & 167.00 \\
\cline{4-11}
& \multirow{2}{*}{$10^{-2}$}
& Id & 93.26 & 80.50 & 1.03 & 137.00 & 10.06 & 84.50 & 0.13 & 140.00 \\ &
& Hess & 126.14 & {\red 96.25} & 2.17 & 168.80 & 16.76 & {\red 93.50} & 0.27 & 167.00 \\
\cline{4-11}
& \multirow{2}{*}{$10^{-1}$}
& Id & 274.76 & 81.00 & 3.26 & 137.00 & 34.61 & 79.00 & 0.41 & 140.00 \\ &
& Hess & 419.41 & {\red 92.75} & 6.85 & 168.80 & 49.84 & {\red 94.75} & 0.86 & 167.00 \\
\cline{4-11}
& \multirow{2}{*}{$1$}
& Id & 960.05 & 79.00 & 10.31 & 137.00 & 113.15 & 84.25 & 1.30 & 140.00 \\ &
& Hess & 1478.60 & {\red 92.00} & 22.16 & 168.80 & $\slash$ & {\red 95.75} & 8.40 & 167.00 \\
\hline
\multirow{8}{*}{BT1} & \multirow{2}{*}{$10^{-4}$} 
& Id   & 6.54 & 93.50 & 0.15 & 31.80 & 1.13 & 99.00 & 0.04 & 33.00 \\ &
& Hess & 6.45 & 92.50 & 0.15 & 42.20 & 1.10 & 99.50 & 0.04 & 45.00 \\
\cline{4-11}
& \multirow{2}{*}{$10^{-2}$} 
& Id   & 64.99 & 93.00 & 1.46 & 31.80 & 11.95 & 99.50 & 0.40 & 33.00 \\ &
& Hess & 62.06 & 93.00 & 1.47 & 42.20 & 10.12 & 100.00 & 0.40 & 45.00 \\
\cline{4-11}
& \multirow{2}{*}{$10^{-1}$} 
& Id   & 217.15 & 91.50 & 4.61 & 31.80 & 36.24 & 100.00 & 1.26 & 33.00 \\ &
& Hess & 192.85 & 92.00 & 4.64 & 42.20 & 32.98 & 100.00 & 1.27 & 45.00 \\
\cline{4-11}
& \multirow{2}{*}{$1$} 
& Id   & 633.13 & 95.00 & 14.55 & 31.80 & 105.72 & 100.00 & 4.10 & 33.00 \\ &
& Hess & 610.65 & 95.00 & 14.77 & 42.20 & $\slash$ & 100.00 & $\slash$ & 45.00 \\
\hline
\multirow{8}{*}{HS51} & \multirow{2}{*}{$10^{-4}$} 
& Id   & 5.97 & 99.30 & 0.08 & 544.01 & 0.78 & 99.60 & 0.01 & 552.01 \\ &
& Hess & 4.12 & {\red 94.40} & 0.04 & 651.01 & 0.50 & {\red 96.00} & 0.00 & 627.01 \\
\cline{4-11}
& \multirow{2}{*}{$10^{-2}$} 
& Id   & 62.54 & 99.40 & 0.85 & 544.01 & 7.00 & 100.00 & 0.11 & 552.01 \\ &
& Hess & 43.26 & {\red 92.70} & 0.36 & 651.01 & 5.06 & {\red 94.70} & 0.04 & 627.01 \\
\cline{4-11}
& \multirow{2}{*}{$10^{-1}$} 
& Id   & 203.20 & 99.40 & 2.69 & 544.01 & 25.05 & 99.70 & 0.35 & 552.01 \\ &
& Hess & 137.75 & {\red 92.20} & 1.12 & 651.01 & 15.61 & {\red 93.80} & 0.14 & 627.01 \\
\cline{4-11}
& \multirow{2}{*}{$1$} 
& Id   & 691.39 & 99.60 & 8.51 & 544.01 & 74.99 & 99.60 & 1.09 & 552.01 \\ &
& Hess & 435.44 & {\red 93.60} & 3.54 & 651.01 & 45.96 & {\red 96.90} & 0.44 & 627.01 \\
\hline
\multirow{8}{*}{BT12} & \multirow{2}{*}{$10^{-4}$}
& Id   & 10.55 & 87.30 & 0.08 & 544.01 & 1.70   & 88.40 & 0.01 & 552.01 \\ &
& Hess & 11.24 & {\red 93.00} & 0.11 & 651.01 & 1.85   & {\red 95.90} & 0.02 & 627.01 \\
\cline{4-11}
& \multirow{2}{*}{$10^{-2}$}
& Id   & 120.78 & 85.00 & 0.82 & 544.01 & 15.44  & 93.60 & 0.13 & 552.01 \\ &
& Hess & 125.22 & 90.81 & 1.13 & 651.01 & 500.20 & 95.05 & 0.17 & 627.01 \\
\cline{4-11}
& \multirow{2}{*}{$10^{-1}$}
& Id   & 329.67 & 89.30 & 2.59 & 544.01 & 54.91  & 88.70 & 0.41 & 552.01 \\ &
& Hess & $\slash$ & 90.13 & 3.62 & 651.01 & $\slash$ & 92.26 & 0.54 & 627.01 \\
\cline{4-11}
& \multirow{2}{*}{$1$}
& Id   & 1021.90 & 89.80 & 8.19 & 544.01 & 157.90   & {\blue 92.20} & 1.30 & 552.01 \\ &
& Hess & $\slash$ & 87.00 & 12.05 & 651.01 & $\slash$ & 87.12 & 1.69 & 627.01 \\
\hline
\multirow{8}{*}{HS42} & \multirow{2}{*}{$10^{-4}$}
& Id & 5.71 & 99.50 & 0.12 & 235.20 & 0.79 & 99.83 & 0.02 & 240.00 \\ &
& Hess & 3.52 & {\red 94.00} & 0.04 & 289.60 & 0.51 & {\red 92.67} & 0.01 & 288.01 \\
\cline{4-11}
& \multirow{2}{*}{$10^{-2}$}
& Id   & 53.13 & 100.00 & 1.17 & 235.20 & 8.64 & 99.83 & 0.17 & 240.00 \\ &
& Hess & 34.27 & {\red 94.17}  & 0.36 & 289.60 & 5.62 & {\red 92.67} & 0.06 & 288.01 \\
\cline{4-11}
& \multirow{2}{*}{$10^{-1}$}
& Id   & 181.17 & 99.67 & 3.69 & 235.20 & 27.12 & 99.67 & 0.55 & 240.00 \\ &
& Hess & 112.10 & {\red 92.33} & 1.14 & 289.60 & 18.17 & {\red 93.83} & 0.18 & 288.00 \\
\cline{4-11}
& \multirow{2}{*}{$1$}
& Id   & 530.18 & 99.67 & 11.68 & 235.20 & 88.91 & 100.00 & 1.75 & 240.00 \\ &
& Hess & 349.85 & 90.67 & 3.57  & 289.60 & 53.69 & {\red 96.00}  & 0.56 & 288.00 \\
\hline
\end{tabular}
\end{adjustbox}					
\caption{Comparison of DF-SSQP and DB-SSQP on 8 CUTEst problems under four noise variances~$\sigma^2$. ``$\slash$" indicates 
cases where the iterate error exceeds 1 (the methods may converge to a stationary point different from the one given by the package). Red numbers indicate cases where second-order methods achieve coverage closer to the nominal 95\% than first-order methods; blue numbers indicate the converse. Unhighlighted entries are cases where either both first- and second-order methods are near-nominal or both exhibit under- or over-coverage.}
\label{tab:1}
\end{table}